\documentclass[12pt]{article}

\usepackage{epsfig,epsfig}
\usepackage{epstopdf}
\usepackage{amsmath}
\usepackage{mathrsfs}
\usepackage{amsfonts}
\usepackage{bm}
\usepackage{amssymb}
\usepackage{amsthm}
\usepackage{booktabs}
\usepackage{color}
\usepackage{multirow}
\usepackage{graphicx}
\usepackage{subfigure}
\usepackage{float}
\usepackage{appendix}
\usepackage{tikz}
\usetikzlibrary{calc}
\usetikzlibrary{matrix}
\usepackage{mathtools}
\usepackage{algorithm,algpseudocode,float}
\usepackage{lipsum}
\makeatletter
\newenvironment{breakablealgorithm}
  {
   \begin{center}
     \refstepcounter{algorithm}
     \hrule height.8pt depth0pt \kern2pt
     \renewcommand{\caption}[2][\relax]{
       {\raggedright\textbf{\ALG@name~\thealgorithm} ##2\par}%
       \ifx\relax##1\relax 
         \addcontentsline{loa}{algorithm}{\protect\numberline{\thealgorithm}##2}%
       \else 
         \addcontentsline{loa}{algorithm}{\protect\numberline{\thealgorithm}##1}%
       \fi
       \kern2pt\hrule\kern2pt
     }
  }{
     \kern2pt\hrule\relax
   \end{center}
  }
\makeatother

\usepackage{lineno}
\usepackage{srcltx,graphicx}

\newtheorem{thm}{Theorem}[section]
\newtheorem{prop}[thm]{Proposition}

\newtheorem{rem}[thm]{Remark}
\newtheorem{example}[thm]{Example}

\newcommand{\beq}{\begin{equation}}
\newcommand{\eeq}{\end{equation}}

\usepackage{hyperref}
\usepackage{authblk}
\usepackage{footmisc}
\usepackage{lipsum}

\numberwithin{equation}{section} \topmargin=-2.0cm \oddsidemargin=1cm
\begin{document}

\baselineskip=1.5pc

\title{\textbf{A well-balanced moving mesh discontinuous Galerkin method for the Ripa model on triangular meshes}}

\author{Weizhang Huang\footnote{Department of Mathematics, University of Kansas, Lawrence, Kansas 66045, USA. E-mail: whuang@ku.edu.},
~Ruo Li\footnote{CAPT, LMAM, and School of Mathematical Sciences, Peking University, Beijing 100871, China.
E-mail: rli@math.pku.edu.cn.
The research of this author is partly supported by the National Key R\&D Program of China, Project Number 2020YFA0712000.},
~Jianxian Qiu\footnote{School of Mathematical Sciences and Fujian Provincial Key Laboratory of Mathematical Modeling and High-Performance Scientific Computing, Xiamen University, Xiamen, Fujian 361005, China.
E-mail: jxqiu@xmu.edu.cn.
The research of this author is partly supported by NSFC grant 12071392.},
~and Min Zhang\footnote{School of Mathematical Sciences, Peking University, Beijing, 100871, China.
E-mail: minzhang@math.pku.edu.cn.}
}

\date{}
\maketitle
\begin{abstract}
A well-balanced moving mesh discontinuous Galerkin (DG) method is proposed for the numerical solution of the Ripa model -- a generalization of the shallow water equations that accounts for effects of water temperature variations.
Thermodynamic processes are important particularly in the upper layers of the ocean where the variations of sea surface temperature play a fundamental role in climate change.
The well-balance property which requires numerical schemes to preserve the lake-at-rest steady state is crucial to the simulation of perturbation waves over that steady state such as waves on a lake or tsunami waves in the deep ocean.
To ensure the well-balance, positivity-preserving, and high-order properties, a DG-interpolation scheme
(with or without scaling positivity-preserving limiter) and special treatments pertaining to the Ripa model are employed in the transfer
of both the flow variables and bottom topography from the old mesh to the new one and in the TVB limiting process.
Mesh adaptivity is realized using an MMPDE moving mesh approach and a metric tensor based on an equilibrium variable and water depth. A motivation is to adapt the mesh according to both the perturbations of the lake-at-rest steady state and the water depth distribution (bottom structure).
Numerical examples in one and two dimensions are presented to demonstrate the well-balance, high-order accuracy, and positivity-preserving properties of the method and its ability to capture small perturbations of the lake-at-rest steady state.
\end{abstract}
\noindent\textbf{The 2020 Mathematics Subject Classification:} 65M50, 65M60, 76B15, 35Q35

\vspace{5pt}


\normalsize \vskip 0.2in

\newpage

\section{Introduction}

We are interested in the numerical solution of the Ripa model that is a generalization of the shallow water equations (SWEs) where a temperature gradient is considered. The SWEs play a critical role in the modeling and simulation of free-surface flows in rivers and coastal areas. They can be used to predict tides, storm surge levels, coastline changes from hurricanes, and etc. The SWEs can be derived by integrating the Navier-Stokes equations in depth under the hydrostatic assumption when the depth of water is small compared to its horizontal dimensions, with the density being assumed constant.
However, the SWEs have the limitation that they cannot represent thermodynamic processes which are important particularly in the upper layers of the ocean where the variations of sea surface temperature are an important factor to climate change.

Many oceanic phenomena can be investigated using layered models where
the flow variables such as density and horizontal velocity are considered vertically uniform in each of the layers. Research on the numerical solution of those layered SWEs \cite{Abgrall-Karni-2009SISC,Bouchut-Zeitlin-2010,Kurganov-Petrova-2009SISC} has attracted tremendous attention in recent years.
Challenges in studying such models arise from the complicated eigenstructure, conditional hyperbolicity, non-conservativeness of the system,
and etc.
The Ripa model was introduced to analyze the ocean currents by Ripa \cite{Ripa-1993} in 1993. It was derived by integrating the velocity field, density, and horizontal gradients along the vertical direction in each layer of multi-layer SWEs. The introduction of temperature is advantageous because the movement and behavior of ocean currents are impacted by forces such as temperature acting upon the water. The Ripa model in conservative form reads as
\begin{equation}\label{ripa-form}
\frac{\partial}{\partial t}
\begin{bmatrix*}[l]
  h\\
  hu\\
  hv\\
  h\theta\\
\end{bmatrix*}
+\frac{\partial}{\partial x}
\begin{bmatrix*}[c]
hu\\
hu^2+\frac{1}{2}g h^2\theta\\
huv\\
hu\theta\\
\end{bmatrix*}
+\frac{\partial}{\partial y}
\begin{bmatrix*}[c]
hv\\
huv\\
hv^2+\frac{1}{2}g h^2\theta\\
hv\theta\\
\end{bmatrix*}
=
\begin{bmatrix*}[c]
  0\\
  -gh\theta b_x\\
  - gh\theta b_y\\
  0\\
\end{bmatrix*} ,
\end{equation}
where $h(x,y,t)\geq0$ is the depth of water, $(u,v)$ are the depth-averaged horizontal velocities,  $\theta(x,y,t)>0$ is a potential temperature field, $b = b(x,y)$ is the bottom topography assumed to be a given time-independent function, and $g$ is the gravitational acceleration constant. The horizontally varying potential temperature field $\theta$ represents the reduced gravity $g\Delta\Theta/\Theta_{ref}$, where $\Delta\Theta$ is the difference in potential temperature from a reference value $\Theta_{ref}$ while
the term, $\frac12 g h^2\theta$, represents the pressure depending on the water temperature.

The Ripa model possesses a number of steady-state solutions \cite{Desveaux-etal-2016}.
Of particular interest are the steady states corresponding to still water (also called ``lake-at-rest"),
 \begin{equation}\label{case1-0}
u=0,\quad v=0,\quad \theta = C_1, \quad h+b = C_2,
  \end{equation}
where $C_1$ and $C_2$ are positive constants.
Many physical phenomena, such as waves on a lake or tsunami waves in the deep ocean,
can be described as small perturbations of these steady-state solutions, and
they are difficult to capture numerically unless numerical schemes preserve the unperturbed steady state at the discrete level.
Thus, it is crucial to develop such steady-state preserving numerical schemes, which are called well-balanced schemes in literature.

Developing well-balanced numerical schemes for the Ripa model is not a trivial task.
Recall that the Ripa model \eqref{ripa-form} reduces to the SWEs when the temperature is constant
and both systems have a similar structure.
In principle, we can extend well-balanced schemes developed for the SWEs to the Ripa model.
However, caution must be taken to make sure that constant temperature is preserved for the still water steady state since the Ripa model has an extra variable (the temperature) and an extra differential equation.
This is especially true when moving meshes are used in the computation.
Moreover, the Ripa model can be rewritten as
\begin{equation}\label{ripa-form-2}
\frac{\partial U}{\partial t}+\nabla \cdot F(U) = S(U,b),
\end{equation}
where
\[
\begin{split}
U = \begin{bmatrix*}[l]
  h\\
  m\\
  w\\
  \eta\\
\end{bmatrix*}
= \begin{bmatrix*}[l]
  h\\
  hu\\
  hv\\
  h\theta\\
\end{bmatrix*}
,\quad
&
F(U)=\begin{bmatrix*}[c]
 m&w\\
  \frac{m^2}{h}+\frac{1}{2}g\eta h& \frac{mw}{h}\\
 \frac{mw}{h}&\frac{w^2}{h}+\frac{1}{2}g\eta h\\
 \frac{m\eta}{h}&\frac{w\eta}{h}\\
\end{bmatrix*}
,\quad S(U,b)=
\begin{bmatrix*}[c]
  0\\
  -g\eta b_x\\
  - g\eta b_y\\
  0\\
\end{bmatrix*}.
\end{split}
\]
In terms of the conservative variables $U = (h, m, w, \eta)^T$,
the lake-at-rest steady-state solution (\ref{case1-0}) reads as
\begin{equation}
\label{case1}
m = 0, \quad w = 0, \quad \eta = C_1 h, \quad h + b = C_2.
\end{equation}
From this, we can see that the main difference between the lake-at-rest steady-state solutions
for the SWEs and the Ripa model is that the Ripa model
has an extra relation $\eta = C_1 h$, which is a linear relation between $\eta$ and $h$ instead of being a constant for a conservative variable.
{\em This relation requires special attention pertaining to the Ripa model
when verifying the well-balance property of numerical schemes.}
Indeed, as can be seen in \S\ref{sec:WB-MMDG}, special treatments are needed in the interpolation/remapping step,
the TVB limiting process,
and positivity-preserving (PP) limiting process for the proposed numerical method to be well-balanced.

In recent years studies have been made on the development of well-balanced numerical schemes for the Ripa model.
The first work seems to be \cite{Chertock-etal-2014} by Chertock et al. who considered central-upwind schemes.
Other examples include
Touma and Klingenberg \cite{Touma-Klingenberg-2015} (a second-order positivity preserving finite volume scheme on rectangular meshes),
S$\acute{a}$nchez-Linares et al. \cite{Sanchez-etal-2016HLLC} (a second-order positivity preserving HLLC scheme based in path-conservative approximate Riemann solvers, for the one-dimensional Ripa model),
Han and Li  \cite{Han-Li-2016} (a high-order finite difference weighted essentially non-oscillatory (WENO) scheme),
Saleem et al. \cite{Saleem-etal-2018KFVX} (a kinetic flux vector splitting scheme on rectangular meshes),
Thanh et al. \cite{Thanh-etal-2021} (a high-order scheme of van Leer's type for the one-dimensional SWEs with temperature gradient),
Rehman et al.  \cite{Rehman-etal-2021} (a fifth-order finite volume multi-resolution WENO scheme on rectangular meshes),
Britton and Xing \cite{Britton-Xing-2020JSC} (a DG scheme for the one-dimensional Ripa model),
Qian et al. \cite{Qian-etal-2018} (a DG method based on a source term approximation technique),
and Li et al. \cite{Li-etal-2020AAMM} (a DG method based on hydrostatic reconstruction on rectangular meshes).
Fixed meshes are employed in the above mentioned works.

The objective of this work is to develop a well-balanced positivity-preserving moving mesh DG (MM-DG) method for the Ripa model on triangular meshes.
The Ripa model is a nonlinear hyperbolic system and can develop discontinuous solutions
such as shock waves, rarefaction waves, and contact discontinuities even when the initial condition is continuous.
Small mesh spacings are required in the regions of these structures in order to resolve them numerically
and mesh adaptation is often necessary to improve computational accuracy and efficiency.
In this work we shall study a rezoning-type moving mesh DG method
in one and two spatial dimensions.
The DG method \cite{DG-review} is known to be suited for the numerical solution of
hyperbolic systems while adaptive mesh movement provides the needed mesh adaptivity.
Our main objective is to show that the MM-DG method maintains high-order accuracy
of DG discretization while preserving the lake-at-rest steady state solutions and the positivity of the water depth and temperature.

The rezoning MM-DG method contains three basic components at each time step,
the adaptive mesh movement,
interpolation/remapping of the solution between the old and new meshes,
and numerical solution of the Ripa model on the new mesh.
The first component is based on the MMPDE moving mesh method \cite{Huang-Kamenski-2015JCP,Huang-etal-1994Siam,Huang-Russell-2011,Huang-Sun-2003JCP}
which has been shown analytically and numerically in \cite{Huang-Kamenski-2018MC} to generate meshes free of tangling for any (convex or concave) domain  in any dimension.
With the MMPDE method, the size, shape, and orientation of the mesh elements are controlled through
the metric tensor, a symmetric and uniformly positive definite matrix-valued function defined on the physical domain
and computed typically using the recovered Hessian of a DG solution.
The equilibrium variable $\mathcal{E}=\frac{1}{2}(u^2+v^2)+g\theta(h+b) $ and the water depth $h$ were used
in \cite{Zhang-Huang-Qiu-2021JSC} to construct the metric tensor so that the resulting mesh adapts
to the perturbations of the lake-at-rest steady-state solution.
Instead of using $\mathcal{E}$ and $h$ directly, we use  $\ln(\mathcal{E})$ and $\ln(h)$ here
to minimize the effect of dimensional difference between $\mathcal{E}$ and $h$.

The interpolation/remapping is a key component for the MM-DG method to maintain high-order accuracy,
preserve the lake-at-rest steady-state solutions, preserve the positivity of water depth and temperature field, and conserve the mass.
Several conservative interpolation schemes between deforming meshes have been investigated;
see, e.g. \cite{Li-Tang-2006JSC,TangTang-2003Siam}.
We use the recently developed DG-interpolation scheme
\cite{Zhang-Huang-Qiu-2020SISC} for this purpose.
This scheme works for large mesh deformation, has high-order accuracy,
conserves the mass, and preserves solution positivity as well as constant solutions.
It has been applied successfully to the moving mesh solution of
the radiative transfer equation \cite{Zhang-Huang-Qiu-2020SISC} and
the SWEs \cite{Zhang-Huang-Qiu-2021JSC}.

A well-balanced DG method is used for the numerical solution of the Ripa model on the new mesh; the detail of the method is presented in \S\ref{sec:WB-MMDG}.
It should be pointed out that a TVB slope limiter \cite{DG-series5} is used in our computation to avoid spurious oscillations. However, the standard TVB limiter procedure applied directly to the local characteristic variables based on $\big(h_h,m_h, w_h,\eta_h\big)^T$ may destroy the well-balance property. Special treatments are needed in their use to warrant the well-balance and high-order accuracy properties of the overall MM-DG method and in the situation with dry or almost dry regions in the domain.
The detail of the discussion is given in \S\ref{sec:WB-MMDG}.

The paper is organized as follows.
\S\ref{sec:WB-MMDG} is devoted to the description of the overall procedure
of the well-balanced MM-DG method and its DG and Runge-Kutta discretization for the Ripa model.
The DG-interpolation scheme and its properties are described in \S\ref{sec:DG-interp}.
Numerical results obtained with the MM-DG method for a selection of one- and two-dimensional
examples are presented in \S\ref{sec:numerical-results}.
Finally, \S\ref{sec:conclusions} contains conclusions and further comments.

\section{Well-balanced MM-DG method for Ripa model}
\label{sec:WB-MMDG}

In this section we describe the well-balanced rezoning-type MM-DG method for the numerical solution of the Ripa model on moving meshes. It contains three basic components at each time step, the adaptive mesh movement (generation of the new mesh), interpolation/remapping of the solution from the old mesh to the new one,
and numerical solution of the Ripa model on the new mesh; see Fig.~\ref{Fig:RMM}.
An MMPDE moving mesh method is used to generate the new mesh (see its description in \S\ref{sec:numerical-results}).
A DG-interpolation scheme is used for the solution interpolation between deforming meshes (see \S\ref{sec:DG-interp}).
In this section we focus on the overall description of the scheme and the discretization of the Ripa model in two dimensions on triangular meshes.

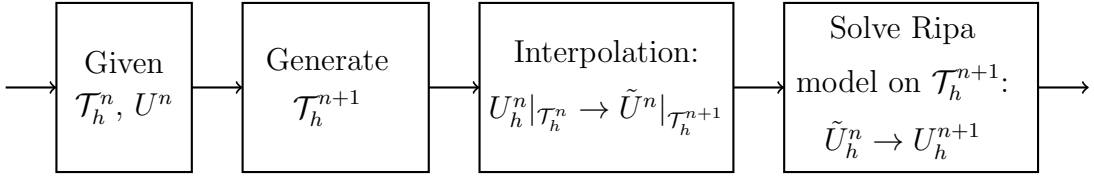
\begin{figure}[h]
\centering
\tikzset{my node/.code=\ifpgfmatrix\else\tikzset{matrix of nodes}\fi}
\begin{tikzpicture}[every node/.style={my node},scale=0.45]
\draw[thick] (0,0) rectangle (4,5);
\draw[thick] (5.5,0) rectangle (11,5);
\draw[thick] (12.5,0) rectangle (20,5);
\draw[thick] (21.5,0) rectangle (29,5);
\draw[->,thick] (-1.5 ,2.5)--(0 ,2.5);
\draw[->,thick] (4 ,2.5)--(5.5 ,2.5);
\draw[->,thick] (11,2.5)--(12.5,2.5);
\draw[->,thick] (20,2.5)--(21.5,2.5);
\draw[->,thick] (29,2.5)--(30.5,2.5);
\node (node1) at (2,2.5)
{Given \\ $\mathcal{T}_h^n$, $U^n$ \\};
\node (node2) at (8,2.5)
{Generate\\ $\mathcal{T}_h^{n+1}$\\ };
\node (node3) at (16.25,2.5)
{Interpolation:\\ $U_h^n|_{\mathcal{T}_h^n}\rightarrow\tilde{U}^n|_{\mathcal{T}_h^{n+1}}$\\};
\node (node4) at (25,2.5) {
Solve Ripa \\model on $\mathcal{T}_h^{n+1}$:\\
$\tilde{U}_h^n\rightarrow U_h^{n+1}$\\};
\end{tikzpicture}
\caption{Illustration of the rezoning-type moving mesh scheme.}\label{Fig:RMM}
\end{figure}

We use a well-balanced Runge-Kutta DG method for solving the Ripa model on triangular meshes $\mathcal{T}_h^{n+1}$ to obtain $U^{n+1}_h$ at physical time $t_{n+1}$.
To this end, we assume that the meshes at $t_n$ and $t_{n+1}$,  $\mathcal{T}_h^{n}$ and $\mathcal{T}_h^{n+1}$, respectively,
and the numerical solutions $U^n_h = (h^n_h,m^n_h,w^n_h,\eta^n_h)^T$ on $\mathcal{T}_h^{n}$ are known.
We also assume that the projection of $U^n_h = (h^n_h,m^n_h,w^n_h,\eta^n_h)^T$ on the new mesh $\mathcal{T}_h^{n+1}$,
denoted by $\tilde{U}^n_h = (\tilde{h}^n_h,\tilde{m}^n_h,\tilde{w}^n_h,\tilde{\eta}^n_h)^T$, has been obtained.
Define the DG finite element space on new mesh $\mathcal{T}_h^{n+1}$ as
\begin{equation}
\mathcal{V}^{k, n+1}_h= \{q\in L^2(\Omega):\; q|_{K}\in \mathbb{P}^{k}(K),
\; \forall K\in \mathcal{T}_h^{n+1}\},
\end{equation}
where $k$ is a positive integer, $\mathbb{P}^k(K)=\text{span}\{\phi_K^{(j)}\}_{j=1}^{n_b}$ is the set of polynomials of degree up to $k$ in the element $K$, and $\phi_K^{(j)}$, $j=1, ..., n_b$ are the local orthogonal basis functions.

Multiplying \eqref{ripa-form-2} by a test function $\phi\in \mathcal{V}^{k, n+1}_h$,
integrating the resulting equation over $K\in \mathcal{T}_h^{n+1}$, and using the divergence theorem, we have
\begin{equation}\label{div}
\frac{d}{d t}\int_{K}U_h \phi d\bm{x}
- \int_{K}F(U_h)\cdot \nabla \phi d\bm{x}
+ \int_{\partial K} F(U_h) \cdot\bm{n}_K \phi  ds
= \int_{K}S(U_h,b_h)\phi d\bm{x},
\end{equation}
where $\bm{n}_K=(n_x,n_y)^T$ is the outward unit normal to the boundary $\partial K$ and $U_h \in \mathcal{V}^{k, n+1}_h$
is a DG approximation to the exact solution $U$, and $b_h$ is a DG polynomial approximation of the bottom topography $b$.
Notice that $U_h$ is discontinuous on interior edges of the mesh.
On any interior edge $e_K \in \partial K$, $U_h$ can be defined using its value either in $K$, denoted by $U_{h,K}^{int}|_{e_K}$,
or in the element $K'$ sharing $e_K$ with $K$, denoted by $U_{h,K}^{ext}|_{e_K}$.
Moreover, we define the Lax-Friedrichs numerical flux approximating the flux function
$F(U) \cdot \bm{n}_K $ on the edge $e_K\in \partial K$ as
\begin{equation}\label{lf-flux}
\begin{split}
\hat{F}|_{e_K}
&=\hat{F}(U_{h,K}^{int}|_{e_K}, U_{h,K}^{ext}|_{e_K})
\\&=
\frac{1}{2}\Big{(}
\big{(}F(U_{h,K}^{int}|_{e_K})+F(U_{h,K}^{ext}|_{e_K}) \big{)} \cdot \bm{n}_K
-\alpha_h \big(U_{h,K}^{ext}|_{e_K}-U_{h,K}^{int}|_{e_K}\big)\Big{)},
\end{split}
\end{equation}
where $\alpha_h$ is the maximum absolute value of the eigenvalues of the Jacobian matrix
of $F(U_{h,K}^{int}|_{e_K})\cdot \bm{n}_K$ and $F(U_{h,K}^{ext}|_{e_K})\cdot \bm{n}_K$ taken over edge $e_K$ or over all elements.
The Jacobian of the matrix-valued flux function $F(U)\cdot \bm{n}_K$ is
\begin{equation*}\label{JF}
 F'(U)\cdot \bm{n}_K=
\begin{bmatrix*}[c]
    0 & n_x & n_y &0\\
    (\frac12gh\theta-u^2)n_x-uvn_y & 2un_x+vn_y & un_y&\frac12ghn_x\\
    (\frac12gh\theta-v^2)n_y-uvn_x & vn_x & un_x+2vn_y&\frac12ghn_y \\
    -\theta(un_x+vn_y)&\theta n_x&\theta n_y&(un_x+vn_y)\\
\end{bmatrix*},
\end{equation*}
and its eigenvalues are given by
\begin{align*}
&\lambda^{1}=un_x +vn_y -c,\quad
\lambda^{2,3}=un_x +vn_y ,
\quad\lambda^{4}=un_x +vn_y +c,
\end{align*}
where $c = \sqrt{gh\theta}$.

The semi-discrete DG scheme for the Ripa model is to find $U_h(t) \in \mathcal{V}^{k, n+1}_h$, $t \in (t_n, t_{n+1}]$, such that, for any $ \phi \in \mathcal{V}^{k, n+1}_h$
\begin{equation}\label{semi-DG-s1}
\frac{d}{d t}\int_{K}U_h \phi d\bm{x}
- \int_{K}F(U_h)\cdot \nabla \phi d\bm{x}
+ \sum_{e_K\in \partial K}\int_{e_K} \hat{F}|_{e_K} \phi  ds
= \int_{K}S(U_h,b_h)\phi d\bm{x}.
\end{equation}
Generally speaking, the above scheme does not preserve the lake-at-rest steady-state solution and thus is not well-balanced.

Following the idea of hydrostatic reconstruction \cite{Audusse-etal-2004SISC,Britton-Xing-2020JSC,Li-etal-2020AAMM}, after computing the boundary values $U_{h,K}^{int}|_{e_K}$ and $U_{h,K}^{ext}|_{e_K}$, we define
\begin{equation}\label{reconst-h}
\begin{split}
&h_{h,K}^{int,*}|_{e_K}=
\max \Big( 0, h_{h,K}^{int}|_{e_K}+b_{h,K}^{int}|_{e_K}
             -b_{h,K}^{*}|_{e_K}\Big),
\\&
h_{h,K}^{ext,*}|_{e_K} =
\max \Big( 0, h_{h,K}^{ext}|_{e_K}+b_{h,K}^{ext}|_{e_K}
             -b_{h,K}^{*}|_{e_K}
             \Big),
\\&
b_{h,K}^{*}|_{e_K} = \max\big(b_{h,K}^{int}|_{e_K},b_{h,K}^{ext}|_{e_K}\big).
\end{split}
\end{equation}
We redefine
\begin{equation}
\label{reconst-U}
U_{h,K}^{int,*}|_{e_K}= \begin{bmatrix*}[c]
     h_{h,K}^{int,*}|_{e_K}   \\
     h_{h,K}^{int,*}|_{e_K}\frac{m_{h,K}^{int}|_{e_K}}{h_{h,K}^{int}|_{e_K}}\\
     h_{h,K}^{int,*}|_{e_K}\frac{w_{h,K}^{int}|_{e_K}}{h_{h,K}^{int}|_{e_K}}\\
     h_{h,K}^{int,*}|_{e_K}\frac{\eta_{h,K}^{int}|_{e_K}}{h_{h,K}^{int}|_{e_K}}\\
 \end{bmatrix*},\quad
U_{h,K}^{ext,*}|_{e_K}= \begin{bmatrix*}[c]
     h_{h,K}^{ext,*}|_{e_K}   \\
     h_{h,K}^{ext,*}|_{e_K}\frac{m_{h,K}^{ext}|_{e_K}}{h_{h,K}^{ext}|_{e_K}}\\
     h_{h,K}^{ext,*}|_{e_K}\frac{w_{h,K}^{ext}|_{e_K}}{h_{h,K}^{ext}|_{e_K}}\\
     h_{h,K}^{ext,*}|_{e_K}\frac{\eta_{h,K}^{ext}|_{e_K}}{h_{h,K}^{ext}|_{e_K}}\\
  \end{bmatrix*}.
\end{equation}
Finally, the new numerical flux is defined as
\begin{equation}\label{new-flux}
\hat{ F}^{*}|_{e_K} =
\hat{F}(U_{h,K}^{int,*}|_{e_K}, U_{h,K}^{ext,*}|_{e_K})
+\big( F(U_{h,K}^{int}|_{e_K})-F(U_{h,K}^{int,*}|_{e_K})\big)\cdot \bm{n}_K.
\end{equation}
Replacing $\hat{F}|_{e_K}$ by $\hat{F}^{*}|_{e_K}$ in \eqref{semi-DG-s1}, we obtain the semi-discrete DG scheme as
\begin{equation}\label{semi-wbDG-s1}
\begin{split}
\frac{d}{d t}\int_{K}U_h \phi d\bm{x}
- \int_{K}F(U_h)\cdot \nabla \phi d\bm{x}
+ \sum_{e_K\in \partial K}\int_{e_K} \hat{F}^{*}|_{e_K} \phi  ds
= \int_{K}S(U_h,b_h)\phi d\bm{x}.
\end{split}
\end{equation}
This can be written as
\begin{equation}\label{semi-wbDG-ode}
\frac{d}{d t}\int_{K}U_h \phi d\bm{x}=R_{h,K}^{*}(U_h,b_h,\phi), \quad \forall \phi \in \mathcal{V}_h^{k, n+1}
\end{equation}
where
\[
R_{h,K}^{*}(U_h,b_h,\phi)
= \int_{K}S(U_h,b_h)\phi d\bm{x}
+\int_{K}F(U_h)\cdot \nabla \phi d\bm{x}
-\sum_{e_K\in \partial K}\int_{e_K} \hat{F}^{*}|_{e_K} \phi  ds .
\]

We now show the above scheme is well-balanced. Assume that the solution is at the lake-at-rest steady state
\eqref{case1}, i.e.,
\begin{equation}
\begin{split}
&h_{h,K}^{int}|_{e_K}+b_{h,K}^{int}|_{e_K}= C_2, \quad h_{h,K}^{ext}|_{e_K}+b_{h,K}^{ext}|_{e_K}=C_2, \\
&\eta_{h,K}^{int}|_{e_K}= C_1 h_{h,K}^{int}|_{e_K}, \quad\quad~ \eta_{h,K}^{ext}|_{e_K}=C_1 h_{h,K}^{ext}|_{e_K},
\\&m_{h,K}^{int}|_{e_K}=m_{h,K}^{ext}|_{e_K}=0, \quad w_{h,K}^{int}|_{e_K}=w_{h,K}^{ext}|_{e_K}=0.
 \end{split}
 \end{equation}
 Using these, \eqref{reconst-h}, and (\ref{reconst-U}),  we get $h_{h,K}^{int,*}|_{e_K} =h_{h,K}^{ext,*}|_{e_K}$
 and then $U_{h,K}^{int,*}|_{e_K} =U_{h,K}^{ext,*}|_{e_K}$.
From the consistency of the numerical flux, we have
\begin{align}
\hat{ F}^{*}|_{e_K} = F(U_{h,K}^{int}|_{e_K})\cdot \bm{n}_K.
\end{align}
Using the above results, we have
\[
\label{residue-wbDG-s1}
\begin{split}
R_{h,K}^{*}(U_h,b_h,\phi) &
= \int_{K}\big(S(U_h,b_h)-\nabla\cdot F(U_h)\big)\phi d\bm{x}
\\&\qquad
+\sum_{e_K\in \partial K}\int_{e_K} \big(F(U_{h,K}^{int}|_{e_K})\cdot \bm{n}_K- \hat{F}^{*}|_{e_K}\big) \phi  ds
\\&=0.
\end{split}
\]
From (\ref{semi-wbDG-ode}), this gives
\[
\frac{d}{d t}\int_{K}U_h \phi d\bm{x}=0, \quad \forall \phi \in \mathcal{V}_h^{k, n+1} .
\]
Hence, the semi-discrete DG scheme (\ref{semi-wbDG-s1}) (or (\ref{semi-wbDG-ode})) preserves
the lake-at-rest steady-state solution \eqref{case1} and therefore, is well-balanced.

In principle, \eqref{semi-wbDG-ode} can be integrated in time using any marching scheme.
We use a third-order strong stability preserving (SSP) Runge-Kutta scheme \cite{SSP-time}.
For $\forall K \in \mathcal{T}_h^{n+1}$, $\forall \phi \in \mathcal{V}_h^{k, n+1}$, we have
\begin{equation}\label{third}
\begin{cases}
\begin{split}
\int_{K} U_h^{(1)}\phi d\bm{x}
& = \int_{K} \tilde{U}_h^n\phi d\bm{x}
+\Delta t_n R_{h,K}^{*}(\tilde{U}_h^n,b_h^{n+1},\phi),
\\
\int_{K} U_h^{(2)}\phi d\bm{x}
& = \frac{3}{4}\int_{K} \tilde{U}_h^n\phi d\bm{x} + \frac{1}{4}\Big(
 \int_{K}U_h^{(1)}\phi d\bm{x}+\Delta t_nR_{h,K}^{*}(U_h^{(1)} ,b_h^{n+1},\phi)\Big),
\\
\int_{K}U_h^{n+1}\phi d\bm{x}
& = \frac{1}{3}\int_{K} \tilde{U}_h^n\phi d\bm{x} +\frac{2}{3}\Big(
\int_{K} U_h^{(2)}\phi d\bm{x}+\Delta t_n R_{h,K}^{*}(U_h^{(2)} ,b_h^{n+1},\phi)\Big),
\end{split}
\end{cases}
\end{equation}
where $\Delta t_n= t_{n+1}-t_{n}$ and $b_h^{n+1}$ is a polynomial approximation of the bottom topography on $\mathcal{T}_h^{n+1}$.

We now show that the fully discrete scheme \eqref{third} is well-balanced.
That is, we need to show $U_h^{n+1} = (h_h^{n+1}, m_h^{n+1}, w_h^{n+1}, \eta_h^{n+1})^T$
satisfies (\ref{case1}) if $U_h^{n} = (h_h^{n}, m_h^{n}, w_h^{n}, \eta_h^{n})^T$ satisfies (\ref{case1}).
To this end, for now we assume that the interpolation/remapping is well-balanced,
i.e., the interpolant $\tilde{U}_h^n$ satisfies (\ref{case1}).
Recall that $R^{*}_{h,K}$ vanishes for the lake-at-rest steady state.
Then, scheme \eqref{third} reduces to
\[
\label{third-lake}
\begin{cases}
\begin{split}
\int_{K} U_h^{(1)}\phi d\bm{x}
& = \int_{K} \tilde{U}_h^n\phi d\bm{x},
\\
\int_{K} U_h^{(2)}\phi d\bm{x}
& = \frac{3}{4}\int_{K} \tilde{U}_h^n\phi d\bm{x} + \frac{1}{4} \int_{K}U_h^{(1)}\phi d\bm{x},
\\
\int_{K}U_h^{n+1}\phi d\bm{x}
& = \frac{1}{3}\int_{K} \tilde{U}_h^n\phi d\bm{x} +\frac{2}{3}\int_{K} U_h^{(2)}\phi d\bm{x}.
\end{split}
\end{cases}
\]
This implies that $U_{h,K}^{n+1} \equiv U_{h,K}^{(2)} \equiv U_{h,K}^{(1)} \equiv \tilde{U}_{h,K}^n$, i.e.,
\[
\label{UUt}
h_{h,K}^{n+1} =\tilde{h}^{n}_{h,K}, \quad m_{h,K}^{n+1} = \tilde{m}_{h,K}^{n}, \quad w_{h,K}^{n+1}= \tilde{w}_{h,K}^{n},\quad \eta_{h,K}^{n+1} = \tilde{\eta}_{h,K}^{n}.
\]
From the arbitrariness of $K$, by combining the above result and the assumption that $\tilde{U}_h^n$ satisfies
(\ref{case1}), we conclude that $U_h^{n+1}$ satisfies (\ref{case1}).
Thus, the scheme \eqref{third} is well-balanced.

\vspace{10pt}

We now show that the interpolation/remapping step is well-balanced while preserving the positivity of the water depth and temperature
for the dry situation.
We first recall that we need to interpolate $U_h^n = (h_h^n, m_h^n, w_h^n, \eta_h^n)$
and $b_h^n$ from the old mesh to the new one. We use a DG-interpolation scheme (with or without the linear scaling PP limiter)
(see \S\ref{sec:DG-interp} for detail) for this purpose.
Specifically, we use
\begin{align}\label{DGInterp-U}
\begin{cases}
\tilde{h}^n_h~~ = \text{PP-DGInterp}(h_h^n),\\
\tilde{\eta}^n_h~~= \text{PP-DGInterp}(\eta_h^n),\\
\tilde{m}^n_h~ = \text{DGInterp}(m_h^n),\\
\tilde{w}^n_h~= \text{DGInterp}(w_h^n),\\
b^{n+1}_h = \text{DGInterp}(h_h^n+b_h^n)-\tilde{h}^n_h.\\
\end{cases}
\end{align}
In this step, $b^{n+1}_h$ is obtained as the difference
between $\text{DGInterp}(h_h^n+b_h^n)$ and $\tilde{h}^n_h$ (instead by simply interpolating $b_h^n$ or $b$
on the new mesh). This treatment is crucial to the preservation
of constant $(h+b)$, as demonstrated in \cite{Zhang-Huang-Qiu-2021JSC}.
Unlike the SWEs, we also need to check whether
the form $\eta = C_1 h$ is preserved or not.
Fortunately, the DG-interpolation scheme in \S\ref{sec:DG-interp}
also satisfies the linearity property (cf. Proposition \ref{LP-interp}) in addition to preserving
constant solutions (cf. Proposition \ref{WB-interp}). From these properties, we have
\[
\tilde{\eta}^n_h = \text{PP-DGInterp}(\eta_h^n) = \text{PP-DGInterp}(C_1 h_h^n) = C_1\cdot \text{PP-DGInterp}(h_h^n)
= C_1 \tilde{h}^n_h .
\]
Thus, $\eta = C_1 h$ is preserved and the interpolation/remapping step is well-balanced.

\vspace{10pt}

Since spurious oscillations and nonlinear instability can occur in numerical solutions,
we need to apply a nonlinear slope limiter after each Runge-Kutta stage.
We use a characteristic-wise TVB limiter  \cite{DG-series5} for this purpose.
However, the TVB limiter procedure applied directly to the local characteristic variables based on $\big(h_h,m_h, w_h,\eta_h\big)^T$ may destroy the well-balance property.
To avoid this difficulty, following \cite{Audusse-etal-2004SISC},
we apply the TVB limiter to the local characteristic variables based on $\big((h_h+b_h),m_h, w_h, (\eta_h + (b\theta)_h)\big)^T$ (instead of $\big(h_h,m_h, w_h,\eta_h\big)^T$) to obtain $\big((h_h+b_h)^{mod},m^{mod}_h, w^{mod}_h,(\eta_h + (b\theta)_h)^{mod}\big)^T$, and then define
\begin{equation}\label{TVB-heta}
\begin{split}
&h^{mod}_h = (h_h+b_h)^{mod}-b_h,
\\&\eta^{mod}_h = (\eta_h + (b\theta)_h)^{mod} - (b\theta)_h.
\end{split}
\end{equation}
It is worth pointing out that the TVB limiter procedure actually contains two steps: the first one is to check whether any limiting is needed based on $\big((h_h+b_h),m_h, w_h, (\eta_h + (b\theta)_h)\big)^T$ in a cell; and, if yes, the second step is to apply the TVB limiter to modify $\big((h_h+b_h),m_h, w_h, (\eta_h + (b\theta)_h)\big)^T$ in the cell.

In the above limiting process, we need to reconstruct its DG approximation $(b\theta)_h \in \mathcal{V}_h^k$
based on $h_h$, $\eta_h$, and $b_h$. If the reconstruction satisfies
\begin{equation}
\label{b-theta-1}
(b\theta)_h = C_1 b_h
\end{equation}
for the lake-at-rest steady-state solution (\ref{case1}), we can show that the above limiting process is well-balanced.
Indeed, from (\ref{b-theta-1}) we have $h_h+b_h = C_2$, and $\eta_h + (b\theta)_h = C_1 h_h + C_1 b_h = C_1C_2$
for the lake-at-rest steady state.
Since the constants $h + b = C_2$ and $\eta_h + (b\theta)_h = C_1C_2$ will not be affected by the limiting procedure, we have $(h_h+b_h)^{mod} = C_2$ and
$(\eta_h + (b\theta)_h)^{mod} = C_1 C_2$.
Combining the above results, from (\ref{TVB-heta}) and (\ref{b-theta-1}) we get
$h_h^{mod} = C_2 - b_h$ and $\eta_h^{mod} = C_1C_2 - C_1 b_h = C_1 h_h^{mod}$.
Thus, the above limiting process maintains the well-balance property.

We use a DG reconstruction of $(b\theta)_h$ based on $h_h$, $\eta_h$, and $b_h$ as
\begin{equation}
\label{Recs-btheta}
(b\theta)_{h,K}=\frac{\bar{\eta}_{h,K}}{\bar{h}_{h,K}}b_{h,K}, \quad \forall K\in \mathcal{T}_h
\end{equation}
where $\bar{\eta}_{h,K}$ and $\bar{h}_{h,K}$ are the cell averages of $\eta_h$ and $h_h$ on $K$, respectively.
It is not difficult to see that this reconstruction satisfies (\ref{b-theta-1}) for $\eta_h = C_1 h_h$.

The above described TVB limiter does not work well for problems with dry or nearly dry regions where the water depth $h$ is close to zero. Although $\eta=h\theta$ also is close to zero in those regions,
the temperature $\theta =\eta/h$ can hardly be computed accurately due to loss of significance in floating-point arithmetic.
To avoid this difficulty, for problems with dry or nearly dry regions we first identify troubled cells using the TVB limiter
based on the free water surface $(h + b)$ and then carry out the polynomial modification phase of the TVB limiter to the local characteristic variables based on $\big(h_h,m_h, w_h,\eta_h\big)^T$ on the troubled cells.

It is known that the above limiting process preserves the cell averages of $h_h$ and $\eta_h$
since the TVB limiter preserves cell averages. However, it does not necessarily preserve the nonnegativity of $h_h$ and $\eta_h$ (or $\theta$). Thus, after each application of the TVB limiter, we apply the linear scaling PP limiter \cite{Liu-Osher1996,Xing-Zhang-Shu-2010ppSWEs,Xing-Zhang-2013ppSWEs} to modify $h^{mod}_h$ and $\eta^{mod}_h$ as
\begin{equation}\label{PP-heta}
\begin{split}
&\hat{h}^{mod}_h(\bm{x}) = \lambda_{h}\big(h^{mod}_h(\bm{x}) - \bar{h}^{mod}_h \big)+\bar{h}^{mod}_h,~
\lambda_{h} = \min\Big\{1, \frac{\bar{h}^{mod}_h}{\bar{h}^{mod}_h-\min\limits_{\bm{\hat{x}}_\alpha \in G_p}h^{mod}_h(\bm{\hat{x}}_\alpha)}\Big\}\\
&\hat{\eta}^{mod}_h(\bm{x}) = \lambda_{\eta}\big(\eta^{mod}_h(\bm{x}) - \bar{\eta}^{mod}_h \big)+\bar{\eta}^{mod}_h,~
\lambda_{\eta} = \min\Big\{1, \frac{\bar{\eta}^{mod}_h}{\bar{\eta}^{mod}_h-\min\limits_{\bm{\hat{x}}_\alpha \in G_p}\eta^{mod}_h(\bm{\hat{x}}_\alpha)}\Big\}
\end{split}
\end{equation}
where $G_p$ is a set of special quadrature points \cite{Xing-Zhang-Shu-2010ppSWEs,Xing-Zhang-2013ppSWEs}. It can be verified that $\hat{h}^{mod}_h(\bm{x}_\alpha) \geq0$ and $\hat{\eta}^{mod}_h(\bm{x}_\alpha) \geq0$ for $\forall \bm{x}_\alpha \in G_p$. We now show that it preserves $\hat{\eta}^{mod}_h(\bm{x}) = C_1\hat{h}^{mod}_h(\bm{x}) $ for the still water steady state. Since the fully discrete scheme and TVB limiter procedure are all well-balanced, we have
\begin{equation}
\bar{\eta}^{mod}_h = C_1 \bar{h}^{mod}_h,\quad \eta^{mod}_h(\bm{x}) = C_1 h^{mod}_h(\bm{x}),\quad \forall \bm{x} \in \Omega .
\end{equation}
Thus, we have $\lambda_{h} =\lambda_{\eta} $, and then we have $\hat{\eta}^{mod}_h(\bm{x}) = C_1\hat{h}^{mod}_h(\bm{x})$.
However, this PP limiter does not preserve $\hat{h}^{mod}_h + b_h = C_2$ (well-balance) in general. To restore the property, following the idea of \cite{Zhang-Huang-Qiu-2021JSC}, we make a high-order correction to the approximation of $b$ according to
the changes in the water depth due to the PP limiting, i.e.,
\begin{equation}\label{B-correct}
\hat{b}_h = b_h  - (\hat{h}^{mod}_h - h^{mod}_h).
\end{equation}

\vspace{8pt}

\begin{rem}
For dry or nearly dry regions where the water height is close to zero, the velocities $u =(hu)/h$ and  $v =(hv)/h$ cannot be computed accurately due to loss of significance in floating-point arithmetic.
Following \cite{Xing-Zhang-Shu-2010ppSWEs,Xing-Zhang-2013ppSWEs,Zhang-Huang-Qiu-2021JSC}, we set $u = 0$ and $v = 0$ when $h <10^{-6}$ in our computation.
\end{rem}

\vspace{10pt}

From the above discussion,
we have seen that the interpolation/remapping step (\ref{DGInterp-U}), the limiting process (\ref{TVB-heta}),
and the physical PDE solver \eqref{third} are all well-balanced. Hence, the MM-DG method is well-balanced.
The procedure of the MM-DG method is summarized in Algorithm~\ref{MM-DG}.

\begin{breakablealgorithm}
\caption{The MM-DG method for the Ripa model on moving meshes.}\label{MM-DG}
\begin{itemize}
\item[0.] {\bf Initialization.}
Choose an initial mesh $\mathcal{T}_h^0$ and project the initial physical conservative variables $U =(h,m,w,\eta)^T$ and bottom topography function $b$ into the DG space $\mathcal{V}_h^{k,0}$ to obtain approximation polynomials $U^0_h =  (h^0_h,m^0_h,w^0_h,\eta^0_h)^T$ and $b_h^0$.

\item[] For $n = 0, 1, ...$, do

\item[1.] {\bf Mesh adaptation.}
Generate the new mesh $\mathcal{T}_h^{n+1}$ using the MMPDE method based on variable $\mathcal{E}_h^n=\frac{1}{2}\big((u_h^n)^2+(v_h^n)^2\big)+g\theta_h^n(h^n_n+b_h^n)$ and water depth $h_h^n$ (cf. \S\ref{sec:numerical-results}).

\item[2.] {\bf Solution interpolation/remapping.} Interpolate $U^n_h = (h^n_h,m^n_h,w^n_h,\eta_h^n)^T$ and $b^n_h$ from $\mathcal{T}^n_h$ to $\mathcal{T}_h^{n+1}$ using DG-interpolation to obtain $\tilde{U}^n_h = (\tilde{h}^n_h,\tilde{m}^n_h,\tilde{w}^n_h,\tilde{\eta}_h^n)^T$ and $b^{n+1}_h$ (cf. \eqref{DGInterp-U} and \S\ref{sec:DG-interp}).

\item[3.] {\bf Numerical solution of the Ripa model on the new mesh.}
Integrate the Ripa model from $t_n$ to $t_{n+1}$ on the new mesh $\mathcal{T}_h^{n+1}$ using the MM-DG scheme \eqref{third} to obtain $U^{n+1}_h = (h_h^{n+1},m^{n+1}_h,w^{n+1}_h,\eta^{n+1}_h)^T$.
After each of the RK stage, the characteristic-wise TVB limiter is applied to obtain $(h^{mod}_h, m^{mod}_h,w^{mod}_h, \eta^{mod}_h)^{T}$.
Moreover, the linear scaling PP limiter is applied to the water depth $h_h$ and temperature field $\eta_h$, and a corresponding high-order correction approximation of the bottom topography.
\end{itemize}
\end{breakablealgorithm}

\section{DG-interpolation scheme}
\label{sec:DG-interp}

In this section we briefly describe a DG-interpolation scheme \cite{Zhang-Huang-Qiu-2020SISC}
to transfer a numerical solution $q_h^n$ between deforming meshes that have the same number
of vertices and elements and the same connectivity.
The scheme works for arbitrary (large or small) mesh deformation.
It has high-order accuracy, conserves the mass, obeys the geometric conservation law (GCL), preserves constant solutions, and satisfies the linearity.
The reader is referred to \cite{Zhang-Huang-Qiu-2020SISC} for the detail of the scheme.

The interpolation/remapping problem between two deforming meshes $\mathcal{T}^n_h$ and $\mathcal{T}_h^{n+1}$
is mathematically equivalent to solving the pseudo-time PDE
\begin{equation}
\label{pde}
\frac{\partial q}{\partial \varsigma} (\bm{x},\varsigma) =0,\quad (\bm{x},\varsigma)\in
\Omega \times(0,1]
\end{equation}
on the moving mesh $\mathcal{T}_h(\varsigma)$ that is defined as the linear interpolant of $\mathcal{T}_h^{n}$
and $\mathcal{T}_h^{n+1}$.
Specifically, $\mathcal{T}_h(\varsigma)$ has the same number of elements and vertices
and the same connectivity as $\mathcal{T}_h^{n}$ and $\mathcal{T}_h^{n+1}$ and its nodal positions and
displacements (or deformation) are given by
\begin{align}
\label{location+speed}
&\bm{x}_i(\varsigma) = (1-\varsigma)\bm{x}_i^{n}+\varsigma \bm{x}_i^{n+1}, \quad i =
1,...,N_v
\\
&\dot{\bm{x}}_i = \bm{x}_i^{n+1}-\bm{x}_i^{n}, \quad i = 1,...,N_v .
\label{location+speed+1}
\end{align}
Using the Reynolds transport theorem, we can rewrite (\ref{pde}) into
\begin{equation}
\label{dg-interp-1}
\frac{d}{d \varsigma}\int_{K} q \phi d\bm{x}
- \int_{\partial K} q \phi \dot{\bm{X}} \cdot \bm{n} ds
+\int_{K} q \dot{\bm{X}} \cdot \nabla  \phi d\bm{x}=0, \quad
\forall  \phi\in \mathcal{V}^{k}_h(\varsigma)
\end{equation}
where $ \phi_i$ is the linear basis function associated with the vertex $\bm{x}_i$ and $\dot{\bm{X}}$ is
 the piecewise linear mesh deformation function defined as
\begin{equation}\label{Xdot-1}
\dot{\bm{X}}(\bm{x},\varsigma) = \sum_{i=1}^{N_v} \dot{\bm{x}}_i \phi_i(\bm{x},\varsigma) .
\end{equation}
The spatial discretization using DG and temporal discretization using the SSP Runge-Kutta scheme for (\ref{dg-interp-1})
are similar to that for the Ripa model (\ref{ripa-form-2}). For this reason, we omit the derivation and formula of the scheme here
and refer the interested reader to \cite{Zhang-Huang-Qiu-2020SISC} for the detail.
We mention that a local Lax-Friedrichs numerical flux is used and $\Delta \varsigma$ is chosen
according to the CFL condition
\begin{equation}
\label{DG-interp-cfl}
\Delta \varsigma = \frac{\mathcal{C}_{p}}{\max\limits_{e,K}|\dot{\bm{X}}^{e}\cdot \bm{n}^e_K |}
\cdot\min\big(a^{n}_{min},a^{n+1}_{min}\big) ,
\end{equation}
where $\mathcal{C}_{p}$ is the CFL condition constant number and usually taken less than $1/(2k+1)$ and $a^{n}_{min}$ and $a^{n+1}_{min}$ are the minimum element heights of $\mathcal{T}_h^n$ and $\mathcal{T}_h^{n+1}$, respectively.
Moreover, the volume of mesh elements need to be updated at each stage of the RK scheme so that the geometric conservation law
is not violated.

The interpolation/remapping scheme has the following properties.

\begin{prop}[\cite{Zhang-Huang-Qiu-2020SISC}]
\label{fully-mass}
The DG-interpolation scheme conserves the mass, i.e.,
\begin{align}
\label{f3-4}
\sum_{K^{\nu+1}}\int_{K^{\nu+1}} q_h^{\nu+1}  d\bm{x}
=\sum_{K^{\nu}}\int_{K^{\nu}} q_h^{\nu}  d\bm{x}, \quad \nu = 0, 1, ...
\end{align}
\end{prop}

\begin{prop}[\cite{Zhang-Huang-Qiu-2020SISC}]
\label{WB-interp}
The DG-interpolation scheme  preserves constant solutions,
i.e., for any arbitrary constant $C$, $q_h^{\nu} \equiv C$ implies $q_h^{\nu+1} \equiv C$.
\end{prop}

\begin{prop}
\label{LP-interp}
The DG-interpolation scheme is linear in the sense that for any constants $C_1$ and $C_2$ and any functions $q_h, \, p_h \in \mathcal{V}_h^k$, the scheme satisfies
\begin{equation}
\text{{\em DGInterp}}(C_1\cdot q_h + C_2\cdot p_h) = C_1\cdot\text{{\em DGInterp}}(q_h) + C_2\cdot\text{{\em DGInterp}}(p_h).
\end{equation}
\end{prop}

\begin{proof}
The linearity of the scheme can be seen readily from the formula of the scheme \cite{Zhang-Huang-Qiu-2020SISC}.
\end{proof}

\begin{prop}
[\cite{Zhang-Huang-Qiu-2020SISC}]
\label{PP-interp}
The DG-interpolation scheme with the linear scaling PP limiter
\cite{Liu-Osher1996,ZhangShu2010,ZhangXiaShu2012} preserves the solution positivity
in the sense that, for $\nu = 0, 1, ...$, if $q_h^{\nu}$ has nonnegative cell averages for all elements
and is nonnegative at a set of special points $G_{K^{\nu}}$ at each element $K^{\nu}\in\mathcal{T}_h(\varsigma^\nu)$,
then $q_h^{\nu+1}$ has nonnegative cell averages for all elements
and is nonnegative at the corresponding set of special points $G_{K^{\nu+1}}$ at each element of $K^{\nu+1}\in\mathcal{T}_h(\varsigma^{\nu+1})$.
\end{prop}

The DG-interpolation with the linear scaling PP limiter will be denoted as the PP-DG-interpolation.

\begin{prop}\label{SM-PP-interp}
For any arbitrary constant $C\geq0$, for any $q_h \in \mathcal{V}_h^k$, the PP-DG-interpolation scheme satisfies
\begin{equation}
\text{{\em PP-DGInterp}}(Cq_h) = C\cdot\text{{\em PP-DGInterp}}(q_h).
\end{equation}
Particularly, the PP-DG-interpolation scheme preserves constant solutions.
\end{prop}

\section{Numerical examples}
\label{sec:numerical-results}
In this section we present numerical results obtained with the $P^2$ MM-DG method described in the previous sections for a selection of one- and two-dimensional examples of the Ripa model on triangular moving meshes. Unless otherwise stated, the CFL number is taken as $0.18$ and $0.1$ in one and two dimensions, respectively.
For the TVB limiter, the constant $M_{tvb}$ is taken as zero except for the accuracy test in Example~\ref{test2-1d} to avoid the accuracy order reduction near the extrema.
The gravitational acceleration constant is taken as $g = 1$ in the computation.
For examples where the analytical exact solution is unavailable, the numerical solution obtained by the $P^2$-DG method with a fixed mesh of $N=3000$ is taken as a reference solution.

For mesh movement, we use the MMPDE moving mesh method; e.g., see
\cite[Section 4]{Zhang-Huang-Qiu-2021JSC} for a brief yet complete description of the method and
\cite{Huang-Kamenski-2015JCP,Huang-Kamenski-2018MC,Huang-etal-1994Siam,Huang-Russell-2011,Huang-Sun-2003JCP}
for a more detailed description and a development history.
A key idea of the MMPDE method is to view any nonuniform mesh as a uniform one
in some Riemannian metric specified by a tensor $\mathbb{M} = \mathbb{M}(\bm{x},t)$,
a symmetric and uniformly positive definite matrix-valued function that  provides the information needed for
determining the size, shape, and orientation of the mesh elements throughout the domain.

In this work we use an optimal metric tensor based on the $L^2$-norm of piece linear interpolation error
\cite{Huang-Russell-2011,Huang-Sun-2003JCP}.
To be specific, we consider a physical variable $q$ and its finite element approximation $q_h$.
Let $H_K$ be a recovered Hessian of $q_h$ on $K\in \mathcal{T}_{h}$ such as one obtained
using least squares fitting. Assuming that the eigen-decomposition of $H_K$ is given by
\[
H_K = Q\hbox{diag}(\lambda_1,\lambda_2)Q^T,
\]
where $Q$ is an orthogonal matrix, we define
\[
|H_K| = Q\hbox{diag}(|\lambda_1|,|\lambda_2|)Q^T.
\]
The metric tensor is defined as
\begin{equation}
\label{M-u}
\mathbb{M}_{K} =\det \big{(}\beta_{h}\mathbb{I}+|H_K|\big{)}^{-\frac{1}{6}}
\big{(}\beta_{h}\mathbb{I}+|H_K|\big{)},
\quad \forall K \in \mathcal{T}_h
\end{equation}
where $\mathbb{I}$ is the identity matrix, $\det(\cdot)$ is the determinant of a matrix,
and $\beta_{h}$ is a regularization parameter defined through the algebraic equation
\[
\sum_{K\in\mathcal{T}_h}|K|\, \hbox{det}(\beta_{h}\mathbb{I}+|H_K|)^{\frac{1}{3}}
=2\sum_{K\in\mathcal{T}_h}|K|\,
\hbox{det}(|H_K|)^{\frac{1}{3}}.
\]
Roughly speaking, the choice of \eqref{M-u} is to concentrate mesh points in regions where
the determinant of the Hessian is large.

Follow the idea of \cite[Section 4]{Zhang-Huang-Qiu-2021JSC}, we compute the metric tensor based on the equilibrium variable $\mathcal{E}=\frac{1}{2}(u^2+v^2)+g\theta(h+b)$
and the water depth $h$ so that the mesh adapts to the features in the water flow and the bottom topography. Instead of using
$\mathcal{E}$ and $h$ directly as in \cite{Zhang-Huang-Qiu-2021JSC}, we use  $\ln(\mathcal{E})$ and $\ln(h)$ here
in attempt to minimize the effect of dimensional difference between $\mathcal{E}$ and $h$.
More specifically, we first compute $\mathbb{M}^{\mathcal{E}}_{K} $ and $\mathbb{M}^{h}_{K} $ using (\ref{M-u}) with $q = \ln(\mathcal{E})$ and $\ln(h)$, respectively.
Then, a new metric tensor is obtained through matrix intersection \cite{Zhang-Cheng-Huang-Qiu-2020CiCP} as
\begin{equation}\label{M-Srho}
\mathbb{M}_{K}
={\mathbb{M}}^{\mathcal{E}}_{K}
\cap (\delta\cdot \mathbb{M}^{h}_{K}),
\end{equation}
where the parameter $\delta$ is taken as $0.1$ and $1$ in our computation for one- and two-dimensions, respectively.

It has been shown analytically and numerically in \cite{Huang-Kamenski-2018MC}
that the moving mesh generated by the MMPDE method stays nonsingular (free of tangling)
if the metric tensor is bounded and the initial mesh is nonsingular.
Moreover, a number of other moving mesh methods have been developed as well; e.g., see
the books \cite{Bai94a,Huang-Russell-2011} and reviews \cite{Baines-2011,BHR09,Tan05} and references therein.


\begin{example}\label{test1-1d}
(The lake-at-rest steady-state flow test for the 1D Ripa model.)
\end{example}
We choose this example to test the well-balance property of the MM-DG method with smooth and discontinuous topography functions. First, we consider the lake-at-rest steady state and discontinuous bottom topography function as
\begin{equation}\label{test1-B1}
\begin{split}
&h+b=2,\quad u =0, \quad \theta=10,\\
&b(x)=
\begin{cases}
1,& \text{for } x \in (0.3, 0.7) \\
0,& \text{for } x \in (0,0.3) \cup (0.7, 1).\\
\end{cases}
\end{split}
\end{equation}
Next, we consider the lake-at-rest steady state with a two bumps bottom function on $(-2,2)$ as
\begin{equation}\label{test1-B2}
\begin{split}
&h+b=6,\quad u =0, \quad \theta=4,\\
&b(x) = \begin{cases}
0.85(\cos(10\pi(x+0.9))+1),& \text{for } x \in (-1, -0.8) \\
1.25(\cos(10\pi(x-0.4))+1),& \text{for } x \in (0.3, 0.5) \\
0,& \text{otherwise.}
\end{cases}
\end{split}
\end{equation}

We expect that the steady-state solution is preserved since the MM-DG method is well-balanced.
The final simulation time is $t=1$.
To show that the well-balance property is attained up to the level of round-off error (double precision in MATLAB), we present the $L^1$ and $L^\infty$ error for $h+b$, $hu$ and $h\theta$ at $t = 1$ in Tables~\ref{tab:test1-1d-error-d} and~\ref{tab:test1-1d-error-s} for the discontinuous bottom topography \eqref{test1-B1} and the smooth bottom topography \eqref{test1-B2}, respectively.
The results clearly show that the MM-DG method is well-balanced.

\begin{table}[h]
\caption{Example \ref{test1-1d} with initial data \eqref{test1-B1}. Well-balance test for the $P^2$ MM-DG method.}
\vspace{3pt}
\centering
\label{tab:test1-1d-error-d}
\begin{tabular}{ccccccc}
 \toprule
 & \multicolumn{2}{c}{$h+b$}&\multicolumn{2}{c}{$hu$}&\multicolumn{2}{c}{$h\theta$}\\
$N$&$L^1$-error  &$L^{\infty}$-error  &$L^1$-error  &$L^{\infty}$-error&$L^1$-error  &$L^{\infty}$-error \\
\midrule
50	&	7.22E-14	&	1.02E-13	&	2.11E-14	&	5.07E-14	&	3.28E-13	&	7.83E-13	\\
100	&	1.12E-14	&	2.10E-14	&	2.59E-14	&	6.34E-14	&	1.20E-13	&	2.64E-13	\\
 \bottomrule	
\end{tabular}
\end{table}

\begin{table}[h]
\caption{Example \ref{test1-1d} with initial data \eqref{test1-B2}.
Well-balance test for the $P^2$ MM-DG method.}
\vspace{3pt}
\centering
\label{tab:test1-1d-error-s}
\begin{tabular}{ccccccc}
 \toprule
 & \multicolumn{2}{c}{$h+b$}&\multicolumn{2}{c}{$hu$}&\multicolumn{2}{c}{$h\theta$}\\
$N$&$L^1$-error  &$L^{\infty}$-error  &$L^1$-error  &$L^{\infty}$-error&$L^1$-error  &$L^{\infty}$-error \\
\midrule
50	&	1.79E-14	&	2.67E-14	&	4.58E-14	&	9.58E-14	&	2.71E-17	&	1.42E-15	\\
100	&	4.34E-14	&	6.89E-14	&	1.92E-13	&	4.07E-13	&	1.02E-16	&	7.88E-15	\\
 \bottomrule	
\end{tabular}
\end{table}

\begin{example}\label{test2-1d}
(The accuracy test for the 1D Ripa model.)
\end{example}

In this example we verify the high-order accuracy of the well-balanced MM-DG method.
The bottom topography is a sinusoidal hump
\begin{equation*}
b(x)=\sin^2(\pi x),\quad x \in (0,1).
\end{equation*}
Periodic boundary conditions are used for all unknown variables.
The initial conditions are given by
\begin{equation*}
\begin{split}
h(x,0)=5+e^{\sin(2\pi x)},
\quad u(x,0)=\frac{\sin(\cos(2\pi x))}{5+e^{\sin(2\pi x)}},
\quad \theta(x,0)=\sin(2\pi x)+2.
\end{split}
\end{equation*}
The final simulation time is $t=0.04$ while the solution remains smooth.
A reference solution is obtained using the $P^2$-DG method with a fixed mesh of $N=20000$.
The $L^1$ and $L^\infty$ norms of the error with moving and fixed meshes for $h$, $hu$, and $h\theta$ are plotted in Fig.~\ref{Fig:test2-1d}.
It can be seen that the MM-DG method has the expected third-order for $P^2$-DG in both $L^1$ and $L^\infty$ norm.
Moreover, the error is comparable for fixed and moving meshes since the solution is smooth.

\begin{figure}[H]
\centering
\subfigure[errors of $h$]{
\includegraphics[width=0.31\textwidth,trim=50 0 65 10,clip]
{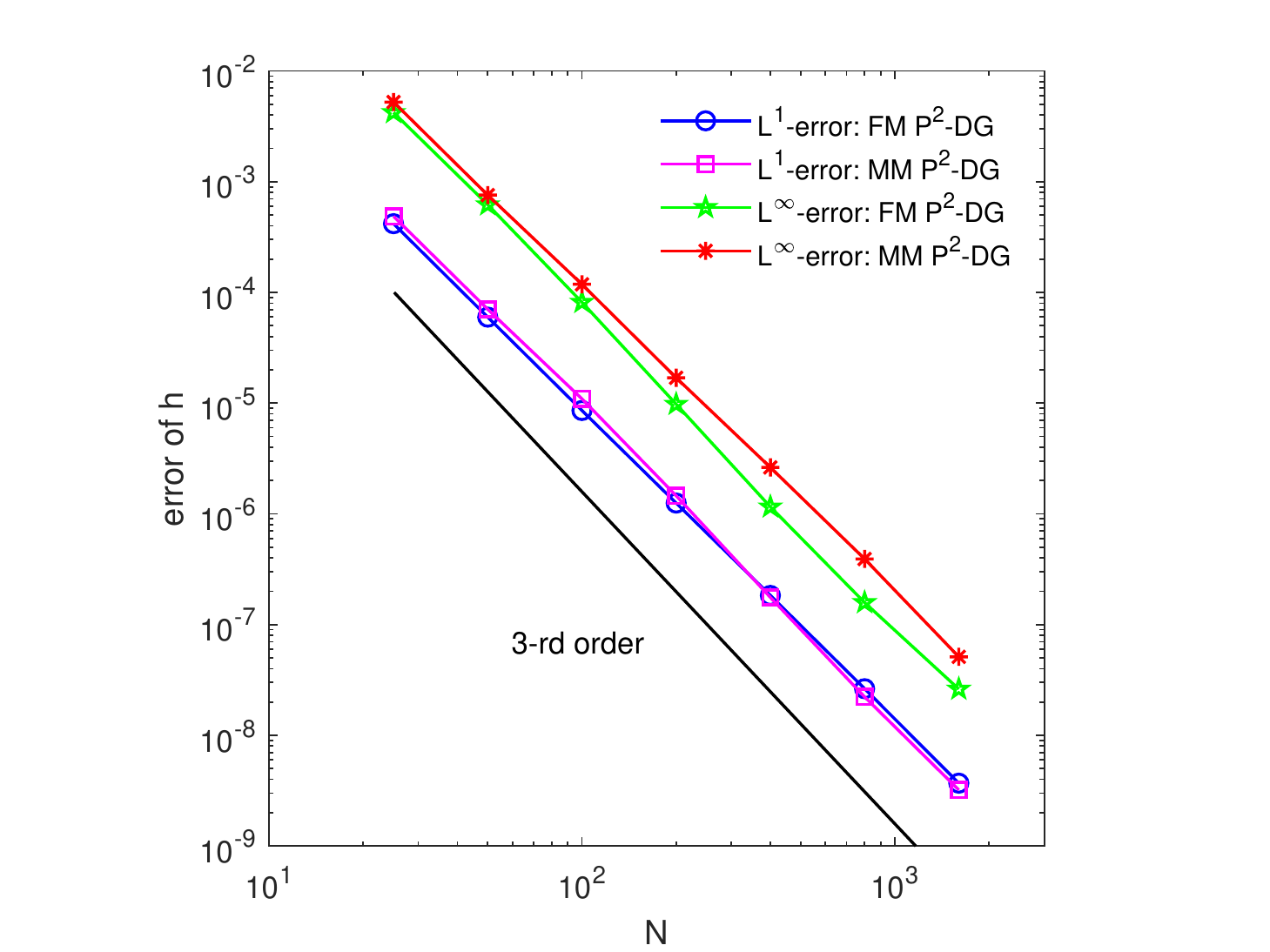}}
\subfigure[errors of $hu$]{
\includegraphics[width=0.31\textwidth,trim=50 0 65 10,clip]
{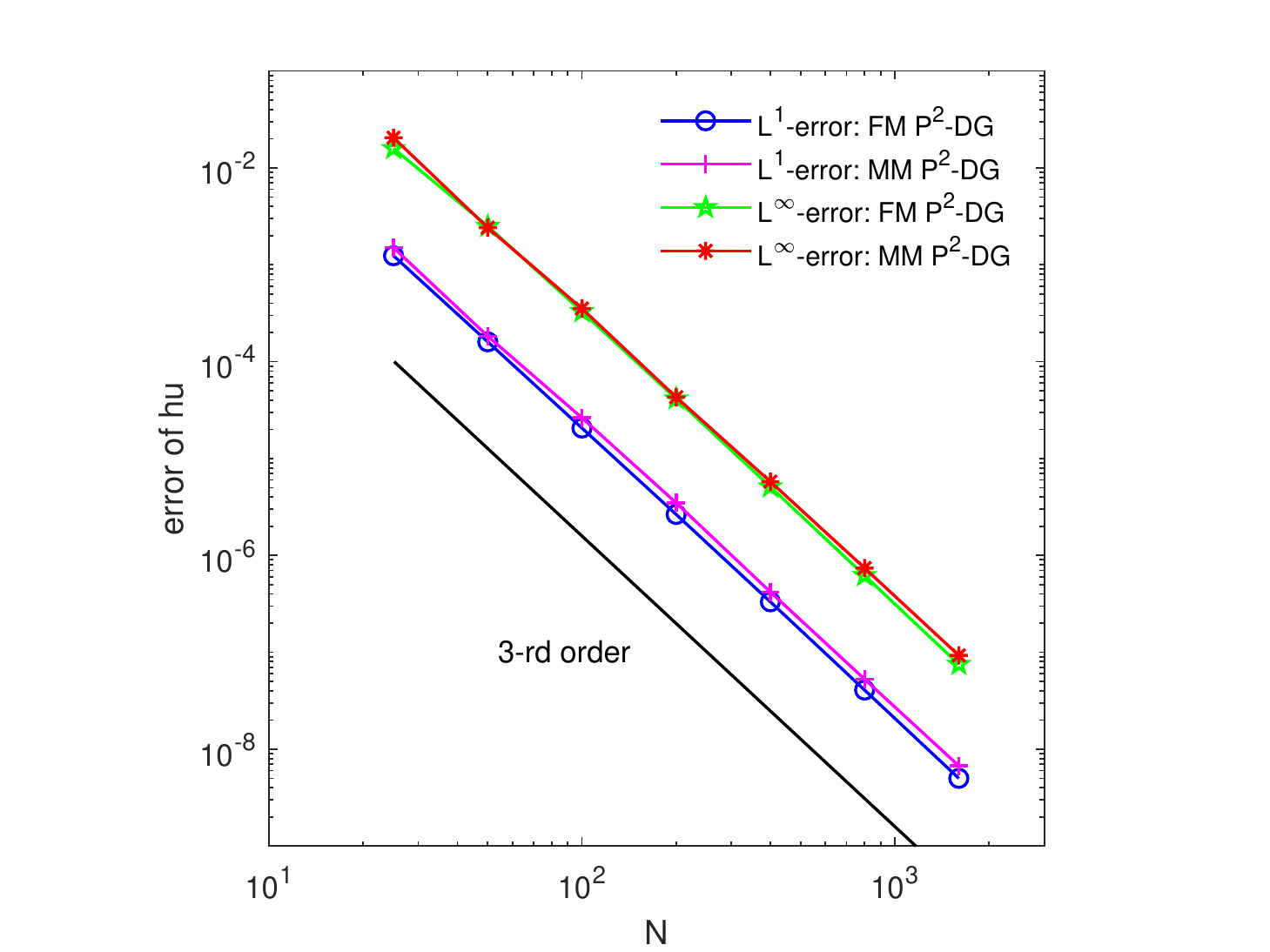}}
\subfigure[errors of $h\theta$]{
\includegraphics[width=0.31\textwidth,trim=50 0 65 10,clip]{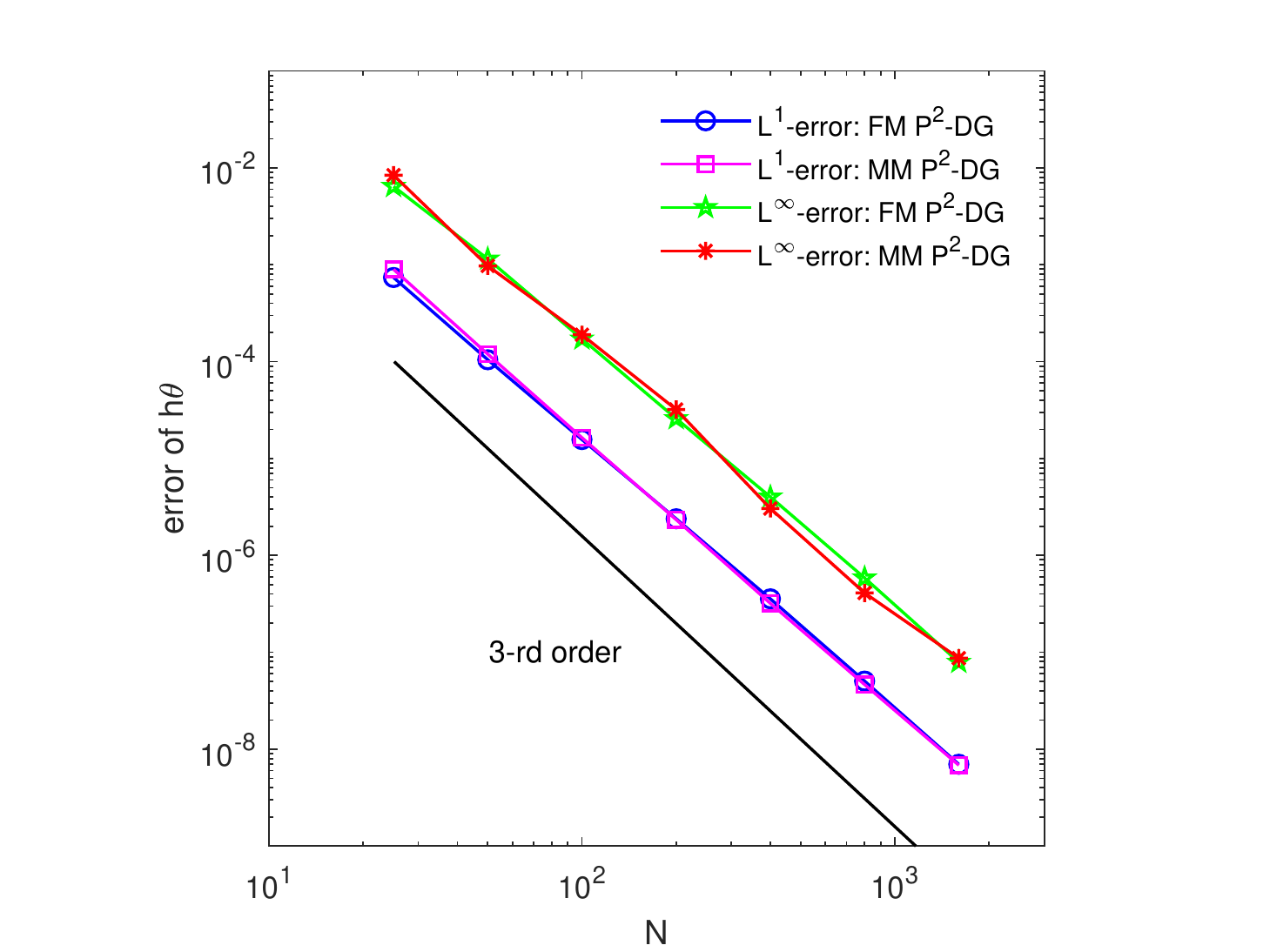}}
\caption{Example \ref{test2-1d}. The $L^1$ and $L^\infty$ norm of the error with moving and fixed meshes for variables $h$, $hu$ and $h\theta$.}
\label{Fig:test2-1d}
\end{figure}

\begin{example}\label{test3-1d}
(The perturbed lake-at-rest steady-state flow test for the 1D Ripa model.)
\end{example}
We choose this example to verify the ability of our well-balanced MM-DG scheme to capture small perturbations over the lake-at-rest water surface and temperature field.
Similar examples have been used by a number of researchers, e.g., \cite{Britton-Xing-2020JSC,Touma-Klingenberg-2015}.
The bottom topography in this example is taken as
\begin{equation}\label{test3-1d-b}
b(x) = \begin{cases}
0.85(\cos(10\pi(x+0.9))+1),& \text{for } x \in (-1, -0.8) \\
1.25(\cos(10\pi(x-0.4))+1),& \text{for } x \in (0.3, 0.5) \\
0,& \text{otherwise}
\end{cases}
\end{equation}
which has two bumps in the computational domain $(-4,2)$.
In this test, we initially add a small perturbation of magnitude $\varepsilon=0.01$ to the free water surface of the still-water steady state. The initial conditions are given by
\begin{equation}\label{Data-1}
\big(h,u,\theta\big)(x,0)=
\begin{cases}
\big(6-b(x)+\varepsilon,~0,~4\big),& \text{for}~x \in (-1.5,-1.4)\\
\big(6-b(x),~~~~~~0,~4\big),& \text{otherwise.}
\end{cases}
\end{equation}
The initial conditions are constituted by the two Riemann problems at $x=-1.5$ and $x=-1.4$, respectively. We compute the solution up to $t=0.4$ when the right wave has already passed the two bottom bumps.

In this case, the temperature $\theta$ stays a constant. As time being, the Riemann problems at $x=-1.5$ and $x=-1.4$ each produce a left-shock wave and a right-shock wave.
At around $t=0.0098$, the right-shock wave of the left Riemann problem at $x=-1.5$ collides with the left-shock wave of the right Riemann problem at $x=-1.4$, which produces a new Riemann problem at $x=-1.45$. Then, the new Riemann problem produces two shock waves that propagate left and right, respectively.

The free water surface evolution along the time obtained with the $P^2$ MM-DG method of $N=300$ is plotted in Fig.~\ref{Fig:test3-1d-Htf-mesh-s2}(a).
The right-propagating waves interacts with the first bottom bump (at around $t = 0.082$)
and generates a complex wave structure in the bump region. Thus, it is beneficial to concentrate mesh points around the bump.
The mesh trajectories ($N=300$) obtained with the $P^2$ MM-DG method are plotted in plotted
in Fig.~\ref{Fig:test3-1d-Htf-mesh-s2}(b).
The mesh has higher concentrations around the shock waves and bottom bumps. It is clear that the mesh adaptation
captures the shock waves before and after the split and interaction of the shock waves with the bottom bumps.

The solutions $h+b$, $hu$, and $h\theta$ at $t = 0.4$ obtained with $P^2$-DG and a moving mesh of $N=300$
and fixed meshes of $N=300$ and $N=900$ are plotted in Figs.~\ref{Fig:test3-1d-s2-H},~\ref{Fig:test3-1d-s2-hu}, and~\ref{Fig:test3-1d-s2-eta}, respectively.
The results show that the DG method with moving or fixed meshes
is able to capture waves of small perturbation.
Moreover, the moving mesh solutions with $N=300$ are more accurate than those with
fixed meshes of $N=300$ and $N=900$ and contain no visible spurious numerical oscillations.

\begin{figure}[H]
\centering
\subfigure[free water surface]{
\includegraphics[width=0.4\textwidth,trim=0 0 20 10,clip]
{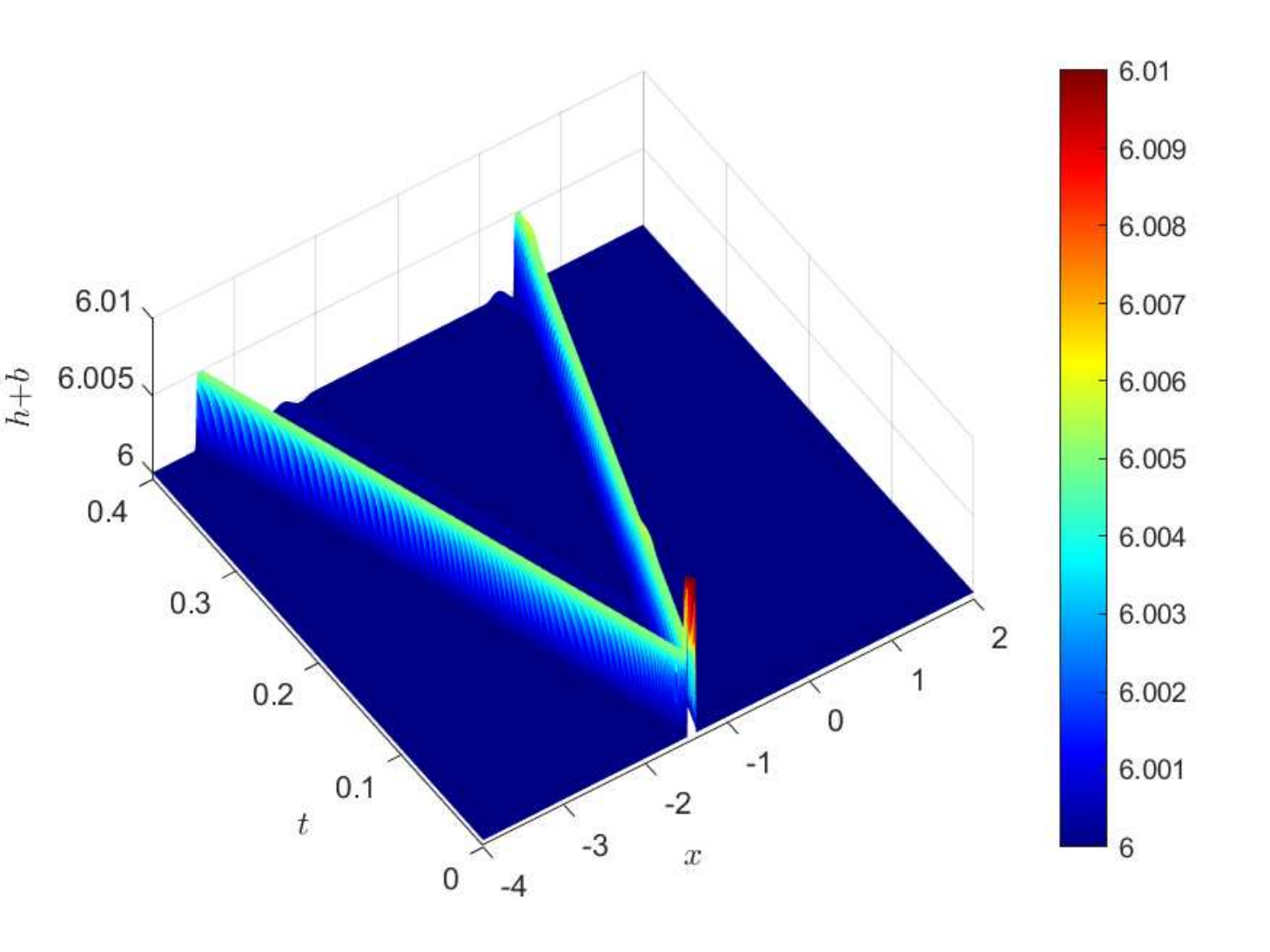}}
\subfigure[mesh trajectories]{
\includegraphics[width=0.4\textwidth,trim=0 0 20 10,clip]{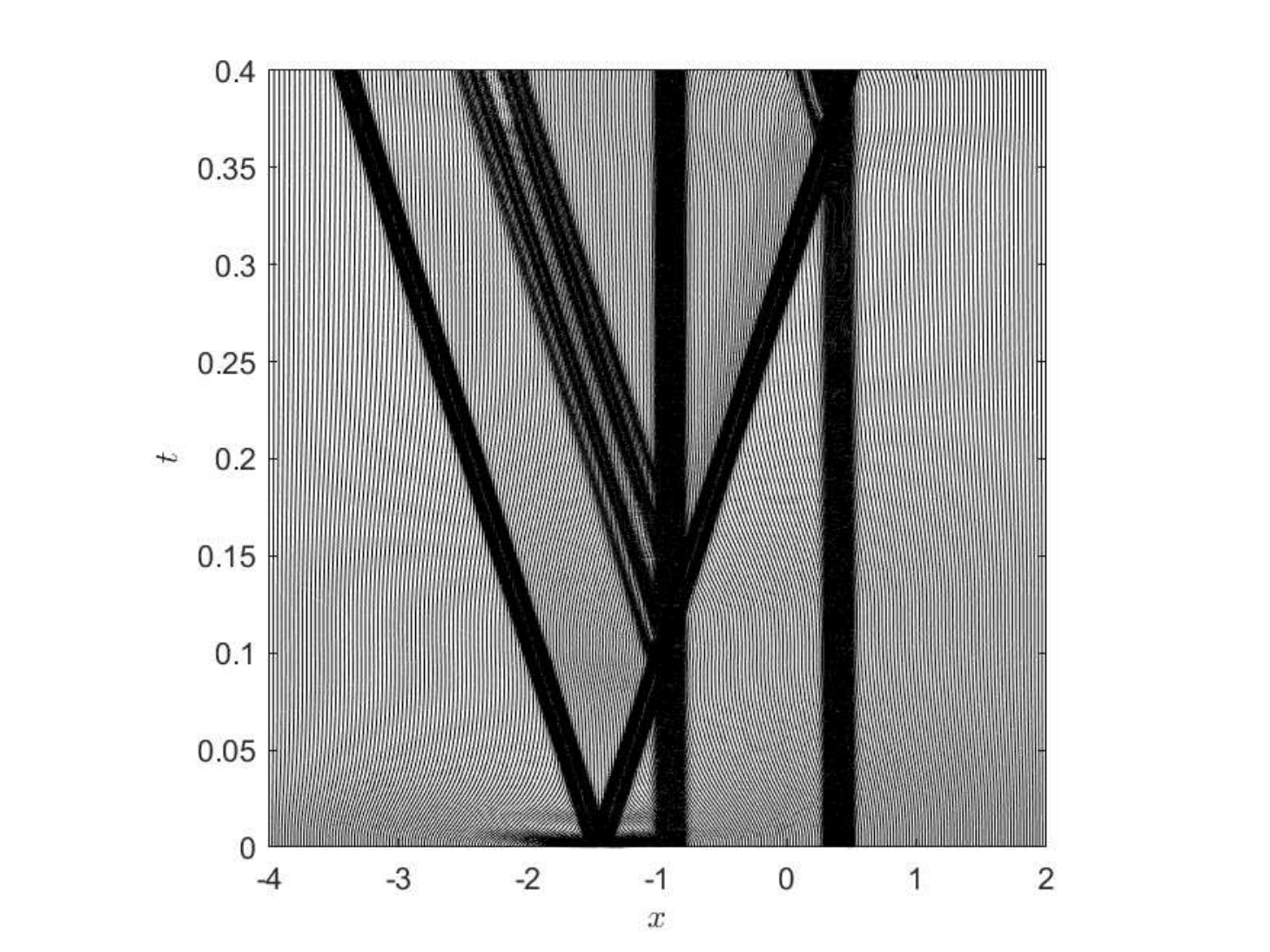}}
\caption{Example \ref{test3-1d} with initial data \eqref{Data-1}. The free water surface evolution along the time and mesh trajectories are obtained with $P^2$-DG of a moving mesh of $N=300$.}
\label{Fig:test3-1d-Htf-mesh-s2}
\end{figure}

\begin{figure}[H]
\centering
\subfigure[$h+b$: MM 300 vs FM 300]{
\includegraphics[width=0.4\textwidth,trim=10 0 40 10,clip]
{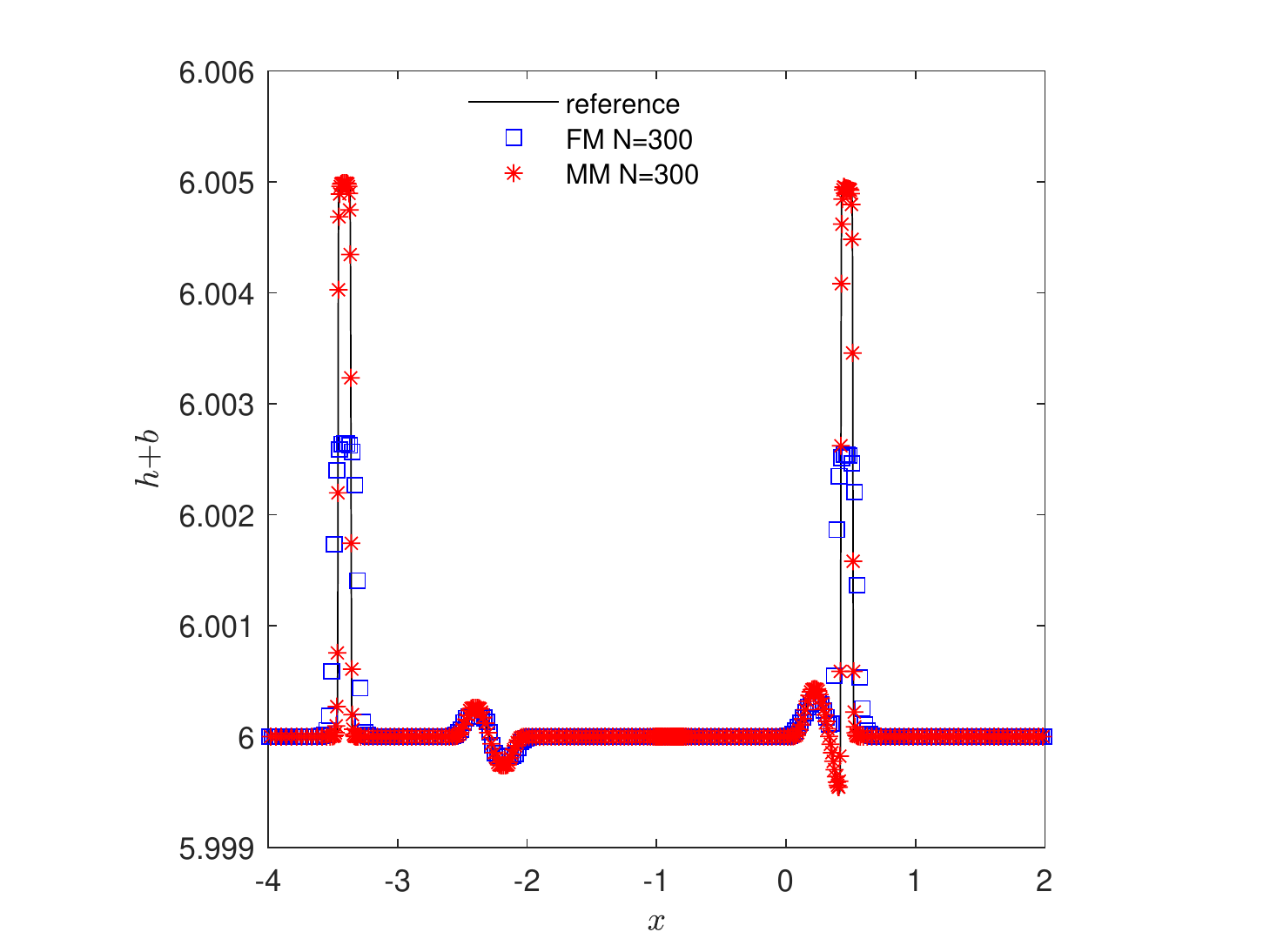}}
\subfigure[close view of (a)]{
\includegraphics[width=0.4\textwidth,trim=10 0 40 10,clip]
{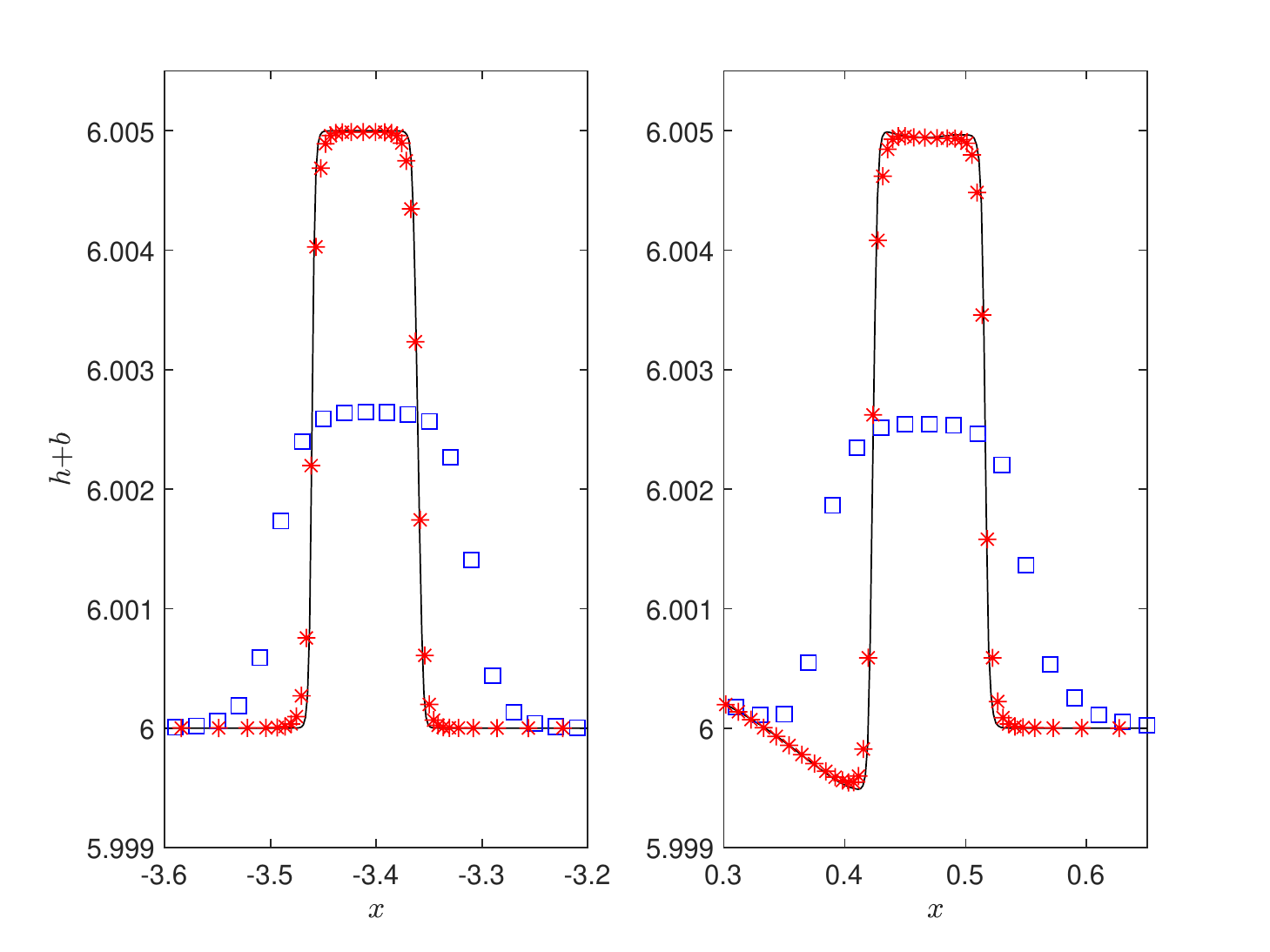}}
\subfigure[$h+b$: MM 300 vs FM 900]{
\includegraphics[width=0.4\textwidth,trim=10 0 40 10,clip]
{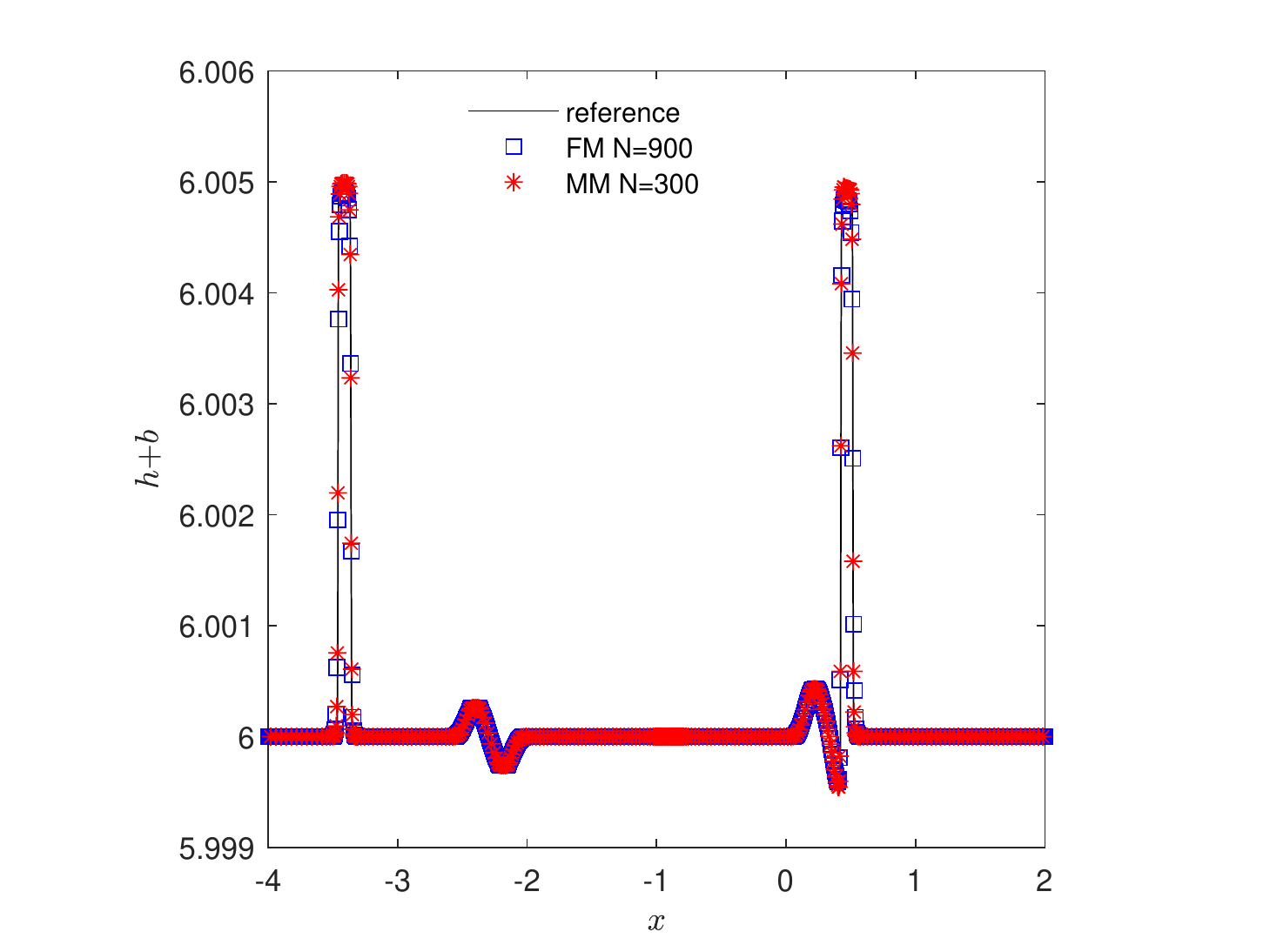}}
\subfigure[close view of (c)]{
\includegraphics[width=0.4\textwidth,trim=10 0 40 10,clip]
{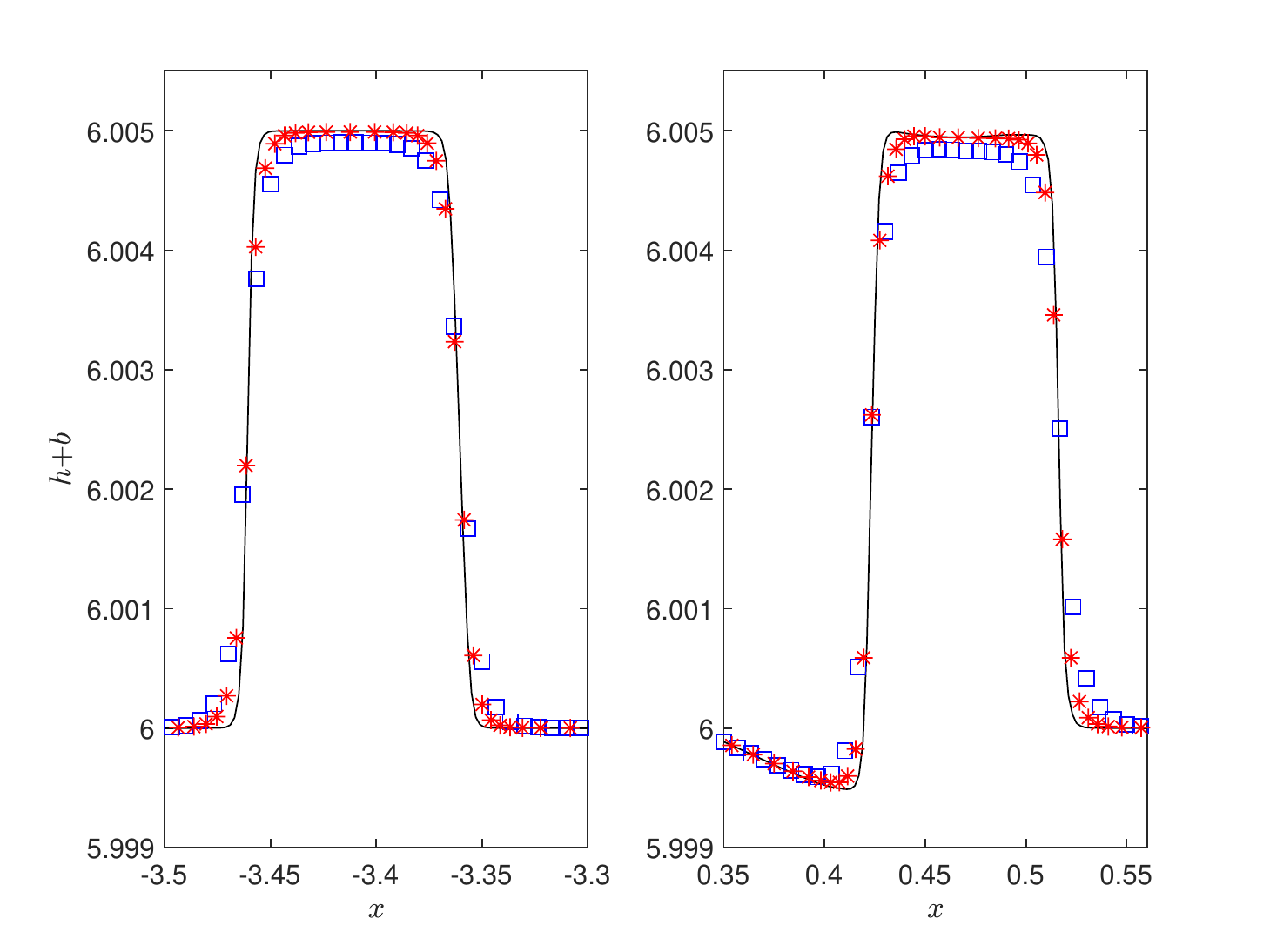}}
\caption{Example \ref{test3-1d} with initial data \eqref{Data-1}. The solutions $h+b$ at $t = 0.4$ obtained with $P^2$-DG and a moving mesh of $N=300$ and fixed meshes of $N=300$ and $N=900$.}
\label{Fig:test3-1d-s2-H}
\end{figure}

\begin{figure}[H]
\centering
\subfigure[$hu$: MM 300 vs FM 300]{
\includegraphics[width=0.4\textwidth,trim=10 0 40 10,clip]
{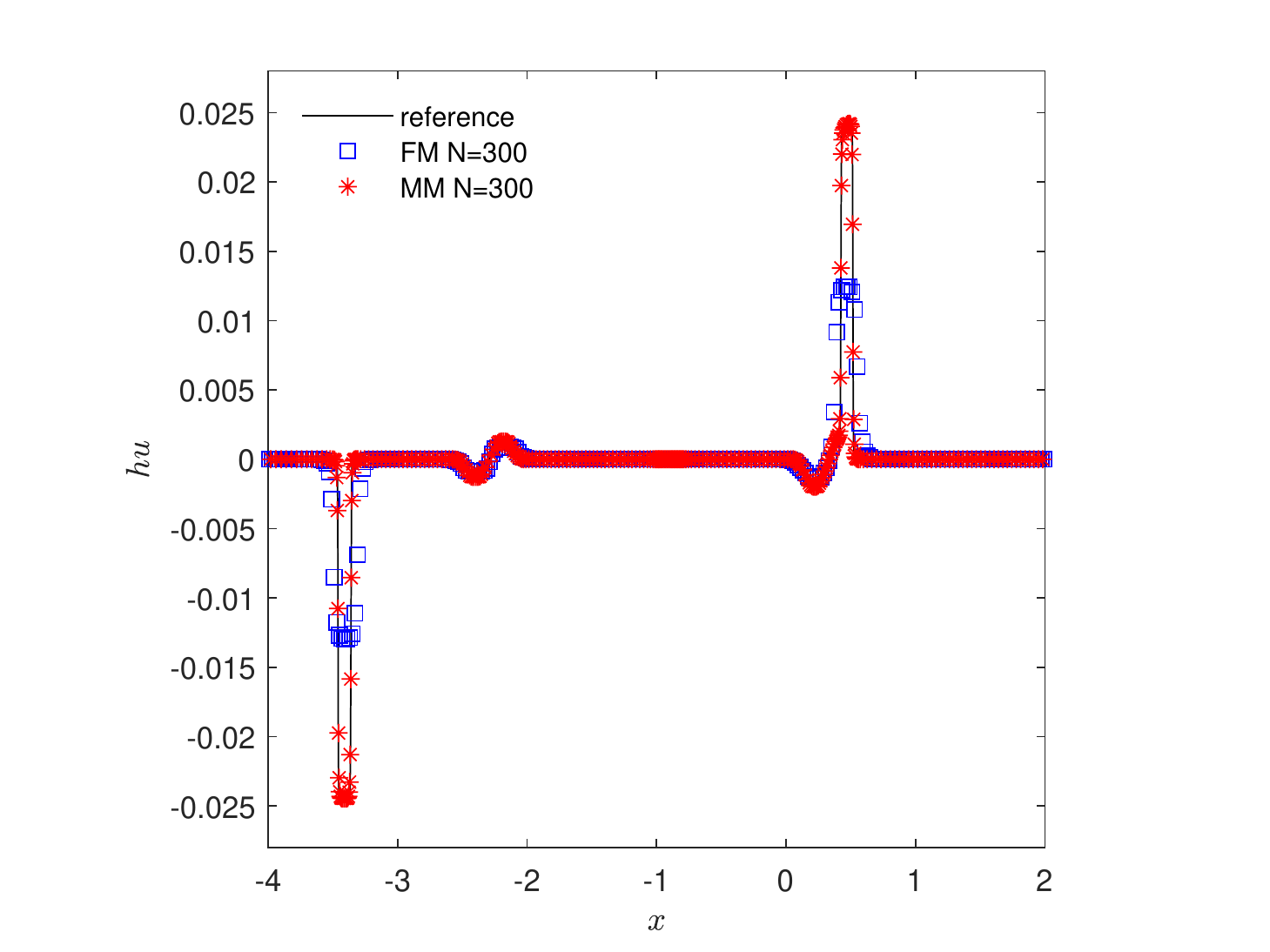}}
\subfigure[close view of (a)]{
\includegraphics[width=0.4\textwidth,trim=10 0 40 10,clip]
{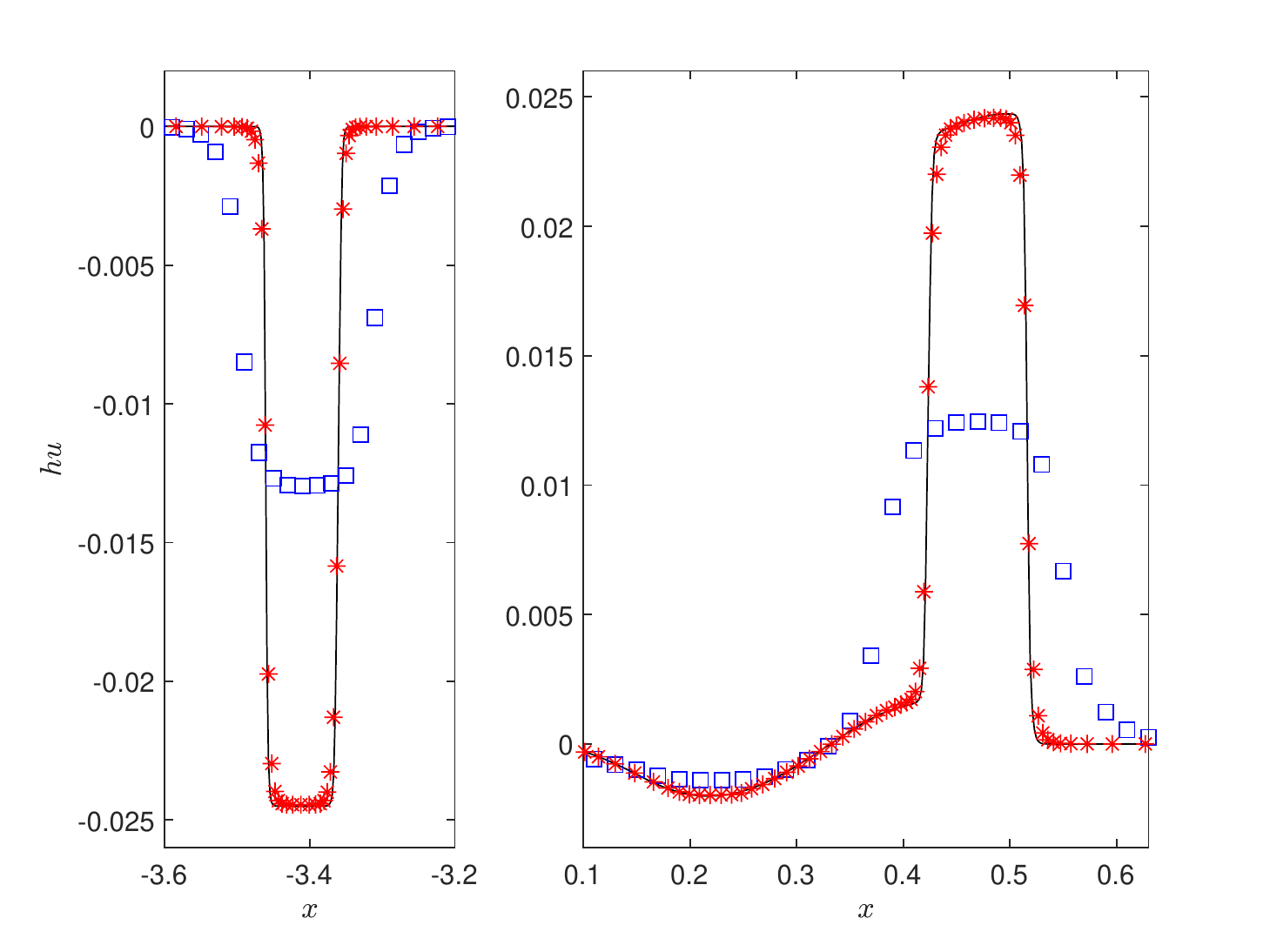}}
\subfigure[$hu$: MM 300 vs FM 900]{
\includegraphics[width=0.4\textwidth,trim=10 0 40 10,clip]
{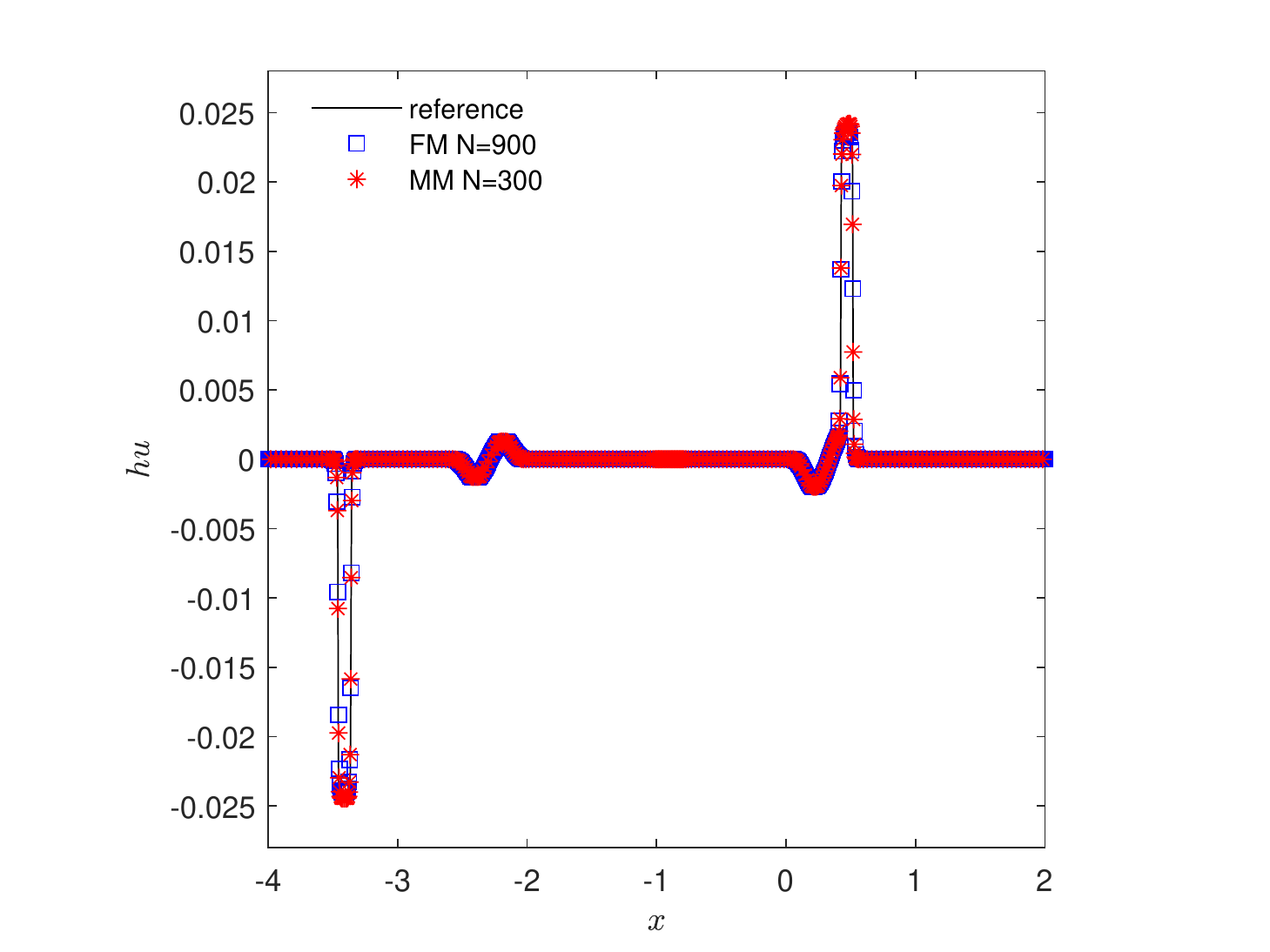}}
\subfigure[close view of (c)]{
\includegraphics[width=0.4\textwidth,trim=10 0 40 10,clip]
{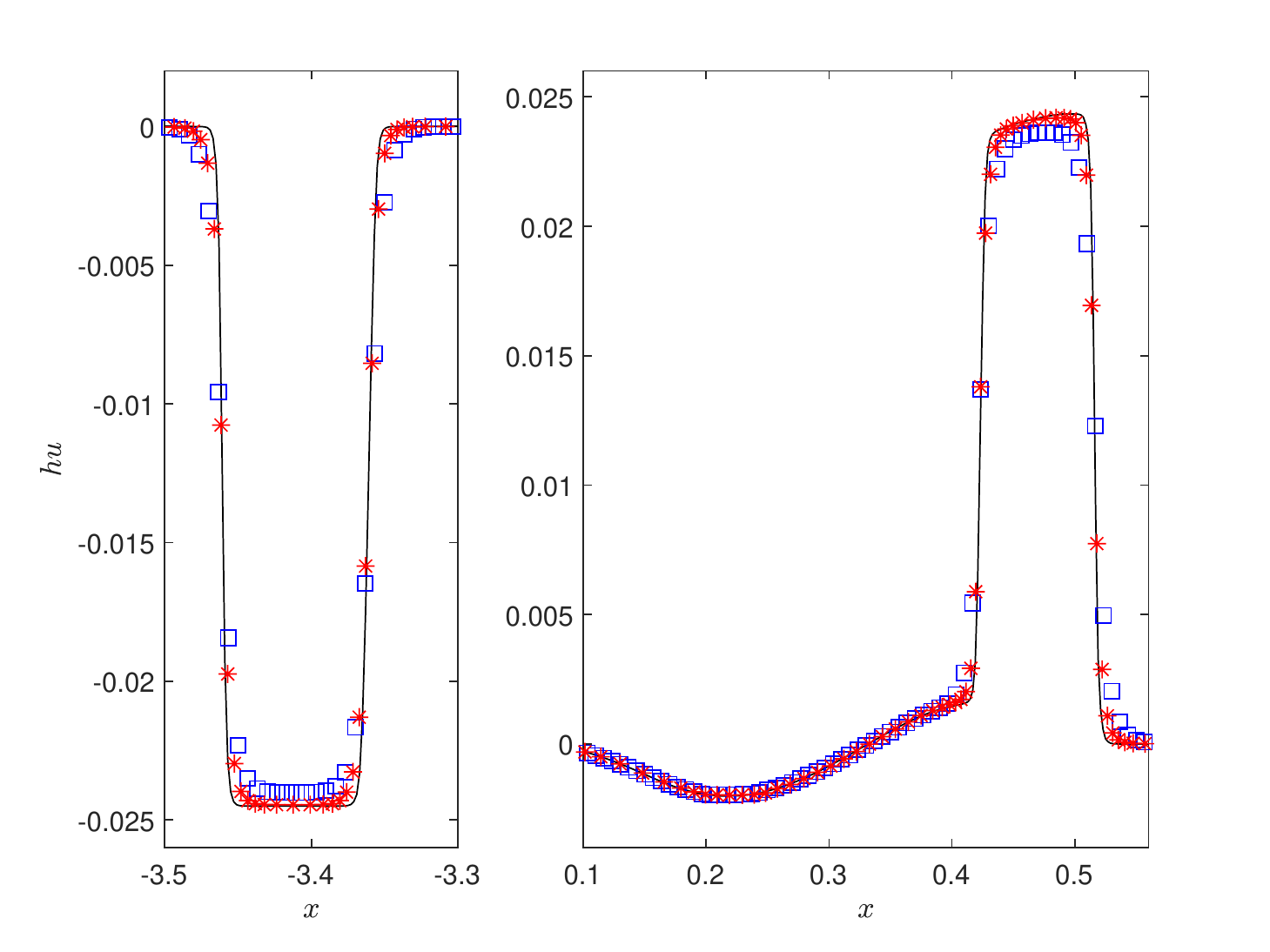}}
\caption{Example \ref{test3-1d} with initial data \eqref{Data-1}. The solutions $hu$ at $t = 0.4$ obtained with $P^2$-DG and a moving mesh of $N=300$ and fixed meshes of $N=300$ and $N=900$.}
\label{Fig:test3-1d-s2-hu}
\end{figure}

\begin{figure}[H]
\centering
\subfigure[$h\theta$: MM 300 vs FM 300]{
\includegraphics[width=0.4\textwidth,trim=10 0 40 10,clip]
{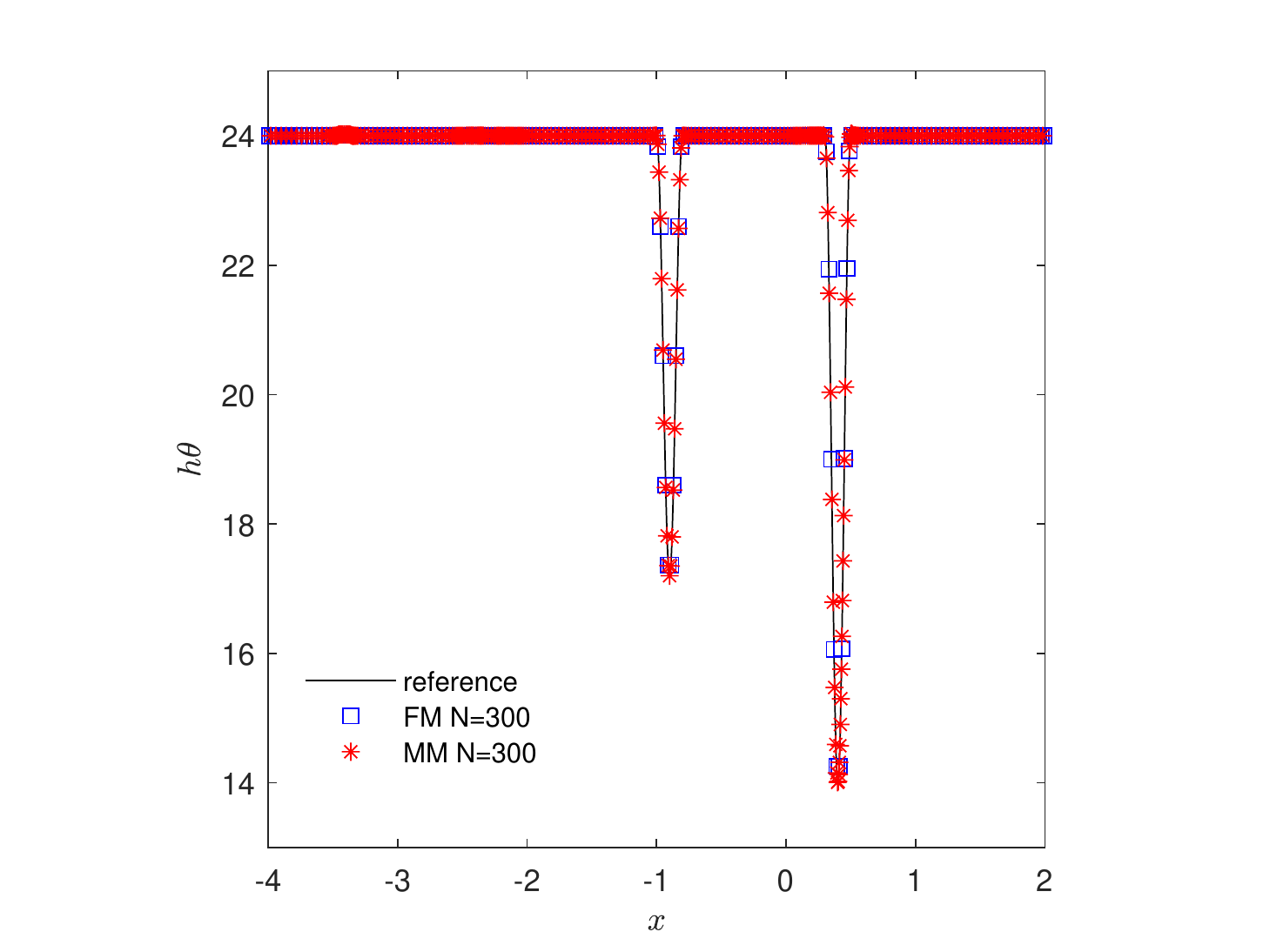}}
\subfigure[close view of (a)]{
\includegraphics[width=0.4\textwidth,trim=10 0 40 10,clip]
{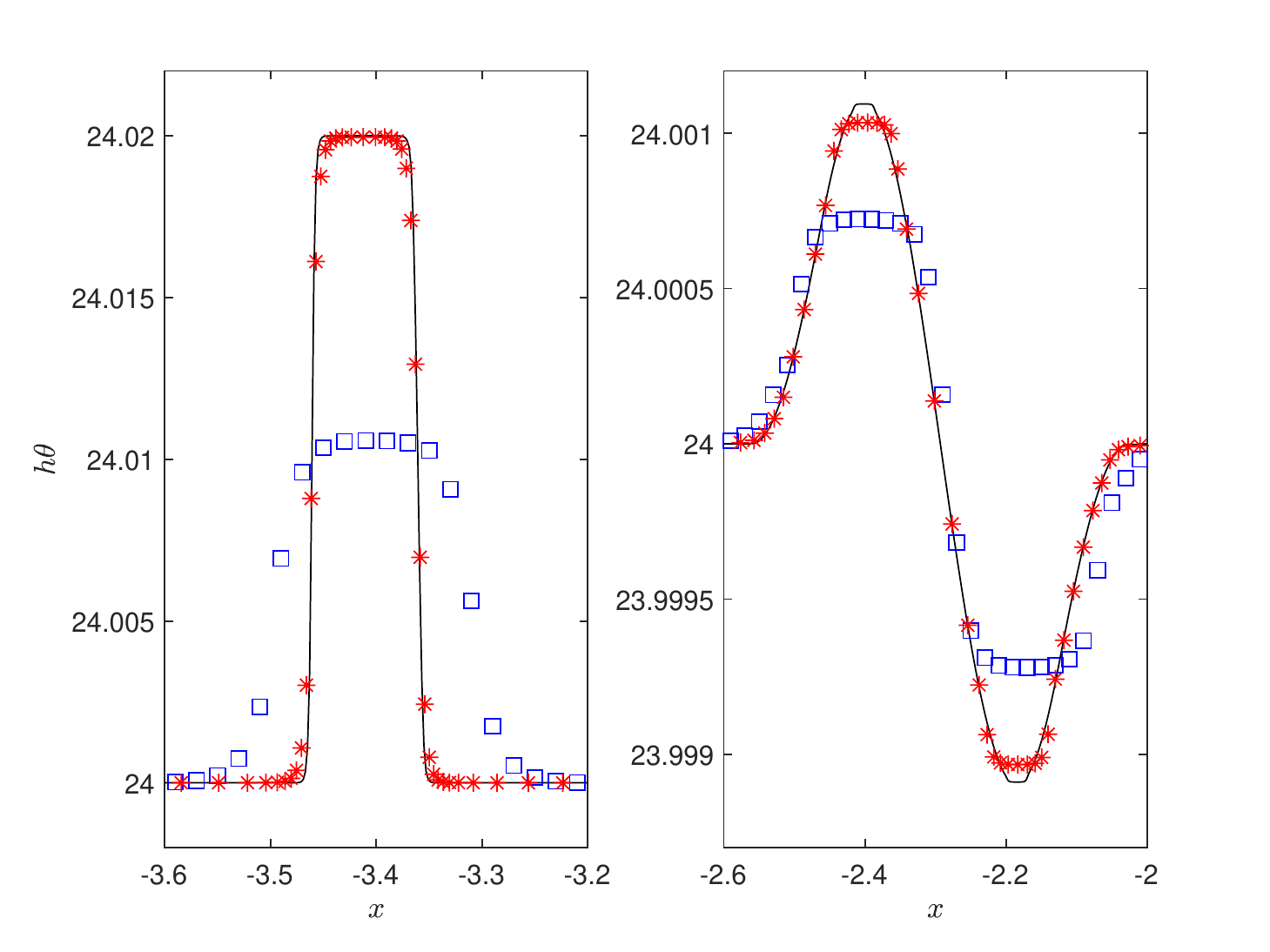}}
\subfigure[$h\theta$: MM 300 vs FM 900]{
\includegraphics[width=0.4\textwidth,trim=10 0 40 10,clip]
{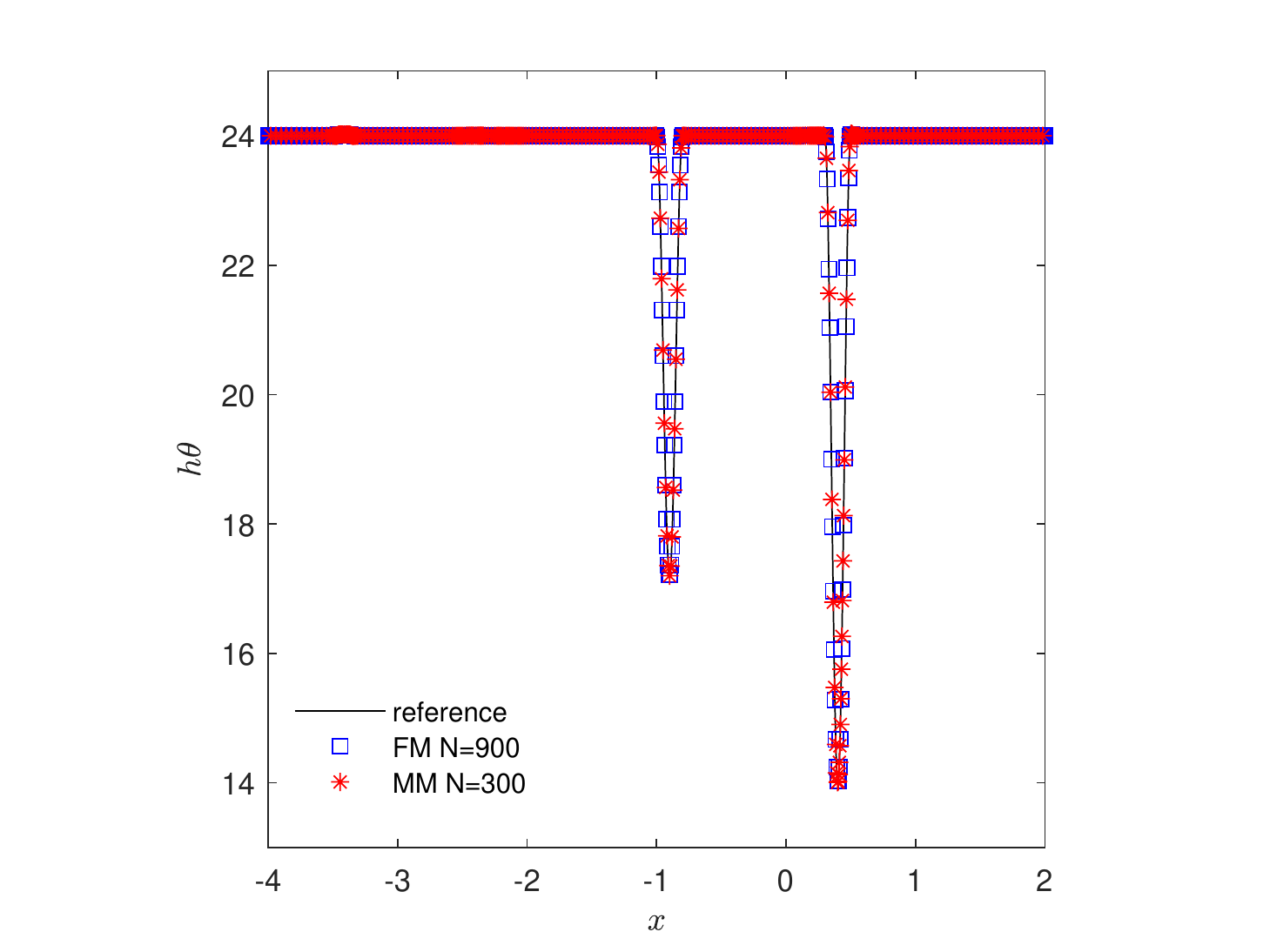}}
\subfigure[close view of (c)]{
\includegraphics[width=0.4\textwidth,trim=10 0 40 10,clip]
{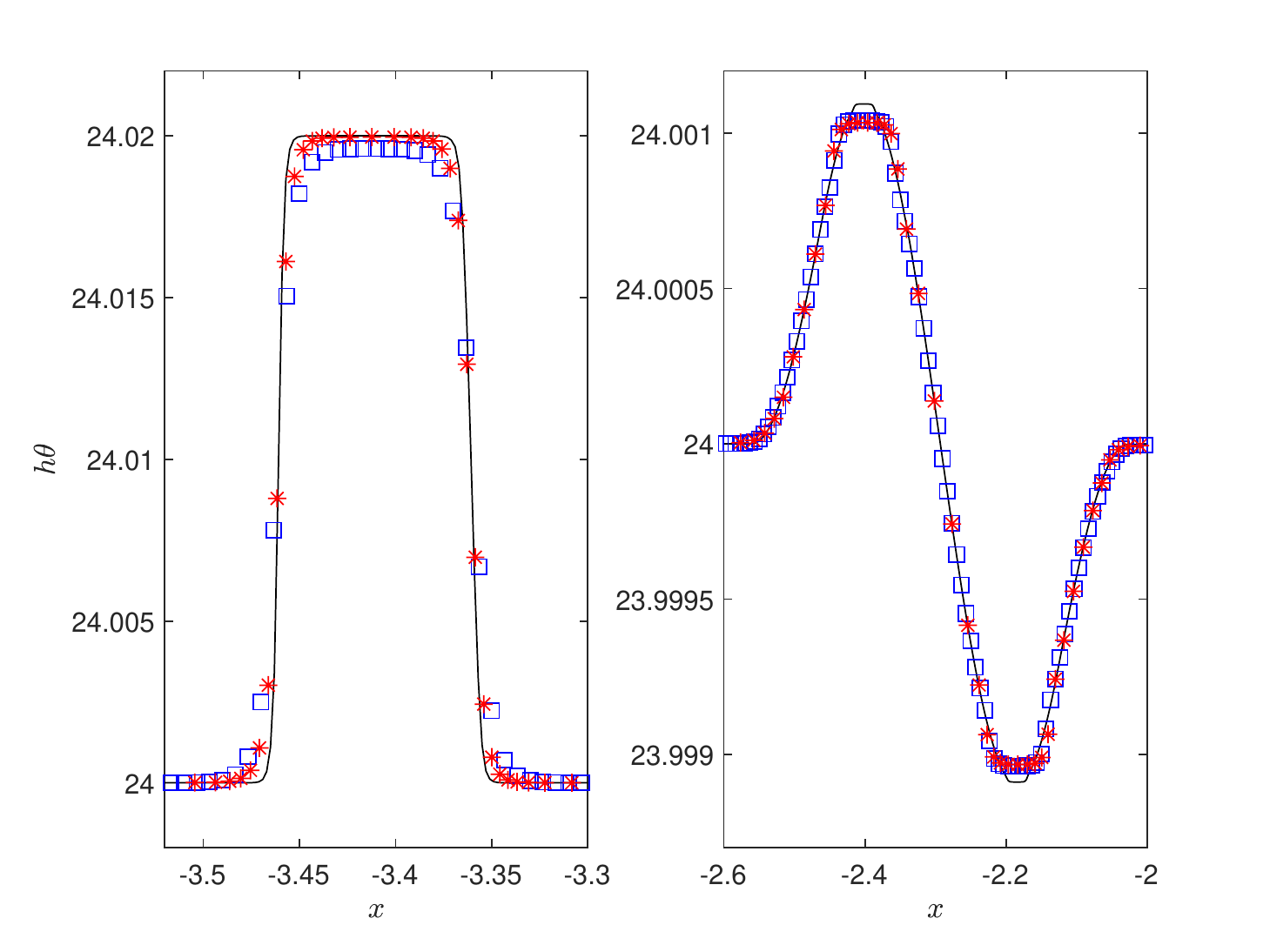}}
\caption{Example \ref{test3-1d} with initial data \eqref{Data-1}. The solutions $h\theta$ at $t = 0.4$ obtained with $P^2$-DG and a moving mesh of $N=300$ and fixed meshes of $N=300$ and $N=900$.}
\label{Fig:test3-1d-s2-eta}
\end{figure}

Next, we consider a situation with perturbations of magnitude $\varepsilon = 0.01$ in both the free water surface and temperature. The initial conditions read as
\begin{equation}\label{Data-2}
\big(h,u,\theta\big)(x,0)=
\begin{cases}
\big(6-b(x)+\varepsilon,~0,~\frac{24}{6+\varepsilon}\big),& \text{for}~x \in (-1.5,-1.4)\\
\big(6-b(x),~~~~~~0,~~4~~\big),&\text{otherwise}
\end{cases}
\end{equation}
which are constituted by the two Riemann problems at $x=-1.5$ and $x=-1.4$, respectively. We compute the solution up to $t=0.4$ when the right wave has already passed the two bottom bumps.

For the current situation, the interactions between waves become more complicated due to the fact that each Riemann problem can
produce left and right shock waves and a center contact discontinuity wave.
For example, at around $t=0.0098$, the right-shock wave of the left Riemann problem at $x=-1.5$ collides with the left-shock wave of the right Riemann problem at $x=-1.4$ and produces a new Riemann problem at $x=-1.45$, which produces two shock waves that propagate left and right, respectively. At around $t=0.02$, these shock waves collide with the contact discontinuities of the Riemann problems at $x=-1.5$ and $x=-1.4$, respectively. Then they produce two new Riemann problem at $x=-1.5$ and $x=-1.4$, respectively.

The free water surface evolution along the time is plotted in Fig.~\ref{Fig:test3-1d-Htf-mesh-s1}(a).
It is interesting to see that the shock waves move away from the point of origin and propagate left and right at the characteristic speeds $\pm \sqrt{gh\theta}$,
respectively, while the the contact discontinuities waves (around a half of magnitude $\varepsilon$, i.e., $0.005$) remains unmoved at the center of the perturbed region.

The mesh trajectories of $N=300$ obtained with the $P^2$ MM-DG method are plotted in Fig.~\ref{Fig:test3-1d-Htf-mesh-s1}(b).
It is clear that the mesh is concentrated around the shock waves, the contact discontinuity waves, and bottom bumps as expected.

The solutions $h+b$, $hu$, and $h\theta$ at $t = 0.4$ obtained with $P^2$-DG and a moving mesh of $N=300$
and fixed meshes of $N=300$ and $N=900$ are plotted in Figs.~\ref{Fig:test3-1d-s1-H},~\ref{Fig:test3-1d-s1-hu}, and~\ref{Fig:test3-1d-s1-eta}, respectively.
The results show that the DG method with moving or fixed meshes is able to capture the waves of small perturbation.
Moreover, the moving mesh solutions with $N=300$ are more accurate than those with
fixed meshes of $N=300$ and $N=900$ and contain no visible spurious numerical oscillations.

\begin{figure}[H]
\centering
\subfigure[free water surface]{
\includegraphics[width=0.4\textwidth,trim=0 0 20 10,clip]
{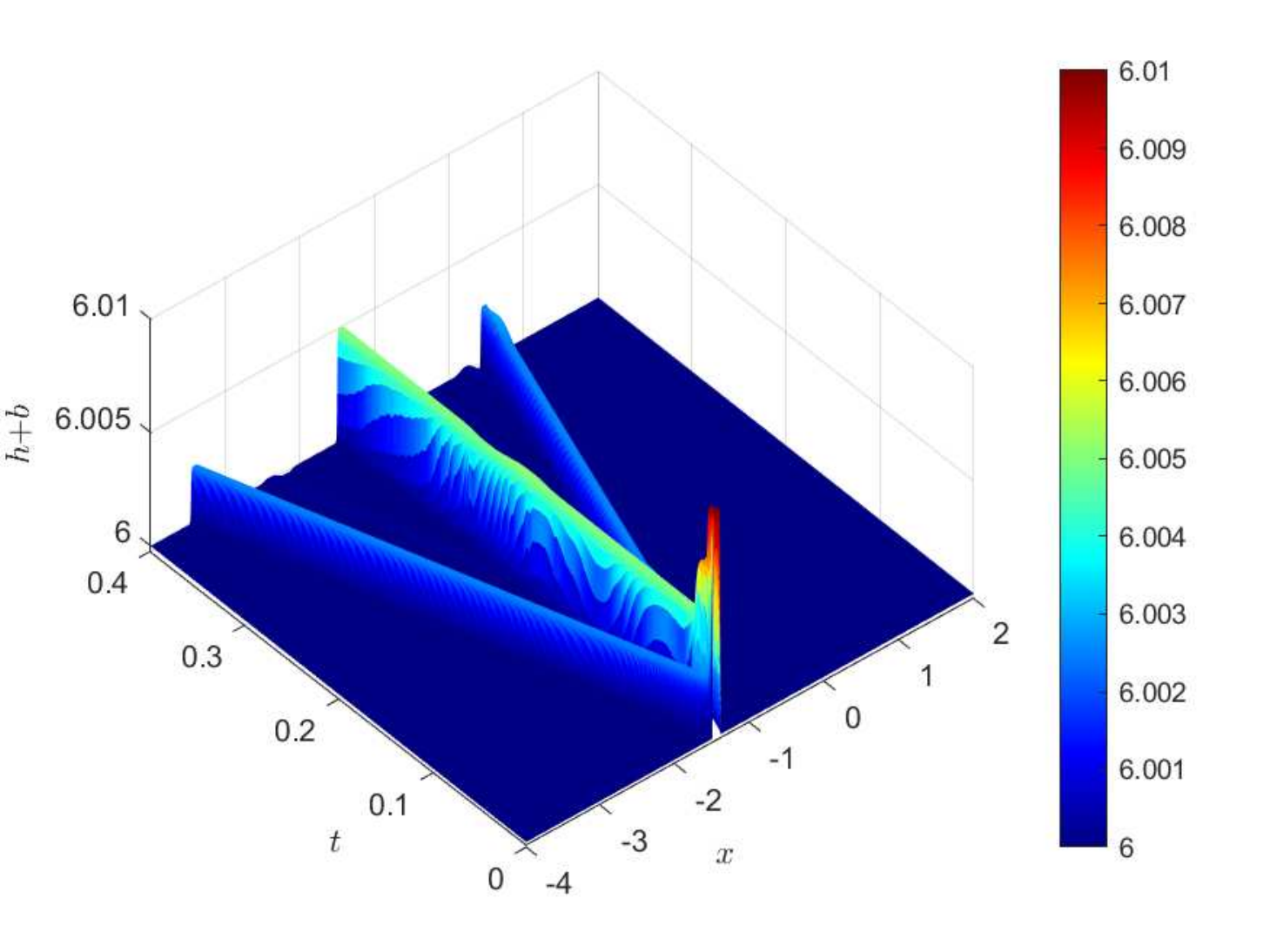}}
\subfigure[mesh trajectories]{
\includegraphics[width=0.4\textwidth,trim=0 0 20 10,clip]{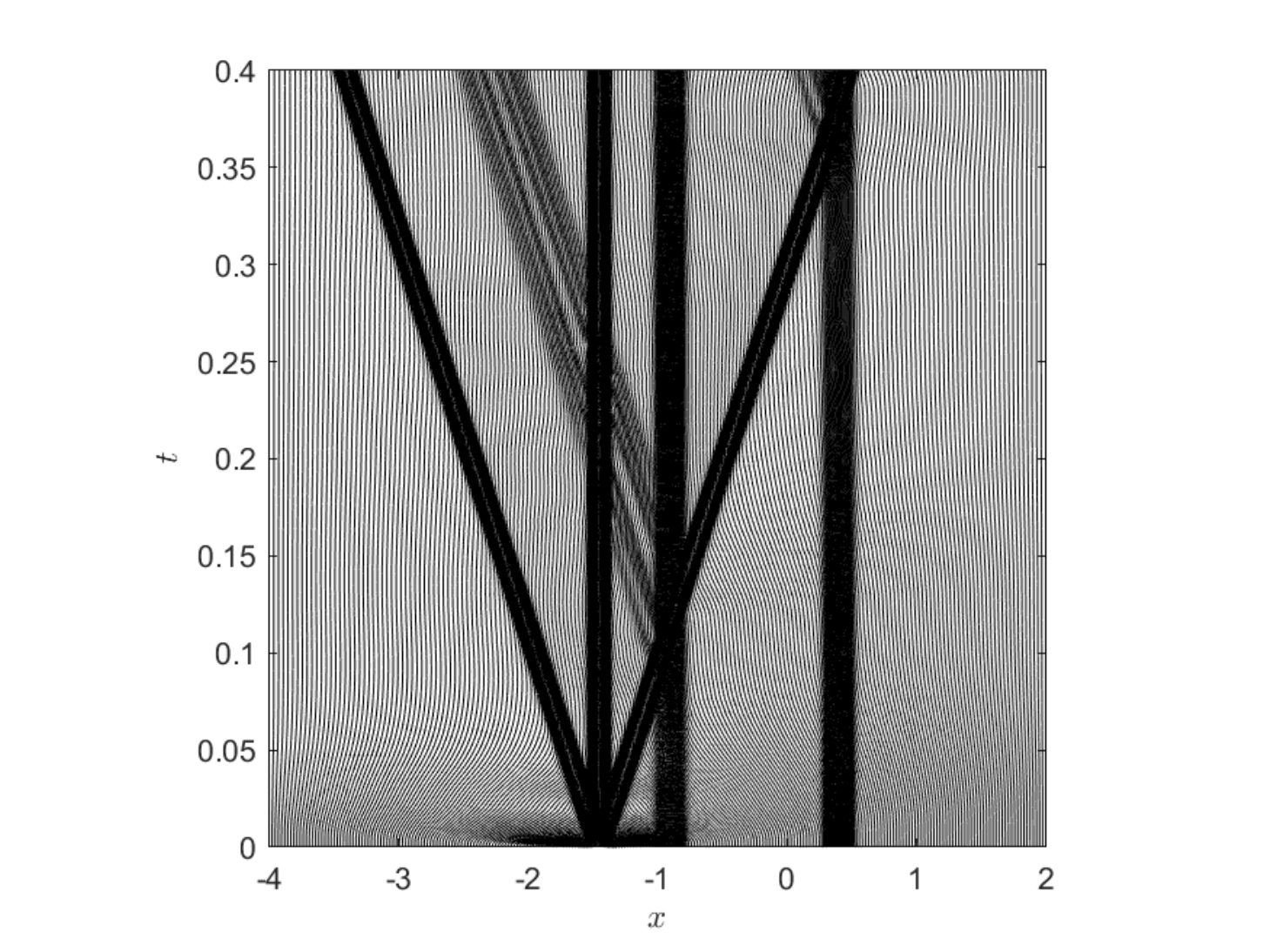}}
\caption{Example \ref{test3-1d} with initial data \eqref{Data-2}. The free water surface evolution along the time and mesh trajectories are obtained with $P^2$-DG of a moving mesh of $N=300$.}
\label{Fig:test3-1d-Htf-mesh-s1}
\end{figure}

\begin{figure}[H]
\centering
\subfigure[$h+b$: MM 300 vs FM 300]{
\includegraphics[width=0.4\textwidth,trim=10 0 40 10,clip]
{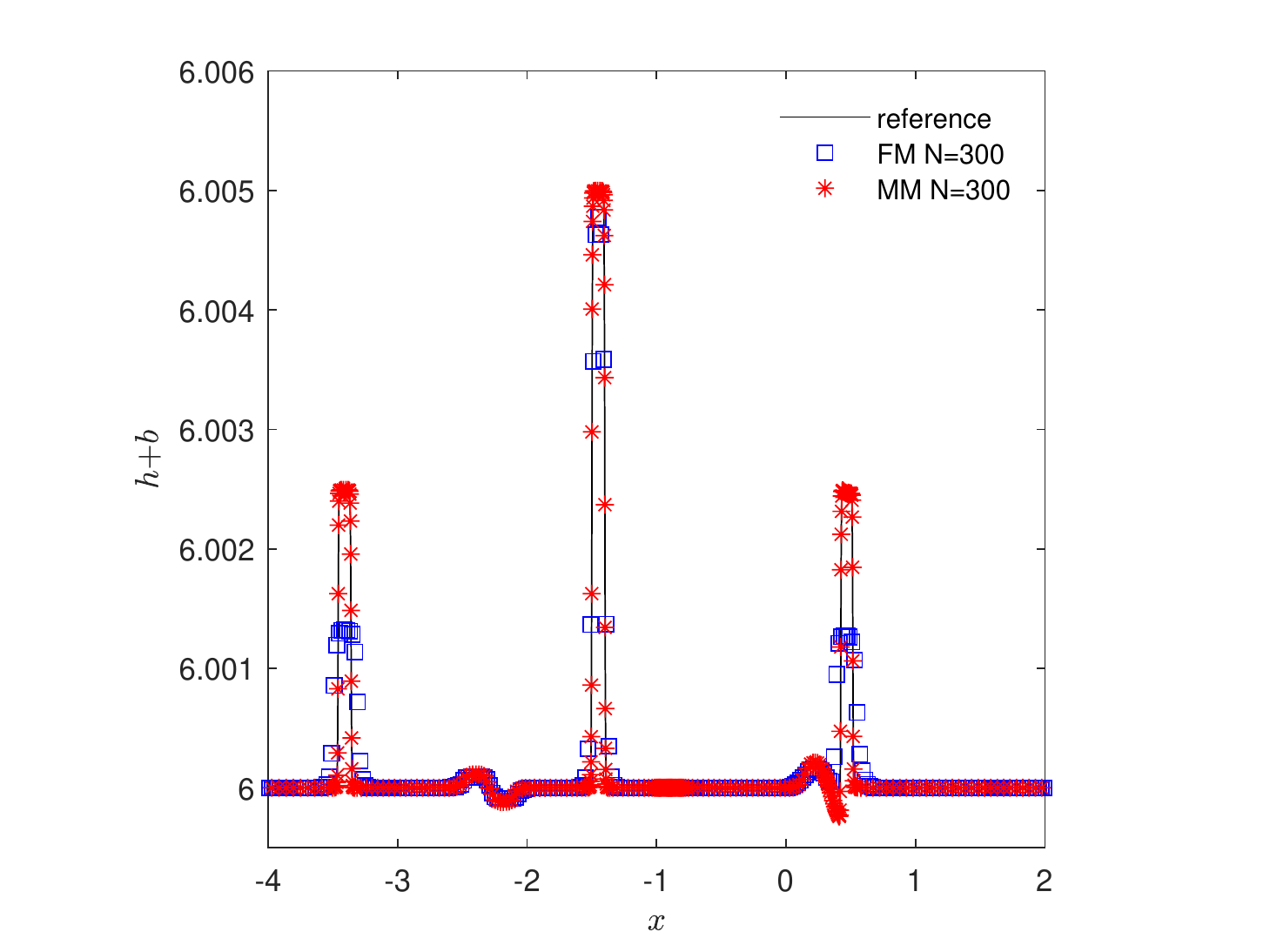}}
\subfigure[close view of (a)]{
\includegraphics[width=0.4\textwidth,trim=10 0 40 10,clip]
{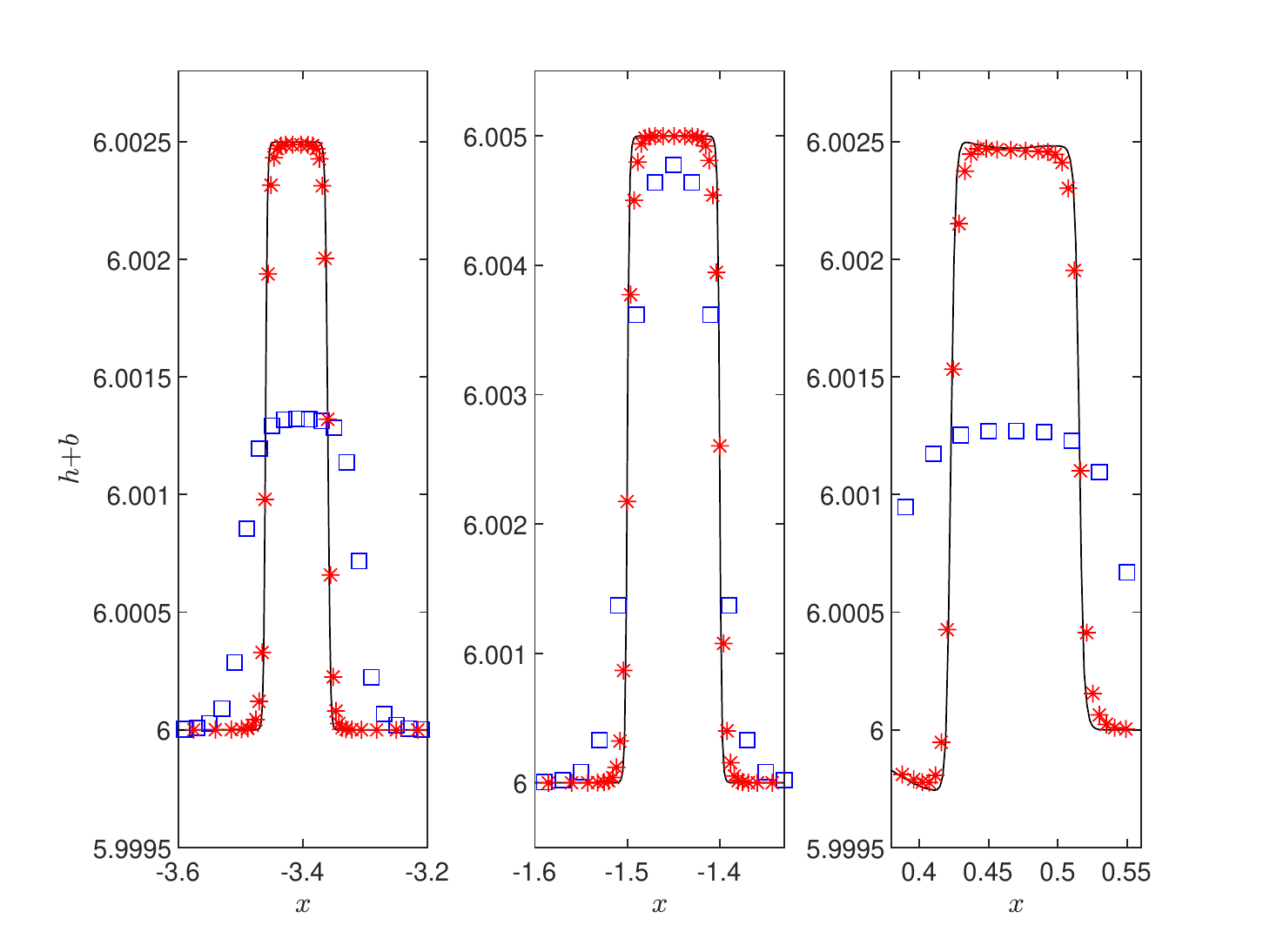}}
\subfigure[$h+b$: MM 300 vs FM 900]{
\includegraphics[width=0.4\textwidth,trim=10 0 40 10,clip]
{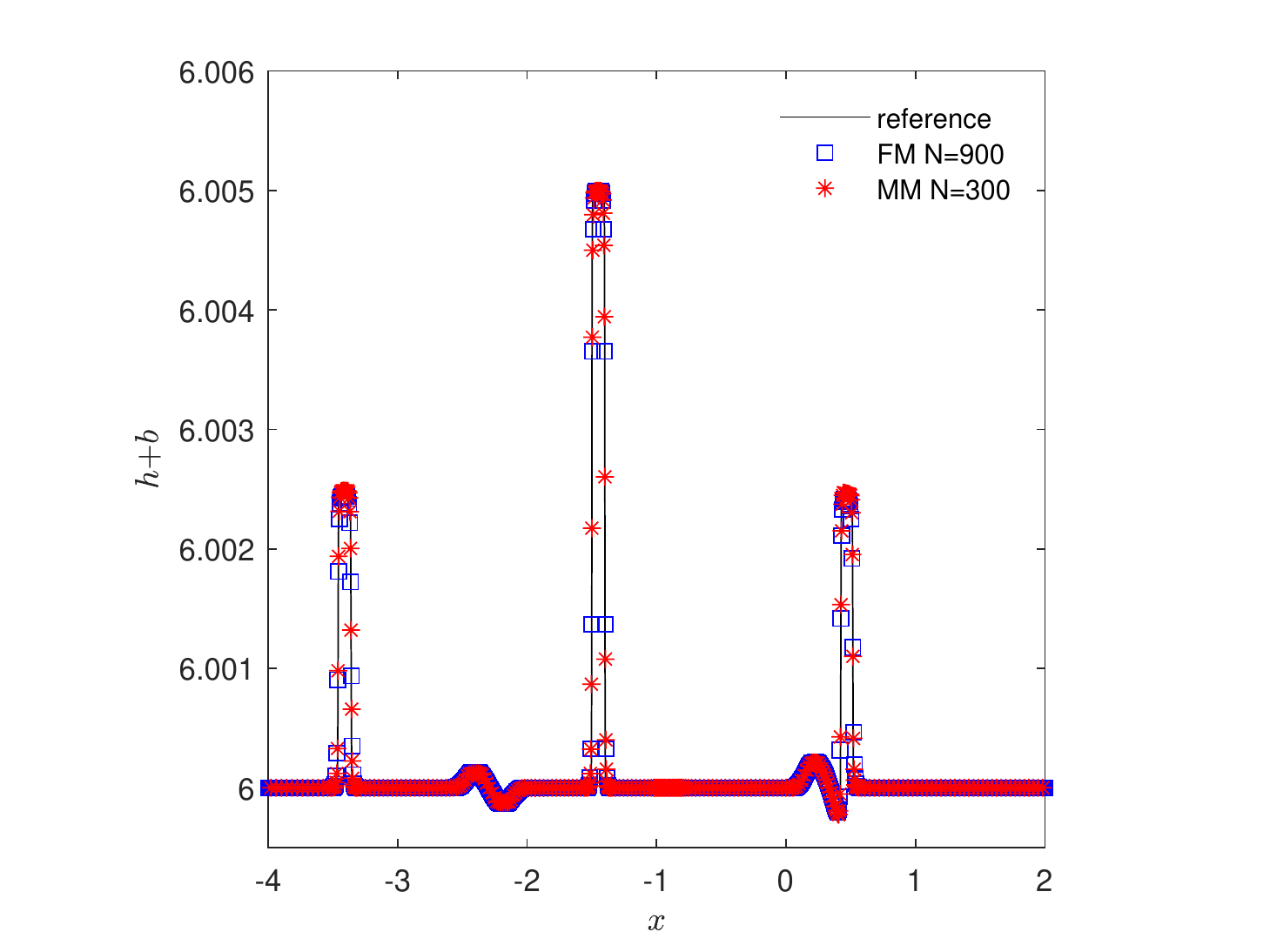}}
\subfigure[close view of (c)]{
\includegraphics[width=0.4\textwidth,trim=10 0 40 10,clip]
{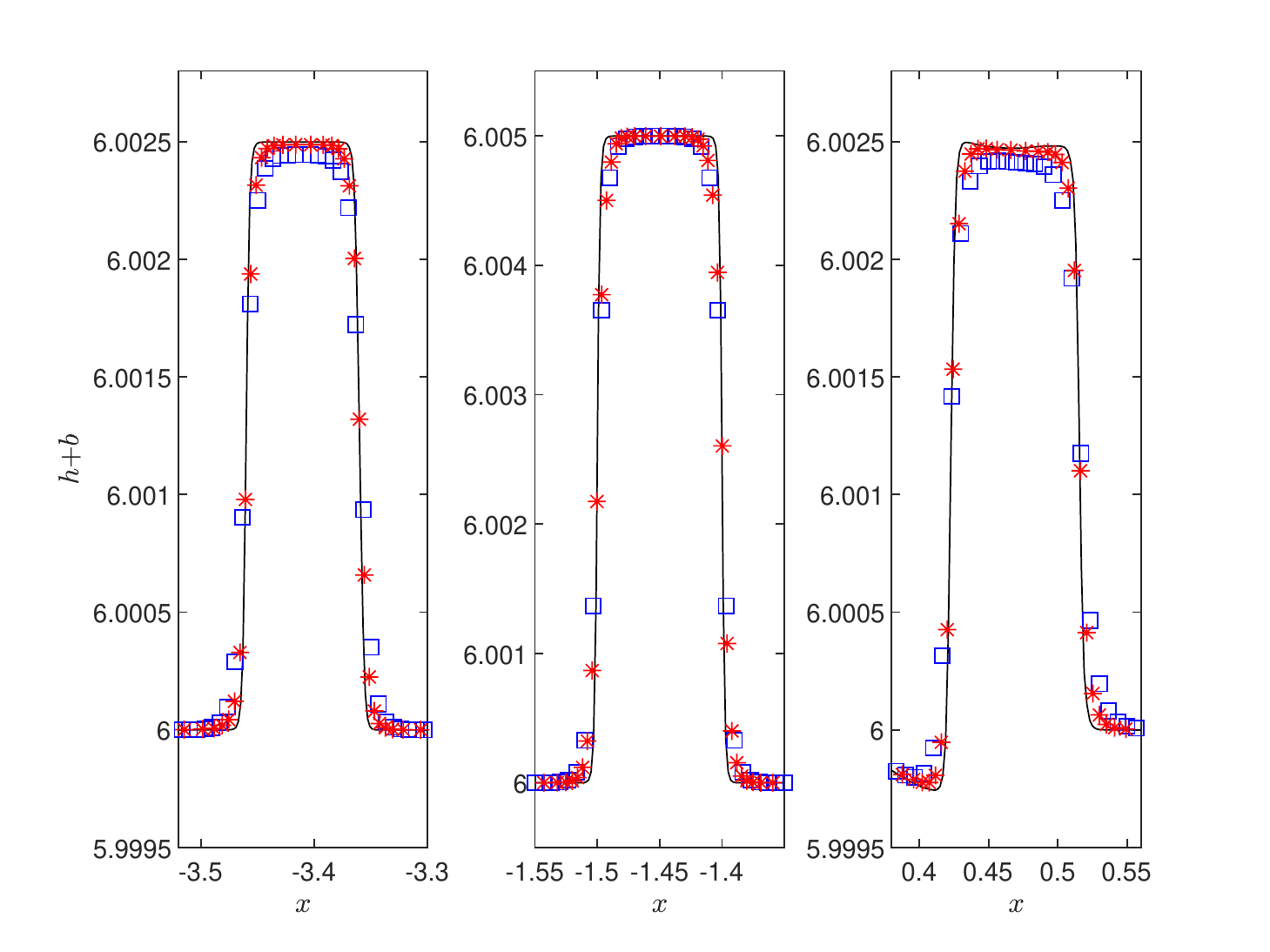}}
\caption{Example \ref{test3-1d} with initial data \eqref{Data-2}. The solution $h+b$ at $t = 0.4$ obtained with $P^2$-DG and a moving mesh of $N=300$ and fixed meshes of $N=300$ and $N=900$.}
\label{Fig:test3-1d-s1-H}
\end{figure}

\begin{figure}[H]
\centering
\subfigure[$hu$: MM 300 vs FM 300]{
\includegraphics[width=0.4\textwidth,trim=10 0 30 10,clip]
{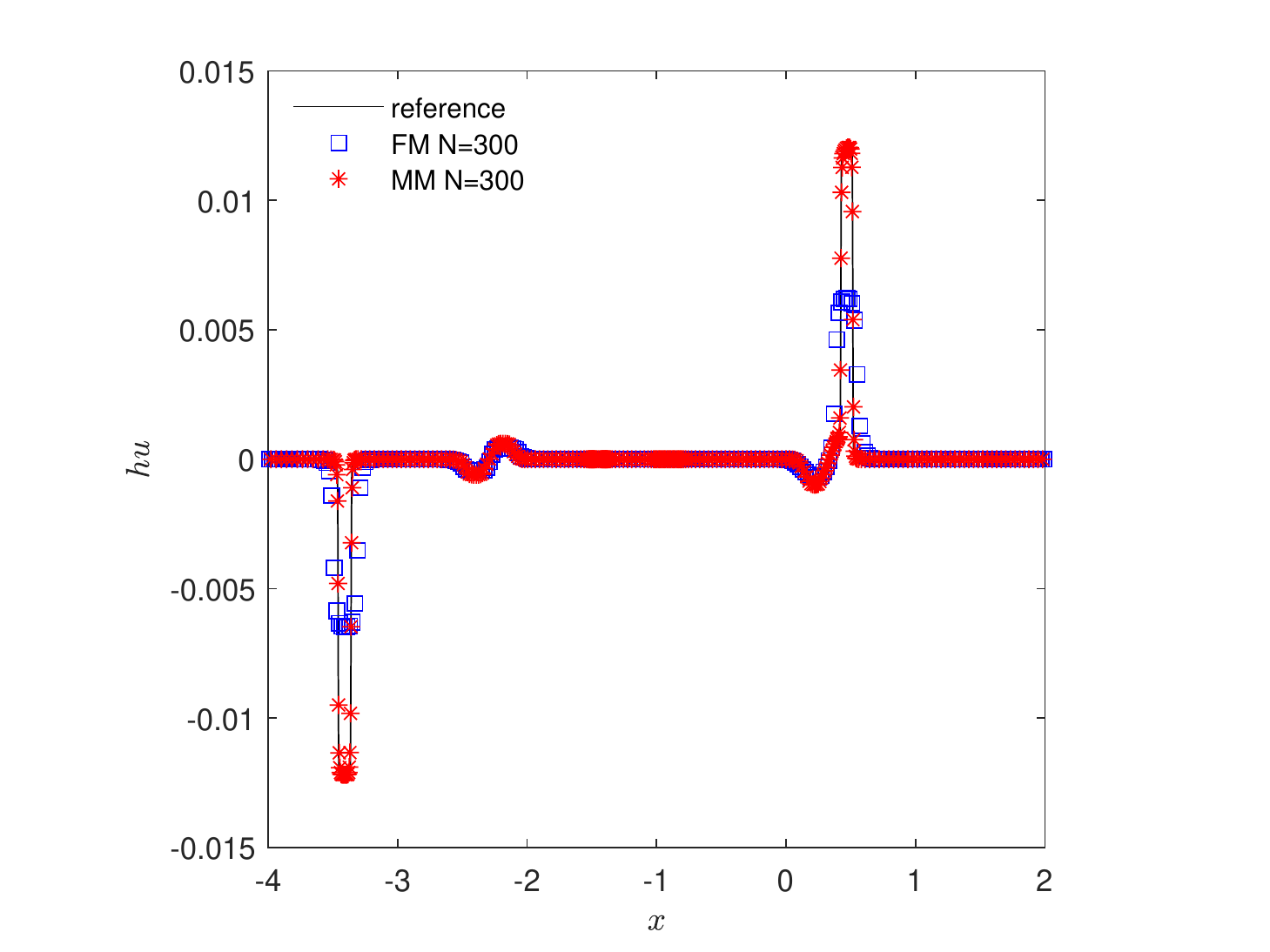}}
\subfigure[close view of (a)]{
\includegraphics[width=0.4\textwidth,trim=10 0 30 10,clip]
{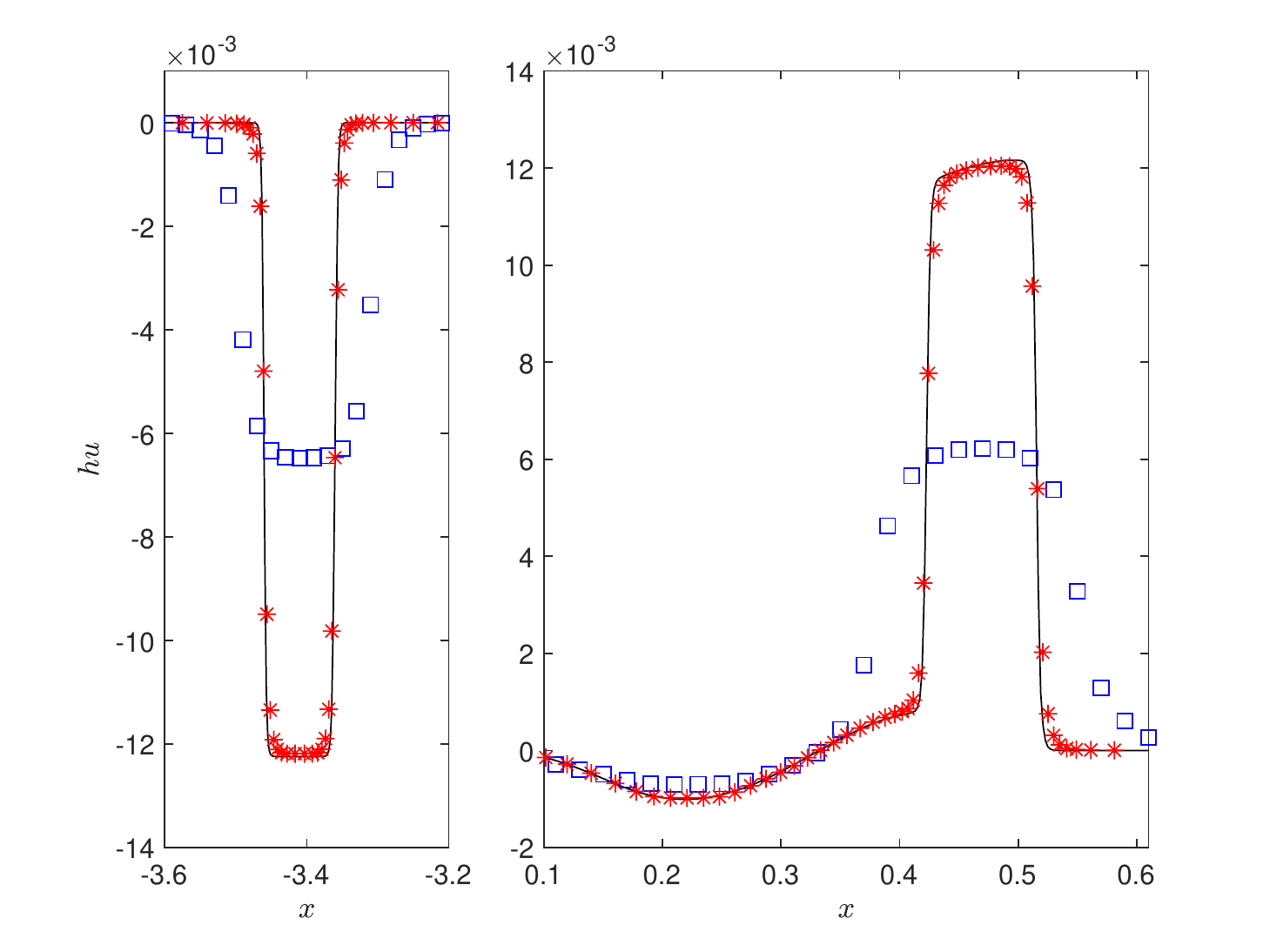}}
\subfigure[$hu$: MM 300 vs FM 900]{
\includegraphics[width=0.4\textwidth,trim=10 0 30 10,clip]
{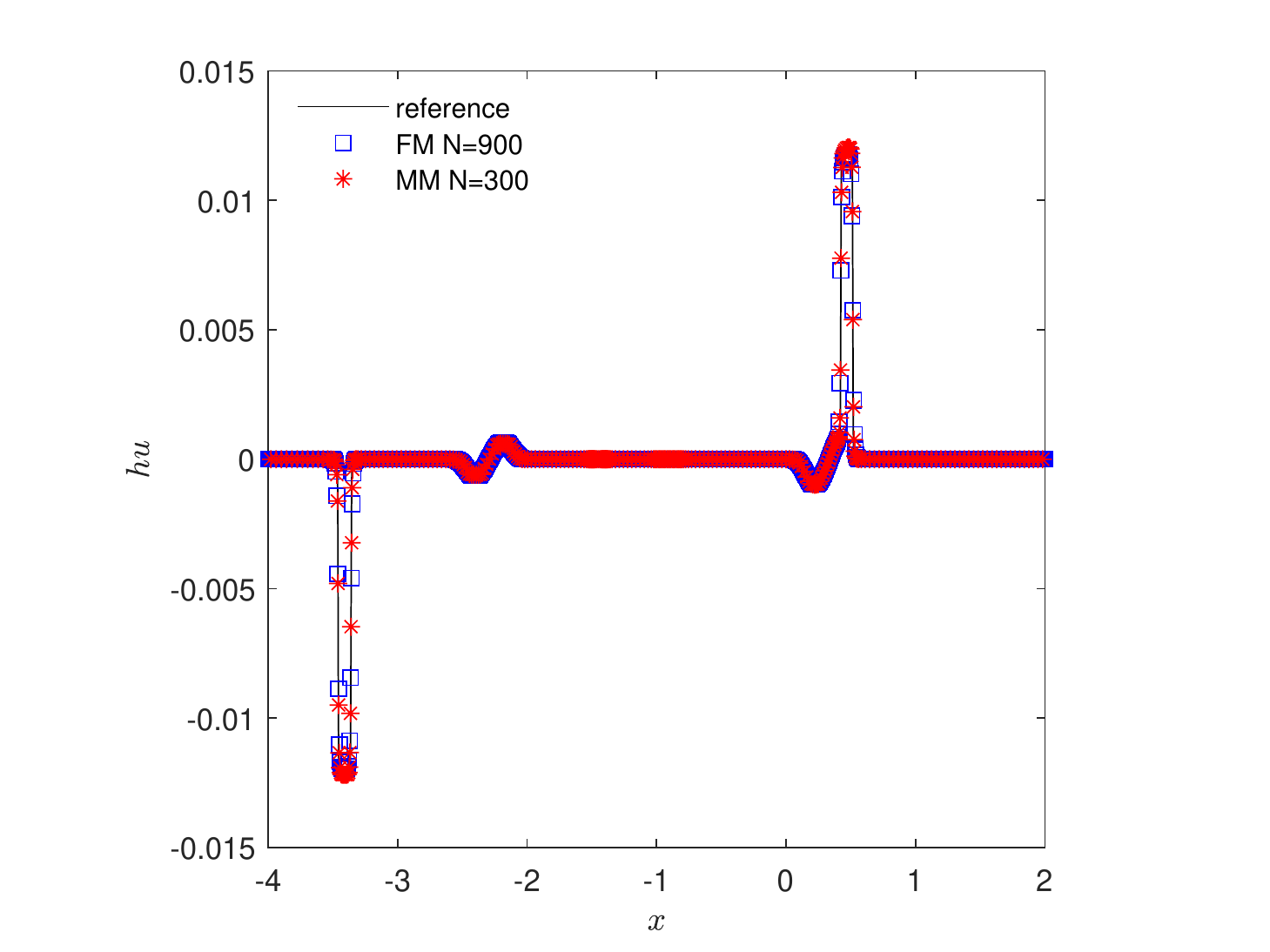}}
\subfigure[close view of (c)]{
\includegraphics[width=0.4\textwidth,trim=10 0 30 10,clip]
{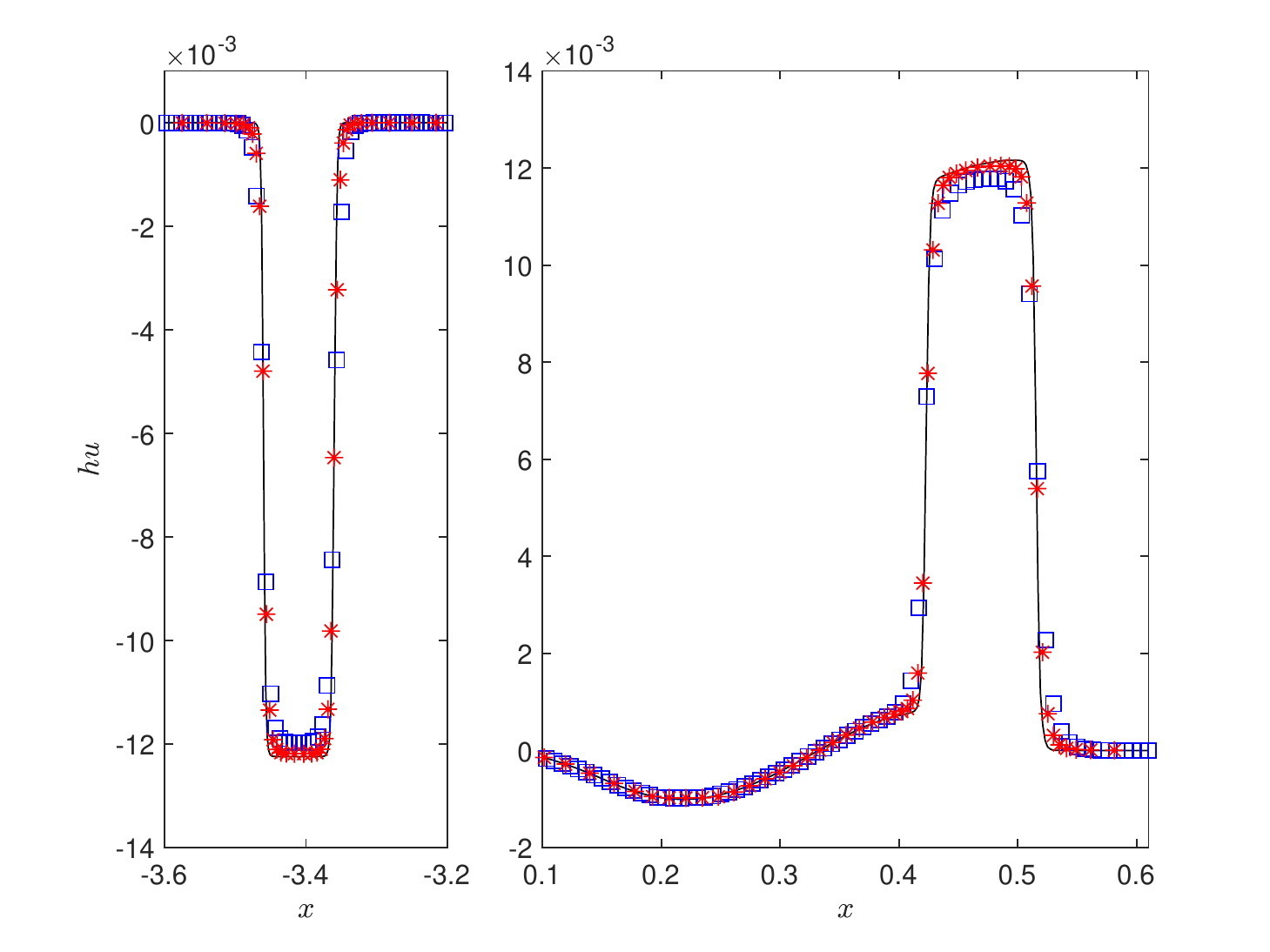}}
\caption{Example \ref{test3-1d} with initial data \eqref{Data-2}. The solution $hu$ at $t = 0.4$ obtained with $P^2$-DG and a moving mesh of $N=300$ and fixed meshes of $N=300$ and $N=900$.}
\label{Fig:test3-1d-s1-hu}
\end{figure}

\begin{figure}[H]
\centering
\subfigure[$h\theta$: MM 300 vs FM 300]{
\includegraphics[width=0.4\textwidth,trim=10 0 40 10,clip]
{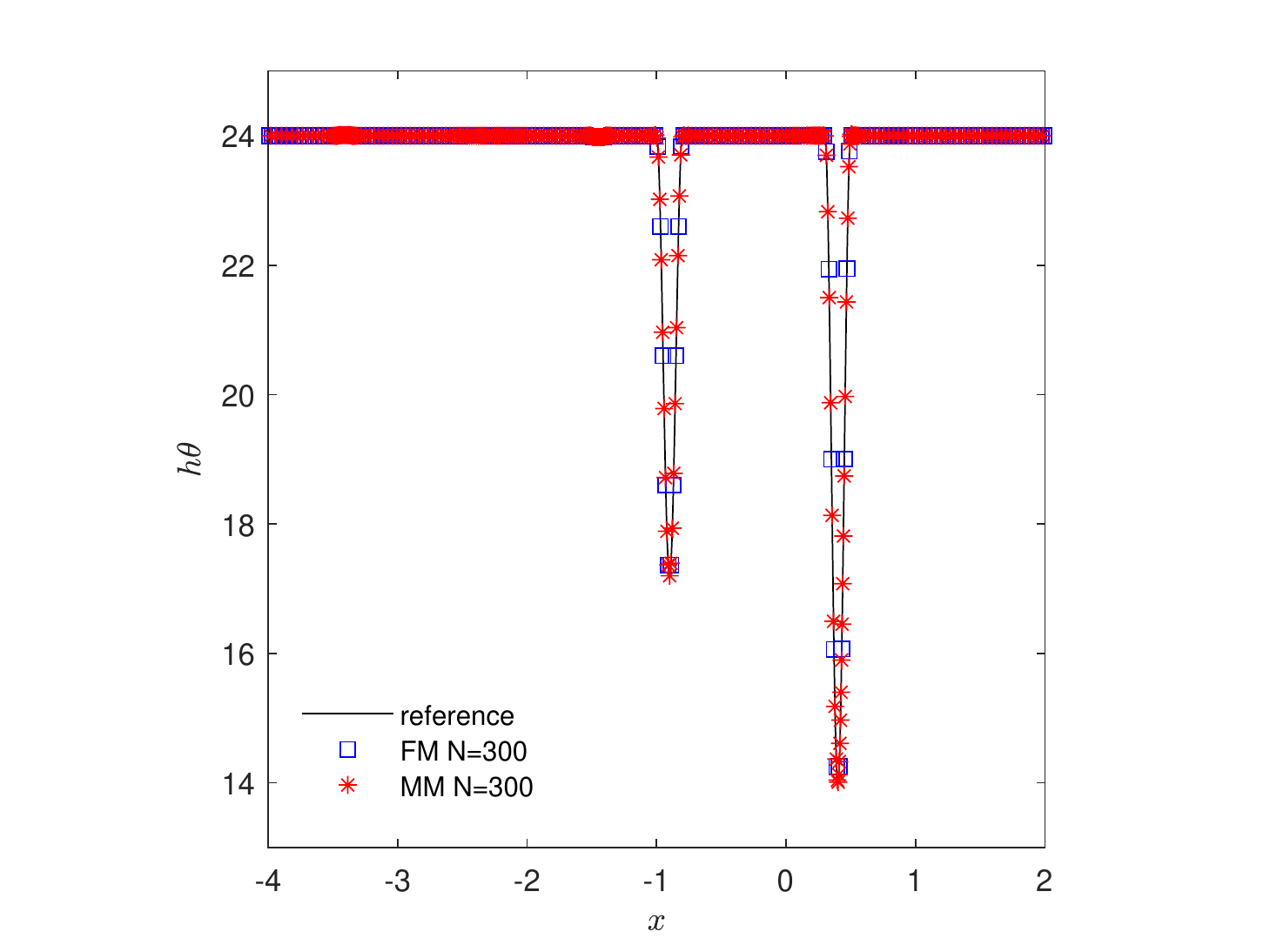}}
\subfigure[close view of (a)]{
\includegraphics[width=0.4\textwidth,trim=10 0 40 10,clip]
{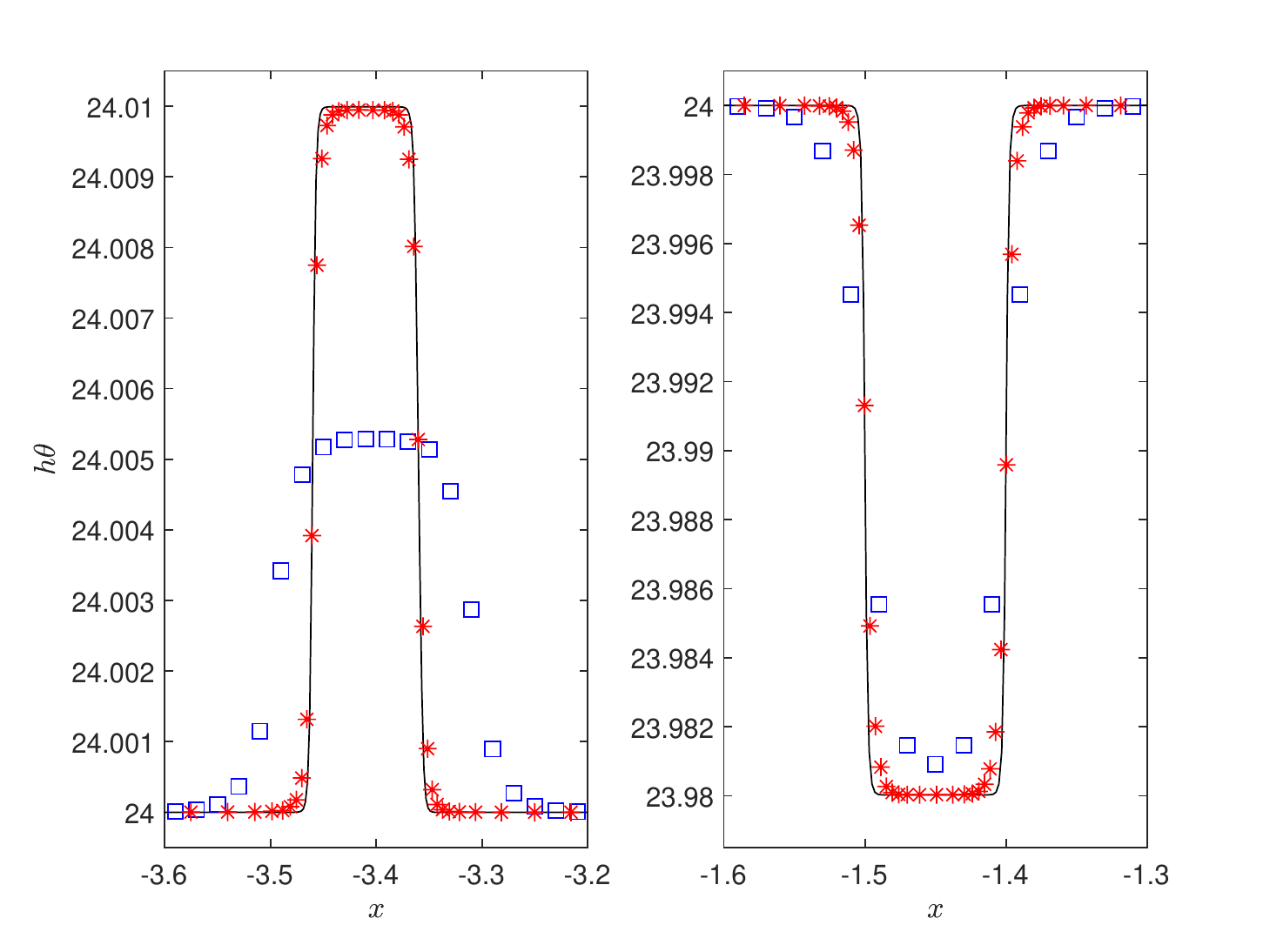}}
\subfigure[$h\theta$: MM 300 vs FM 900]{
\includegraphics[width=0.4\textwidth,trim=10 0 40 10,clip]
{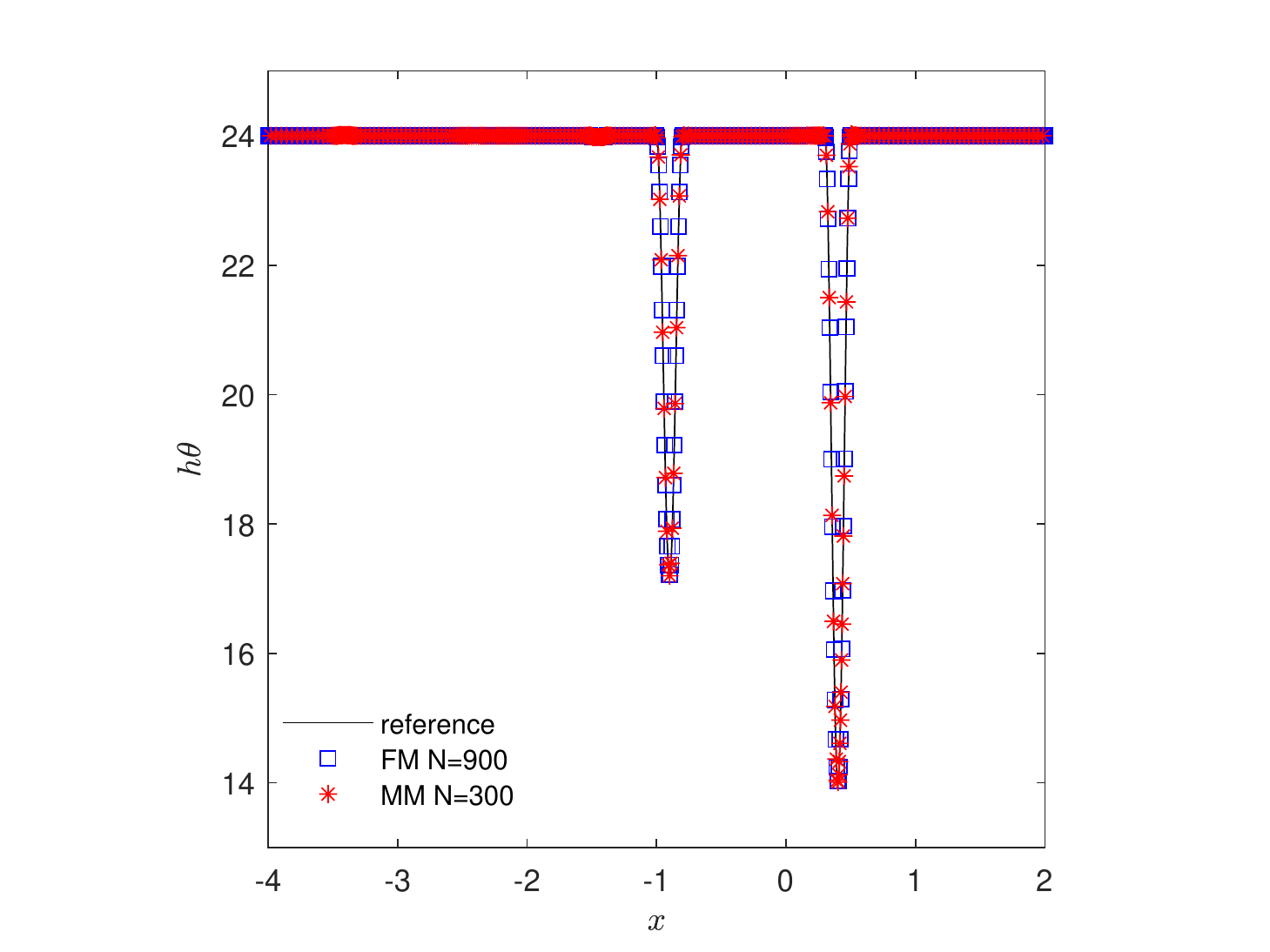}}
\subfigure[close view of (c)]{
\includegraphics[width=0.4\textwidth,trim=10 0 40 10,clip]
{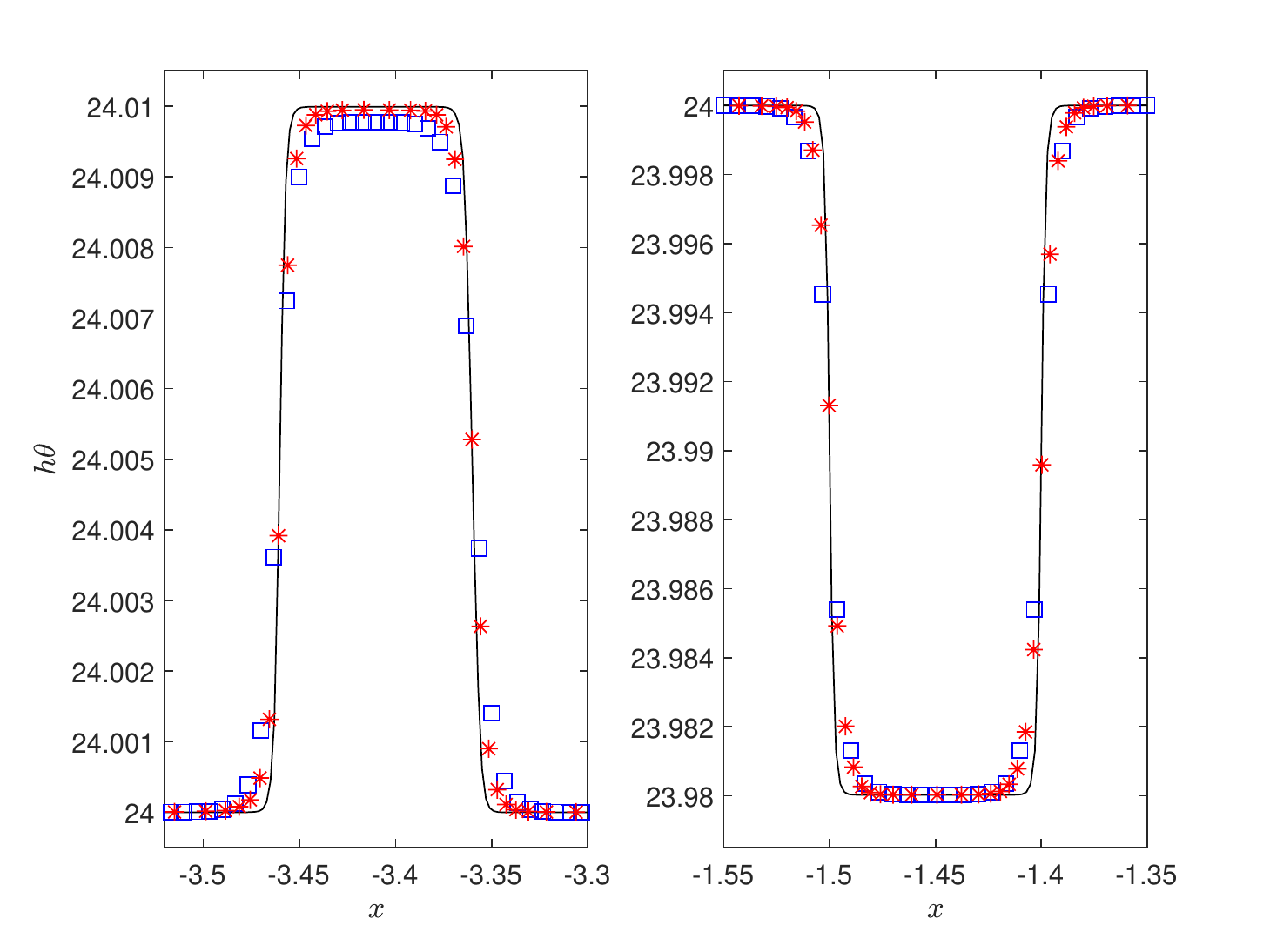}}
\caption{Example \ref{test3-1d} with initial data \eqref{Data-2}. The solution $h\theta$ at $t = 0.4$ obtained with $P^2$-DG and a moving mesh of $N=300$ and fixed meshes of $N=300$ and $N=900$.}
\label{Fig:test3-1d-s1-eta}
\end{figure}

\begin{example}\label{test4-1d}
(The dam break over the non-flat bottom for the 1D Ripa model.)
\end{example}
We choose this example to verify the ability of our well-balanced MM-DG scheme to simulate the dam break problem over a non-flat bottom. The computational domain is  $(-1,1)$.
The bottom topography and the initial conditions \cite{Britton-Xing-2020JSC} are given by
\begin{equation}\label{test4-1d-data}
\begin{split}
&b(x) = \begin{cases}
0.5(\cos(10\pi(x+0.3))+1),& \text{for } x \in (-0.4, -0.2) \\
0.75(\cos(10\pi(x-0.3))+1),& \text{for } x \in (0.2, 0.4) \\
0,& \text{otherwise}
\end{cases}
\\&
\big(h,u,\theta\big)(x,0)=
\begin{cases}
\big(5-b(x),~0,~3\big),& \text{for}~x < 0\\
\big(2-b(x),~0,~5\big),&\text{otherwise}.
\end{cases}
\end{split}
\end{equation}
We compute the solution up to $t=0.14$.

The mesh trajectories obtained with the $P^2$ MM-DG method with a moving mesh of $N=200$ are plotted
in Fig.~\ref{Fig:test4-1d-mesh-dam}.
The mesh has higher concentrations around the shock waves, contact discontinuity, and bottom bumps. It is clear that the mesh adaptation captures the shock waves before and after the split and the interaction of the shock waves with the bottom bumps.

The free water surface $h+b$ and the temperature $\theta$ and pressure $\frac{1}{2}gh^2\theta$
at $t = 0.14$ obtained with $P^2$-DG and a moving mesh of $N=200$
and fixed meshes of $N=200$ and $N=600$ are plotted in Figs.~\ref{Fig:test4-1d-dam-H} and
\ref{Fig:test4-1d-dam-P}, respectively.
The results show that the DG method with moving or fixed meshes
capture the discontinuity very well.
Moreover, the moving mesh solutions with $N=200$ are more accurate than those with
fixed meshes of $N=200$ and $N=600$.

\begin{figure}[H]
\centering
\includegraphics[width=0.4\textwidth,trim=0 0 20 10,clip]{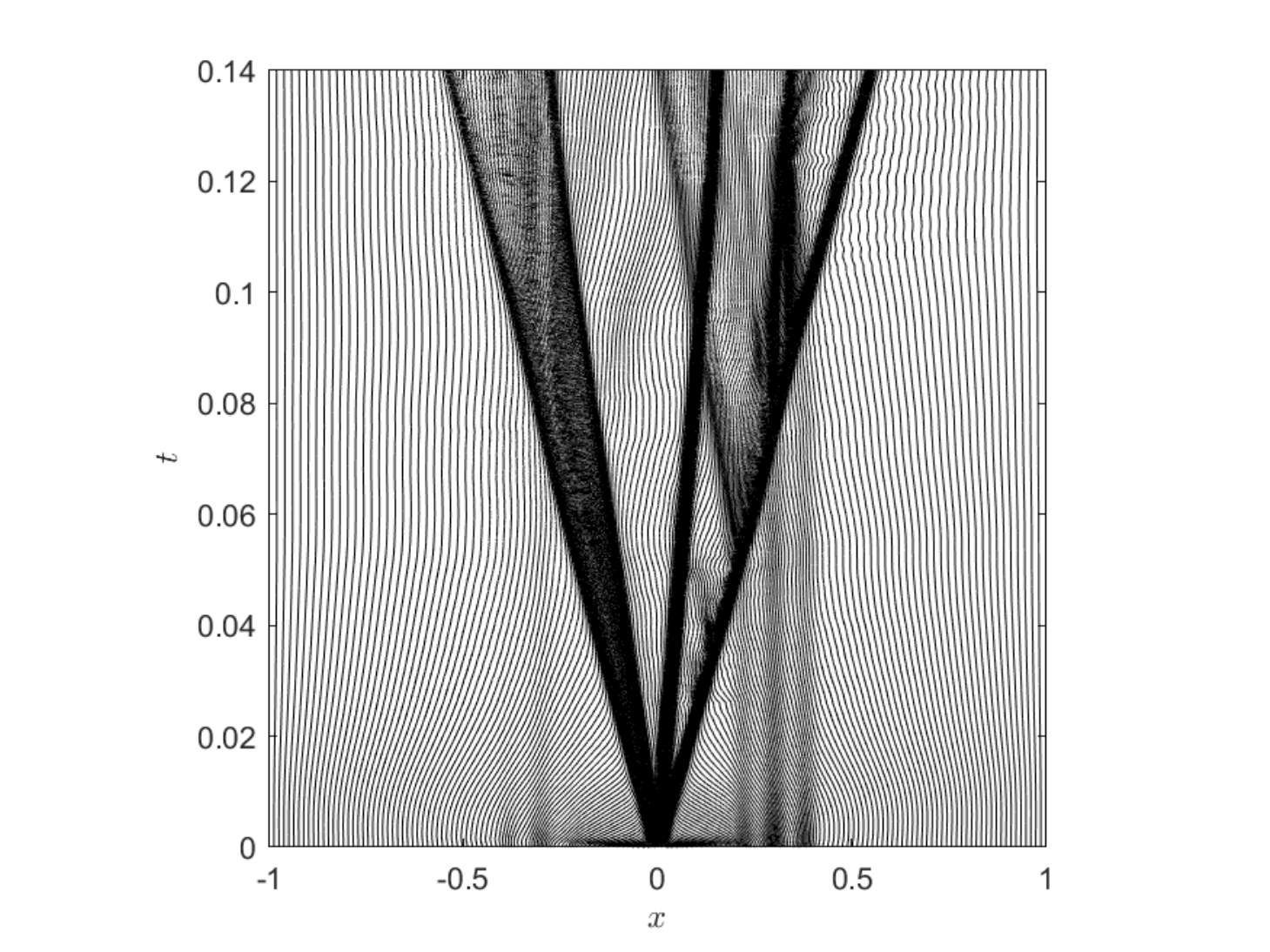}
\caption{Example \ref{test4-1d}. The mesh trajectories are obtained with $P^2$-DG of a moving mesh of $N=200$.}
\label{Fig:test4-1d-mesh-dam}
\end{figure}

\begin{figure}[H]
\centering
\subfigure[$h+b$: MM 200 vs FM 200]{
\includegraphics[width=0.4\textwidth,trim=10 0 40 10,clip]
{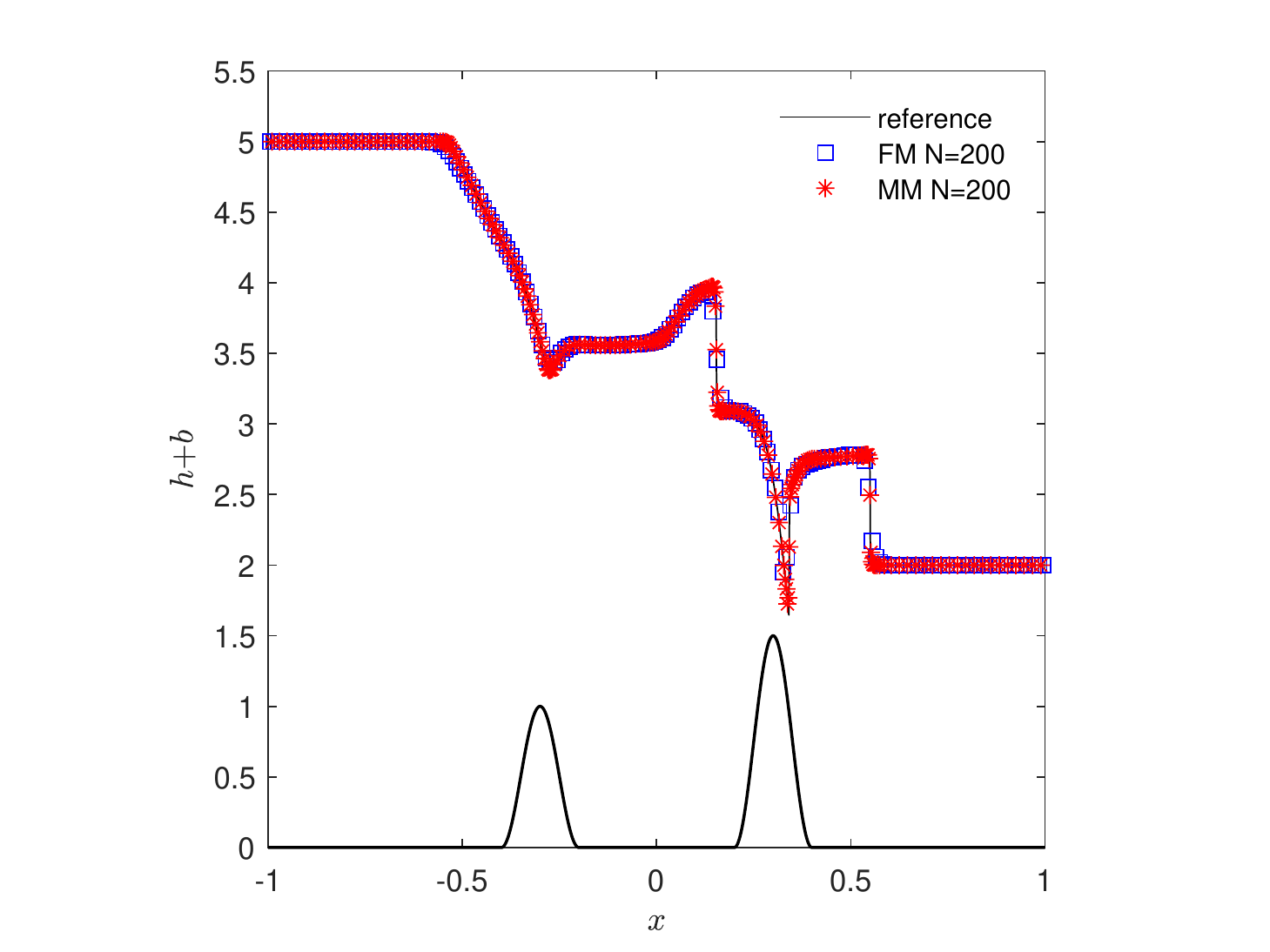}}
\subfigure[close view of (a)]{
\includegraphics[width=0.4\textwidth,trim=10 0 40 10,clip]
{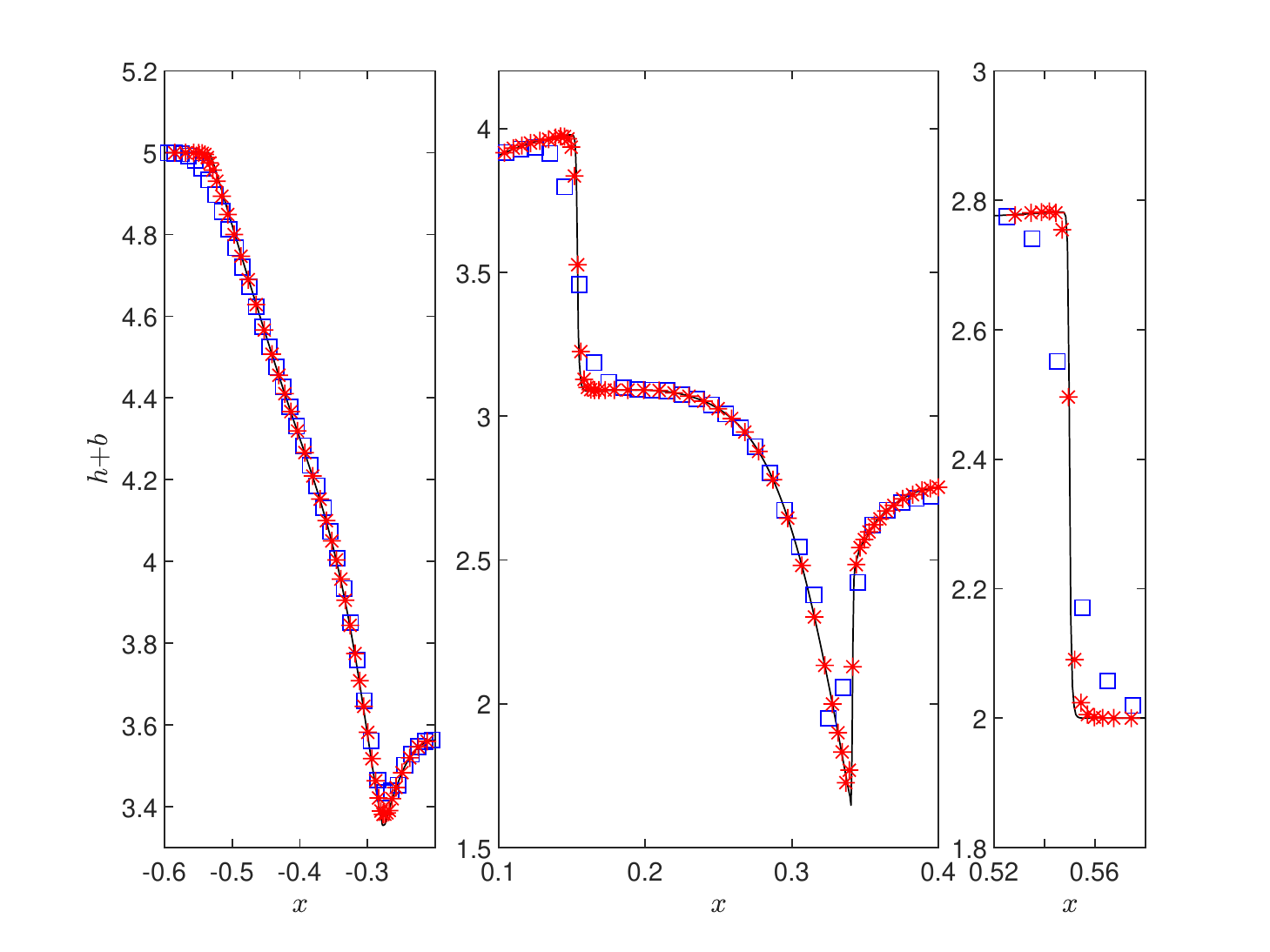}}
\subfigure[$h+b$: MM 200 vs FM 600]{
\includegraphics[width=0.4\textwidth,trim=10 0 40 10,clip]
{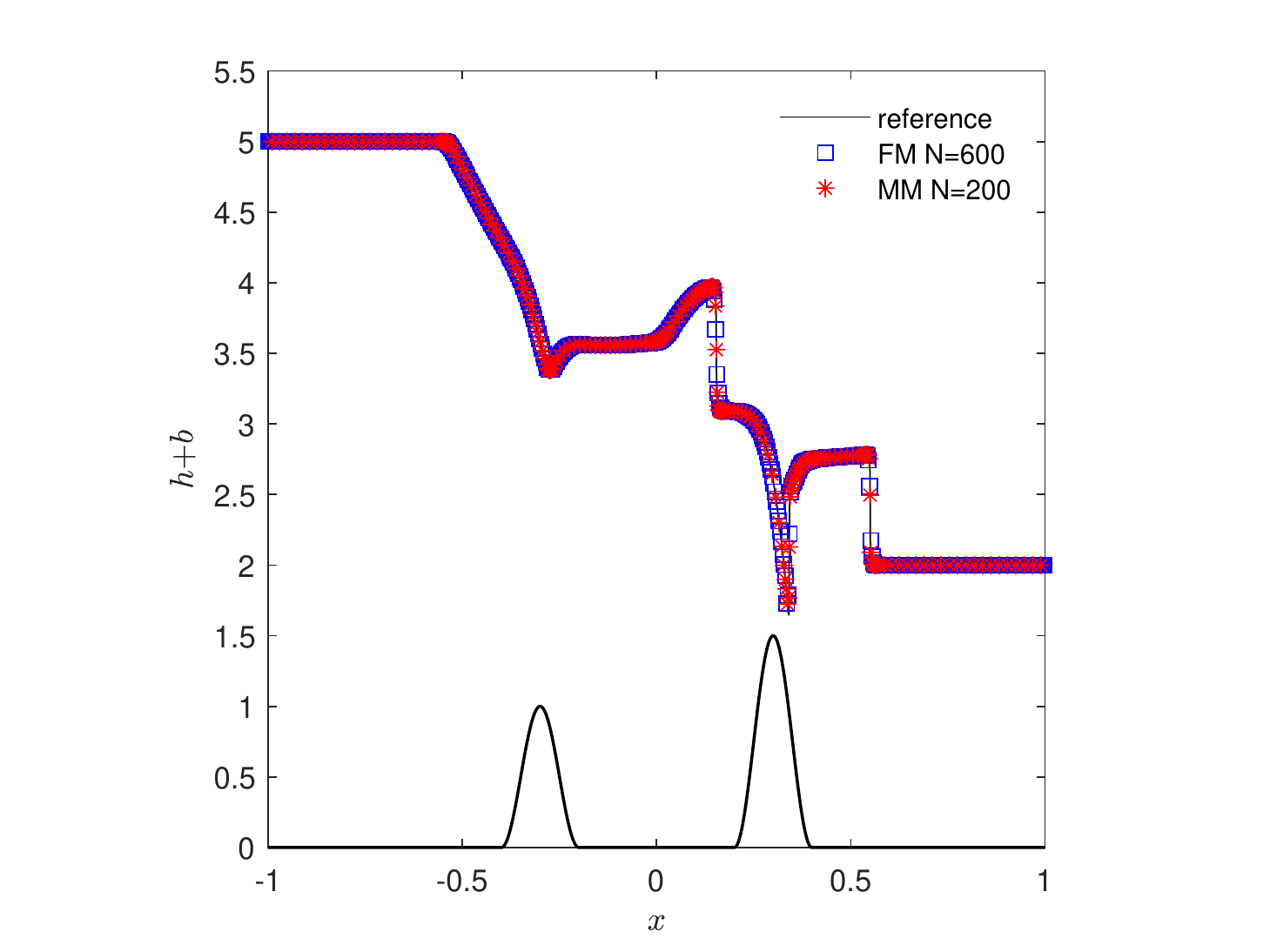}}
\subfigure[close view of (c)]{
\includegraphics[width=0.4\textwidth,trim=10 0 40 10,clip]
{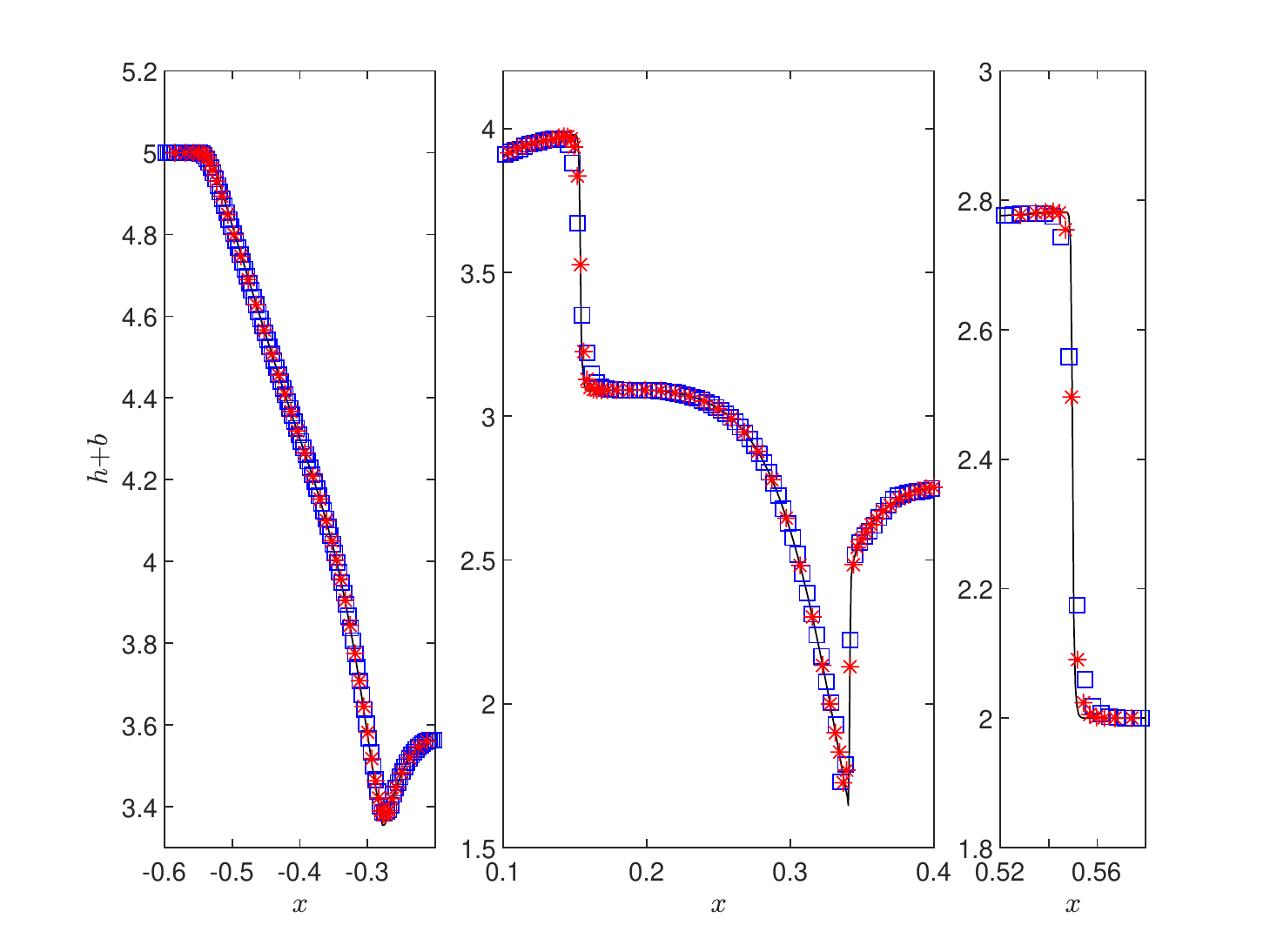}}
\caption{Example \ref{test4-1d}. The free water surface $h+b$ at $t = 0.14$ obtained with $P^2$-DG and a moving mesh of $N=200$ and fixed meshes of $N=200$ and $N=600$.}
\label{Fig:test4-1d-dam-H}
\end{figure}

\begin{figure}[H]
\centering
\subfigure[$\theta$: MM 200 vs FM 200]{
\includegraphics[width=0.4\textwidth,trim=10 0 40 10,clip]
{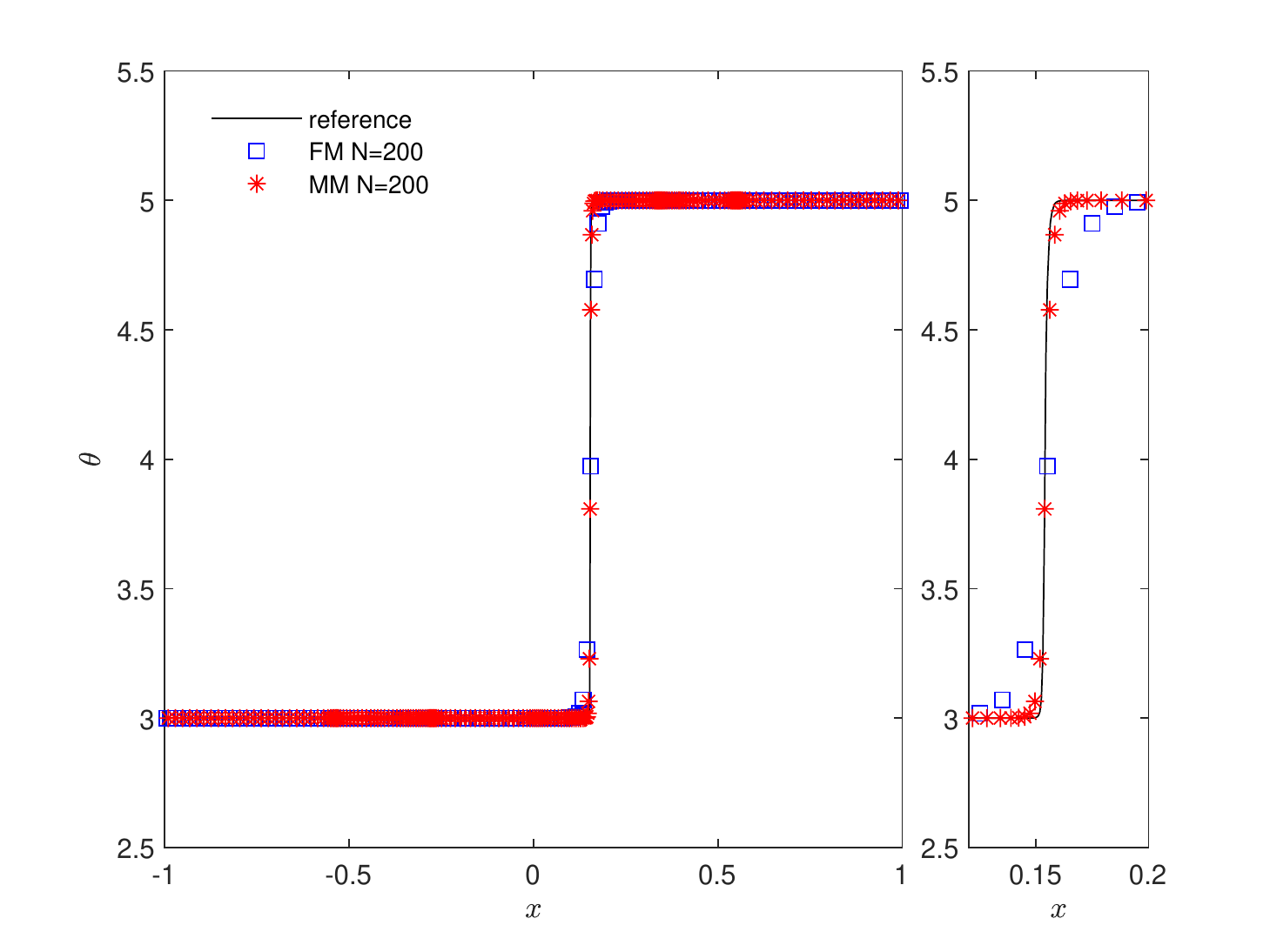}}
\subfigure[$\theta$: MM 200 vs FM 600]{
\includegraphics[width=0.4\textwidth,trim=10 0 40 10,clip]
{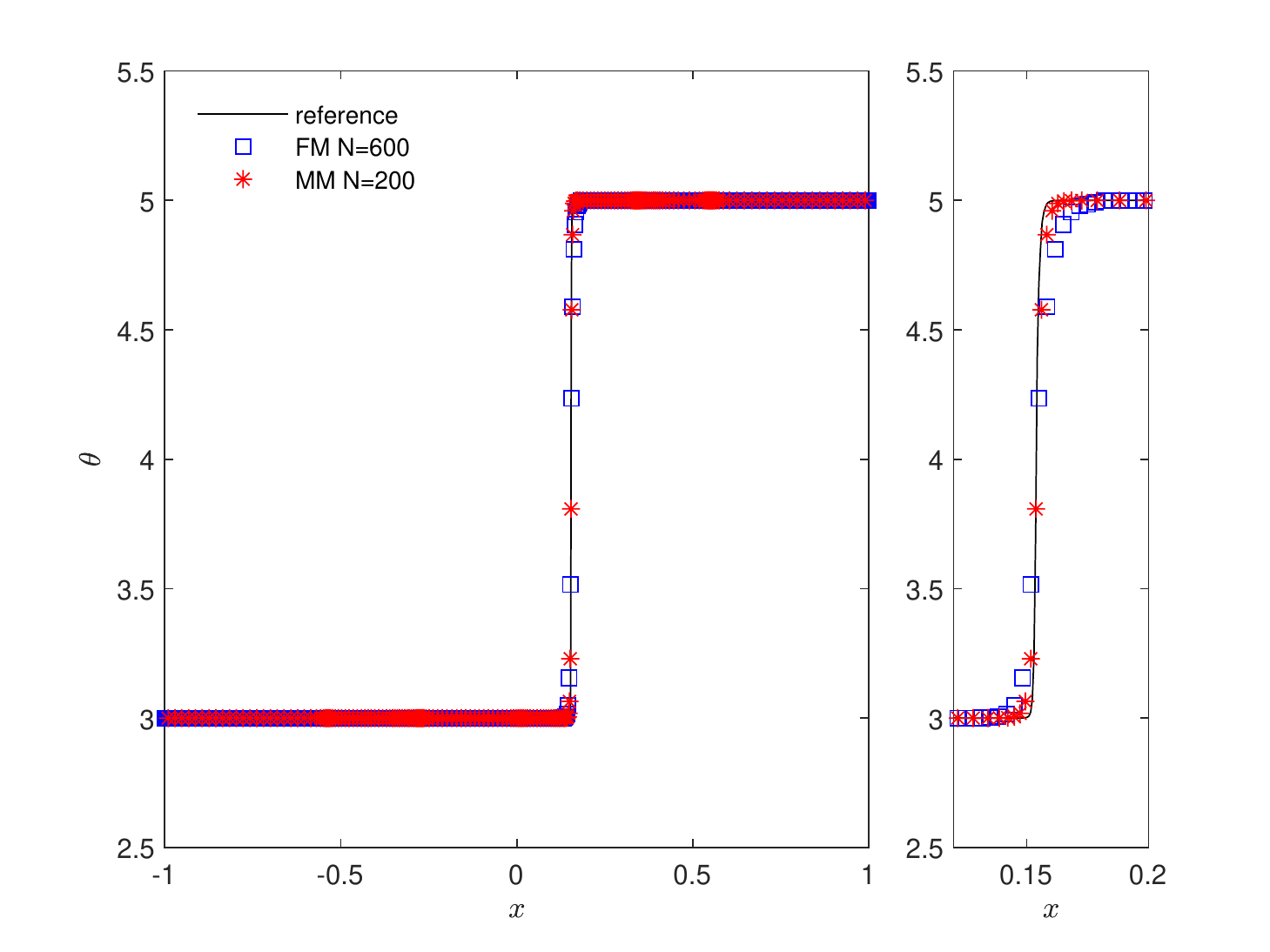}}
\subfigure[$\frac{1}{2}gh^2\theta$: MM 200 vs FM 200]{
\includegraphics[width=0.4\textwidth,trim=10 0 40 10,clip]
{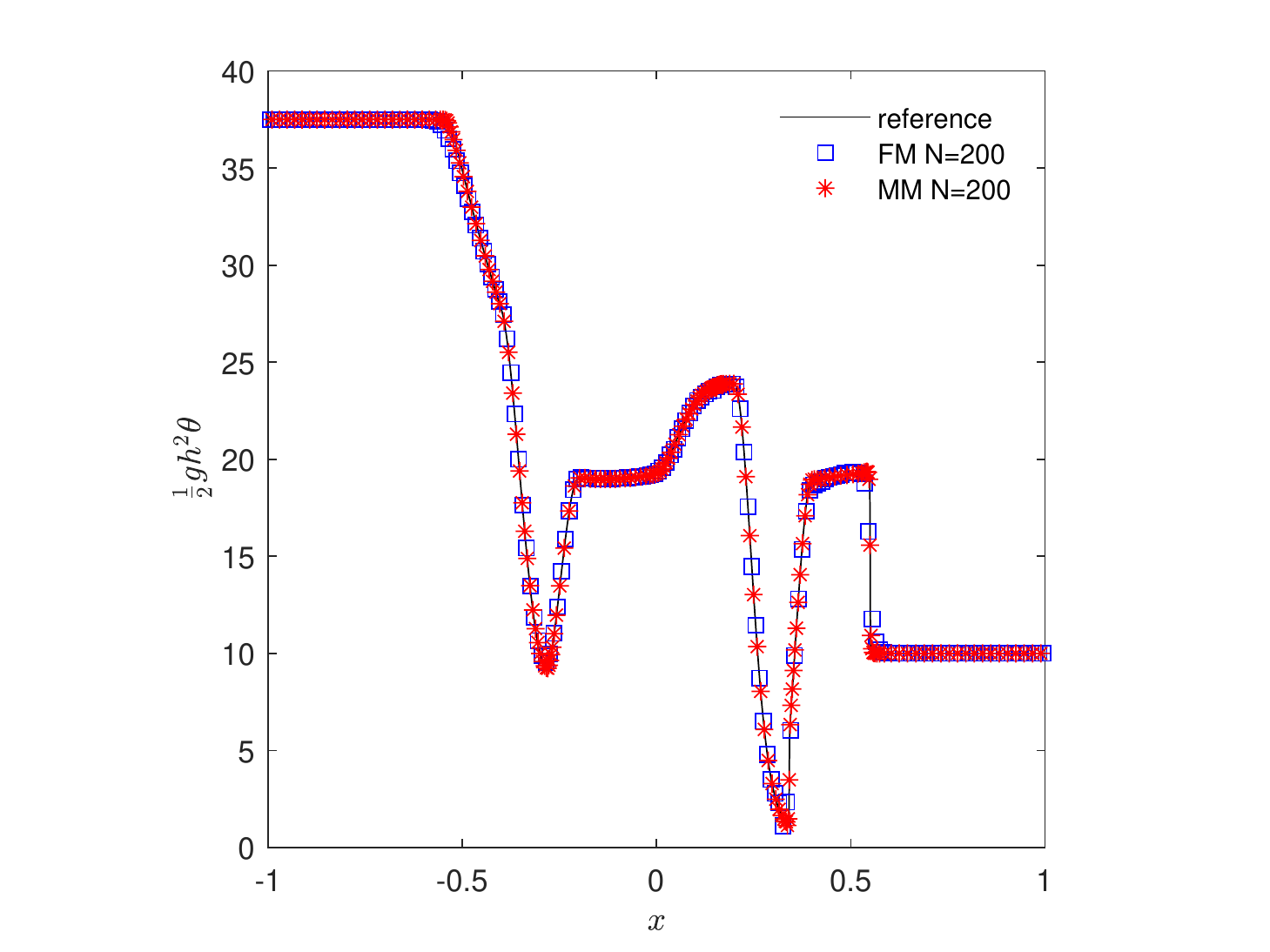}}
\subfigure[$\frac{1}{2}gh^2\theta$: MM 200 vs FM 600]{
\includegraphics[width=0.4\textwidth,trim=10 0 40 10,clip]
{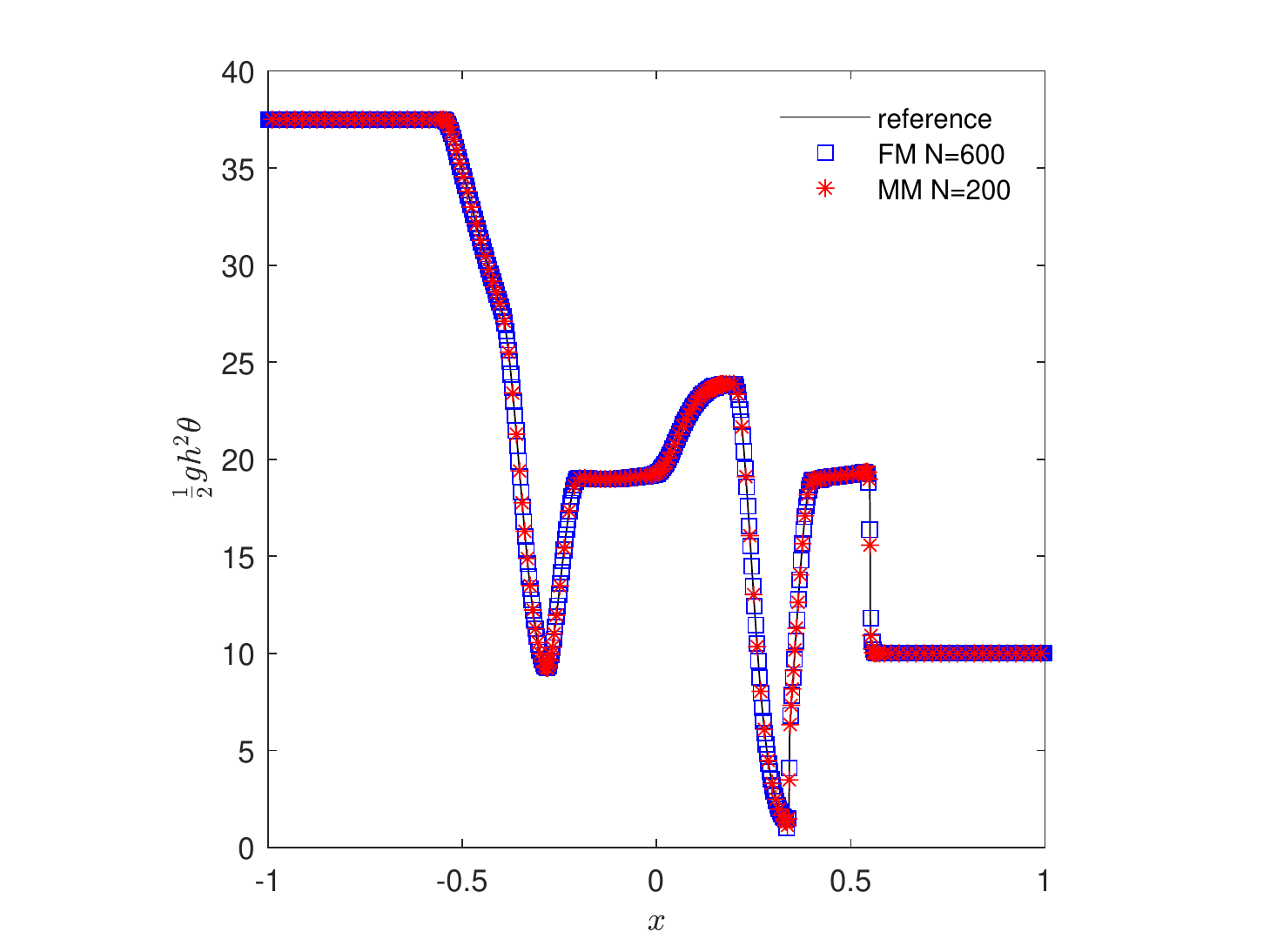}}
\caption{Example \ref{test4-1d}. The temperature $\theta$ and pressure $\frac{1}{2}gh^2\theta$ at $t = 0.14$ obtained with $P^2$-DG and a moving mesh of $N=200$ and fixed meshes of $N=200$ and $N=600$.}
\label{Fig:test4-1d-dam-P}
\end{figure}

\begin{example}\label{test0-1d}
(The dam break over the non-flat bottom with a dry region for the 1D Ripa model.)
\end{example}
We choose this example to verify the ability of our well-balanced MM-DG scheme to preserve the positivity of the water depth and temperature. The computational domain is  $(-1, 1)$.
Similar examples have been used in \cite{Chertock-etal-2014,Saleem-etal-2018KFVX,Sanchez-etal-2016HLLC}.
The initial data and the bottom topography for this example are given by
\begin{equation}\label{test0-1d-U0}
\begin{split}
&b(x) = \begin{cases}
2.0(\cos(10\pi(x+0.3))+1),& \text{for } x \in (-0.4, -0.2) \\
0.5(\cos(10\pi(x-0.3))+1),& \text{for } x \in (0.2, 0.4) \\
0,& \text{otherwise}
\end{cases}\\
&\big(h,u,\theta\big)(x,0)=
\begin{cases}
\big(5-b(x),~0,~1\big),& \text{for}~x < 0\\
\big(1-b(x),~0,~5\big),&\text{otherwise}
\end{cases}
\end{split}
\end{equation}
which contain a dry region near $x =0.3$. To ensure positivity preservation \cite{Xing-Zhang-Shu-2010ppSWEs} we take a smaller CFL number 0.15.  The solution is computed up to $t=0.3$.

The mesh trajectories obtained with the $P^2$ MM-DG method with a moving mesh of $N=200$ are plotted
in Fig.~\ref{Fig:test0-1d-mesh-pp}.
The mesh has higher concentrations around the shock waves, contact discontinuity, and bottom bumps, which shows that the adaptation captures the shock waves before and after the split and the interaction of the shock waves with the bottom bumps.

The free water surface $h+b$ and the temperature $\theta$ and pressure $\frac{1}{2}gh^2\theta$
at $t = 0.3$ obtained with $P^2$-DG and a moving mesh of $N=200$
and fixed meshes of $N=200$ and $N=600$ are plotted in Figs.~\ref{Fig:test0-1d-H-pp} and \ref{Fig:test0-1d-P-pp}, respectively.
The results show that the DG method with moving or fixed meshes
capture the discontinuity very well.
Moreover, the moving mesh solutions with $N=200$ are more accurate than those with
fixed meshes of $N=200$ and $N=600$. The results also demonstrate that
the scheme can preserve the positivity of the water depth $h$ and temperature $\theta$.
\begin{figure}[H]
\centering
\includegraphics[width=0.4\textwidth,trim=0 0 20 10,clip]{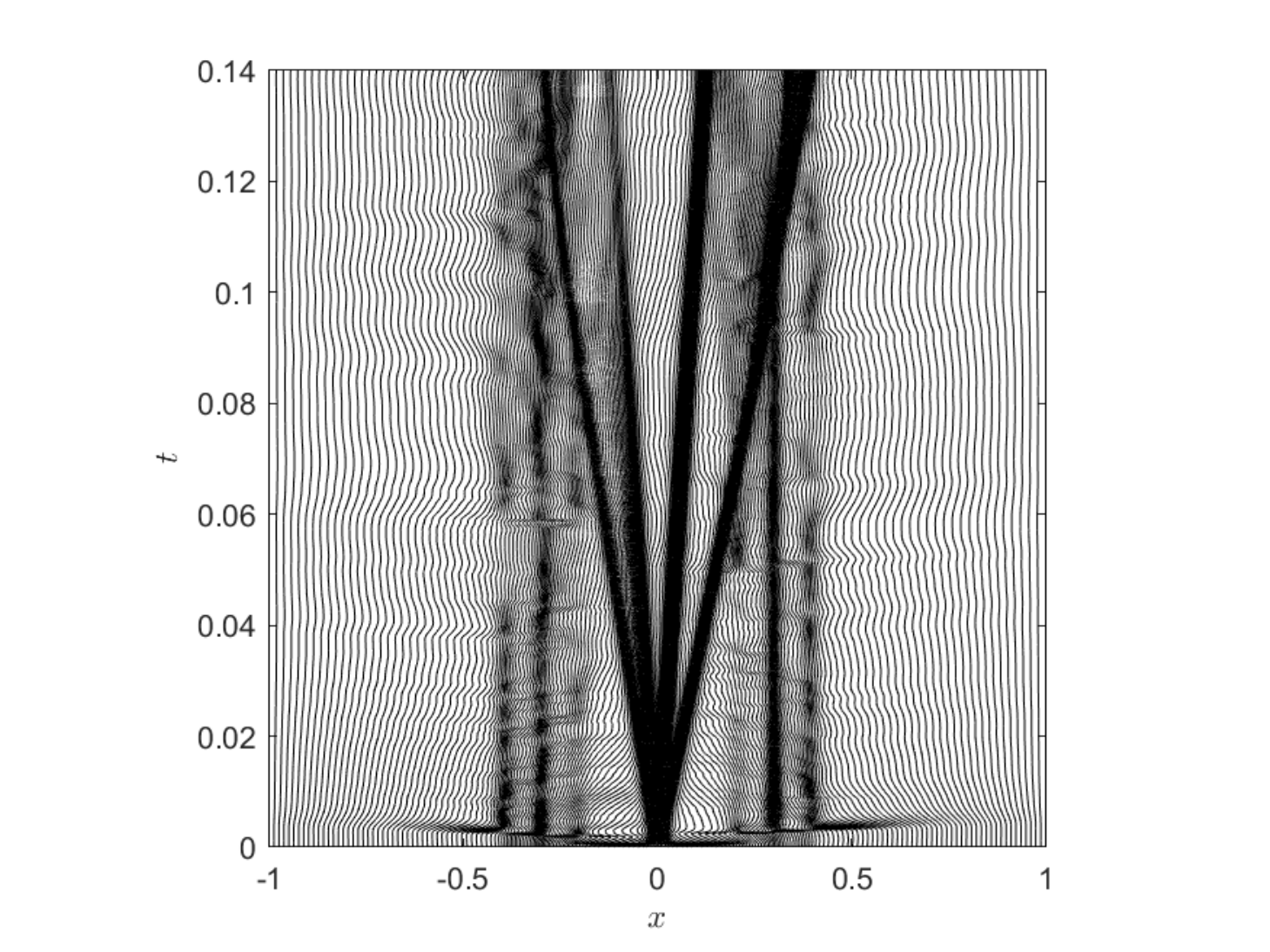}
\caption{Example \ref{test0-1d}. The mesh trajectories are obtained with $P^2$-DG of a moving mesh of $N=200$.}
\label{Fig:test0-1d-mesh-pp}
\end{figure}

\begin{figure}[H]
\centering
\subfigure[$h+b$: MM 200 vs FM 200]{
\includegraphics[width=0.4\textwidth,trim=10 0 40 10,clip]
{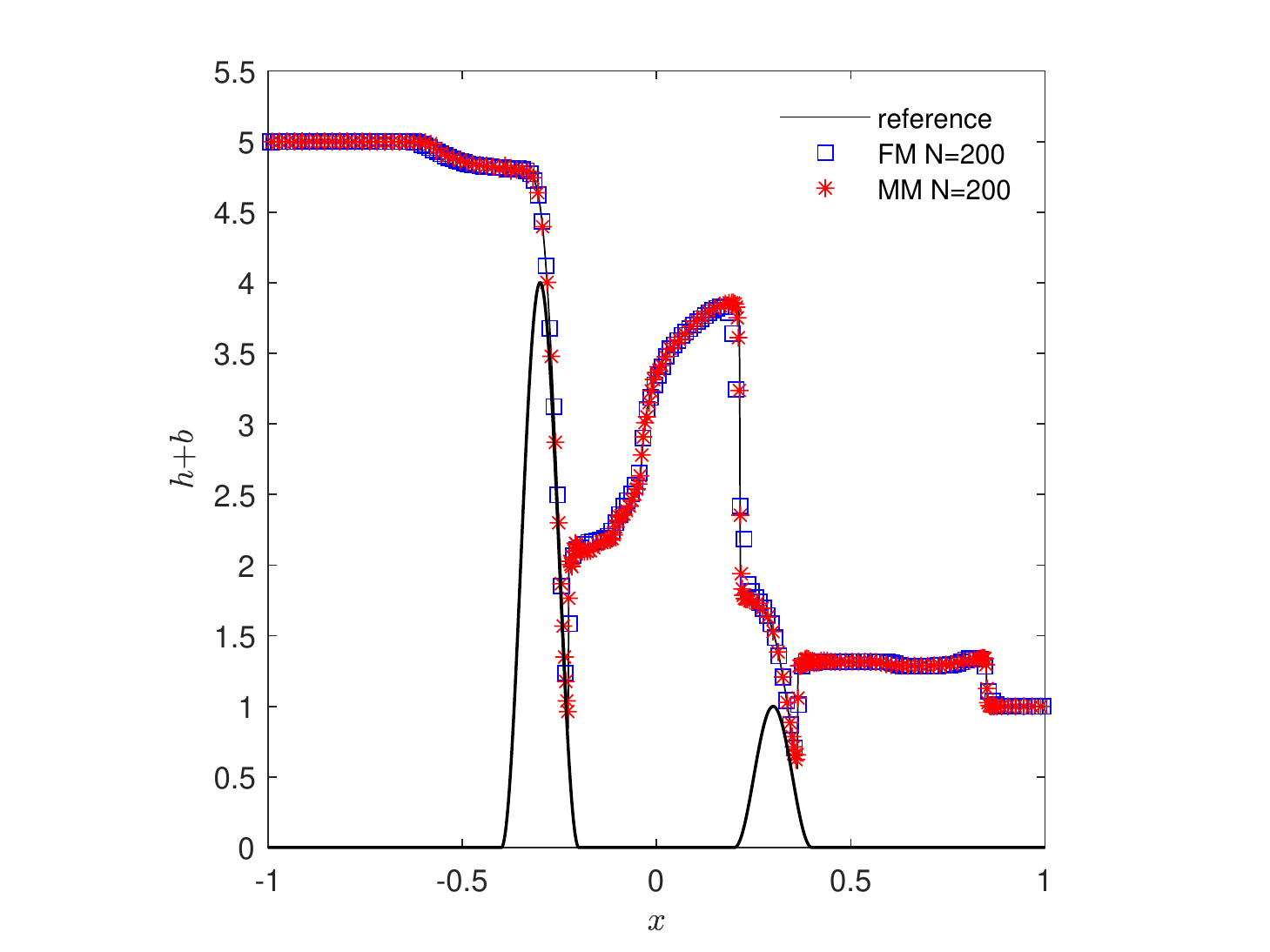}}
\subfigure[close view of (a)]{
\includegraphics[width=0.4\textwidth,trim=10 0 40 10,clip]
{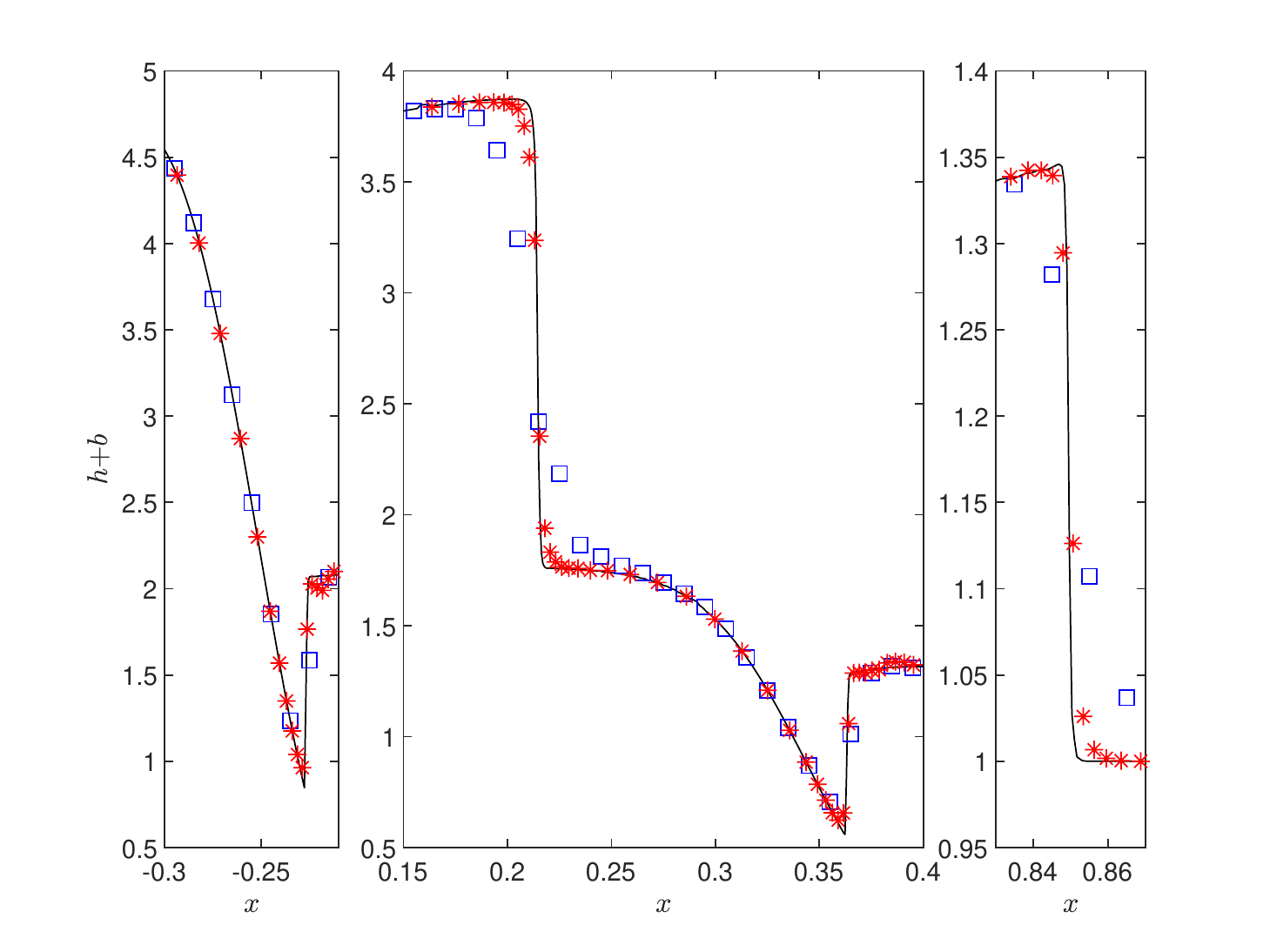}}
\subfigure[$h+b$: MM 200 vs FM 600]{
\includegraphics[width=0.4\textwidth,trim=10 0 40 10,clip]
{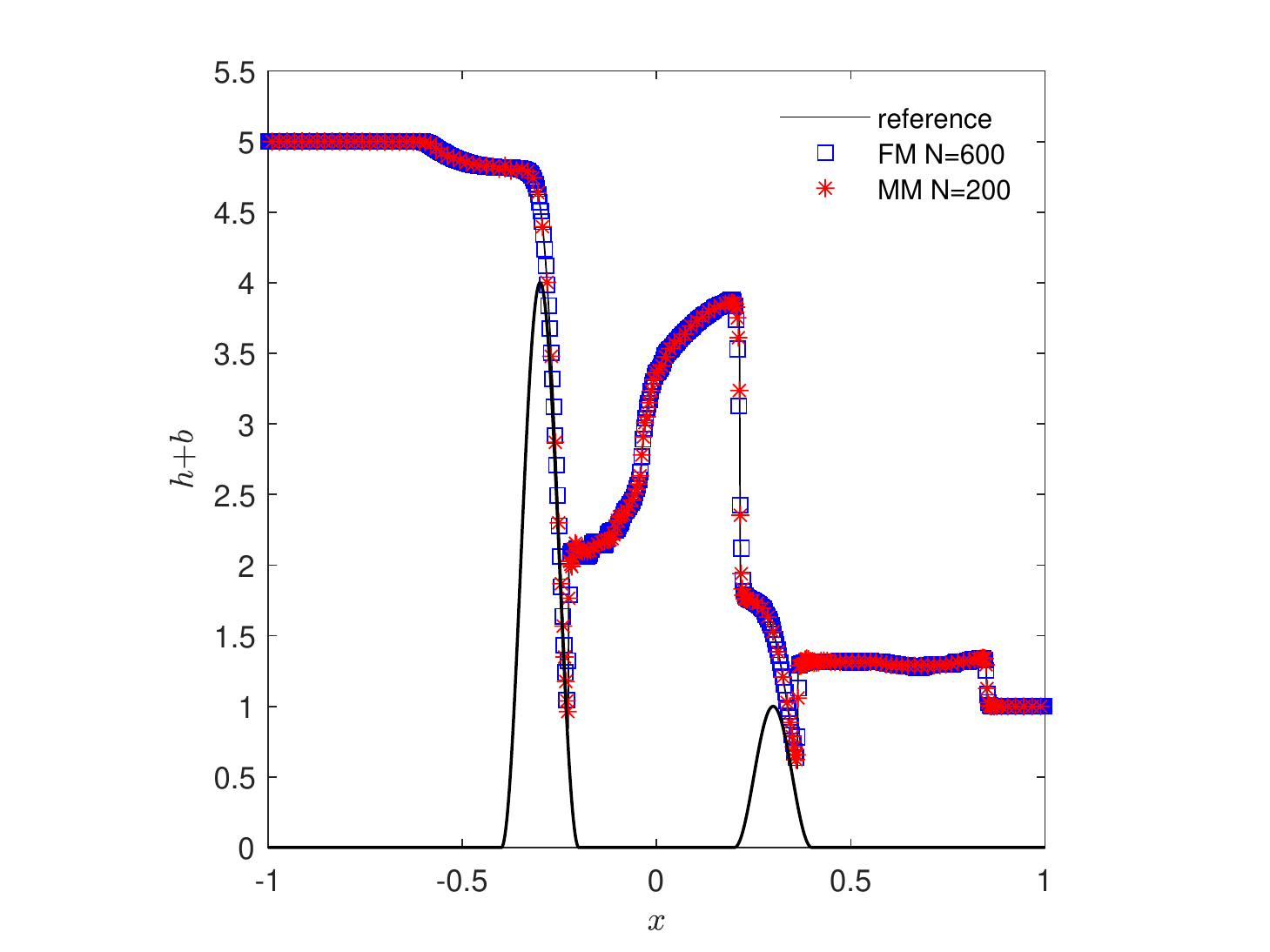}}
\subfigure[close view of (c)]{
\includegraphics[width=0.4\textwidth,trim=10 0 40 10,clip]
{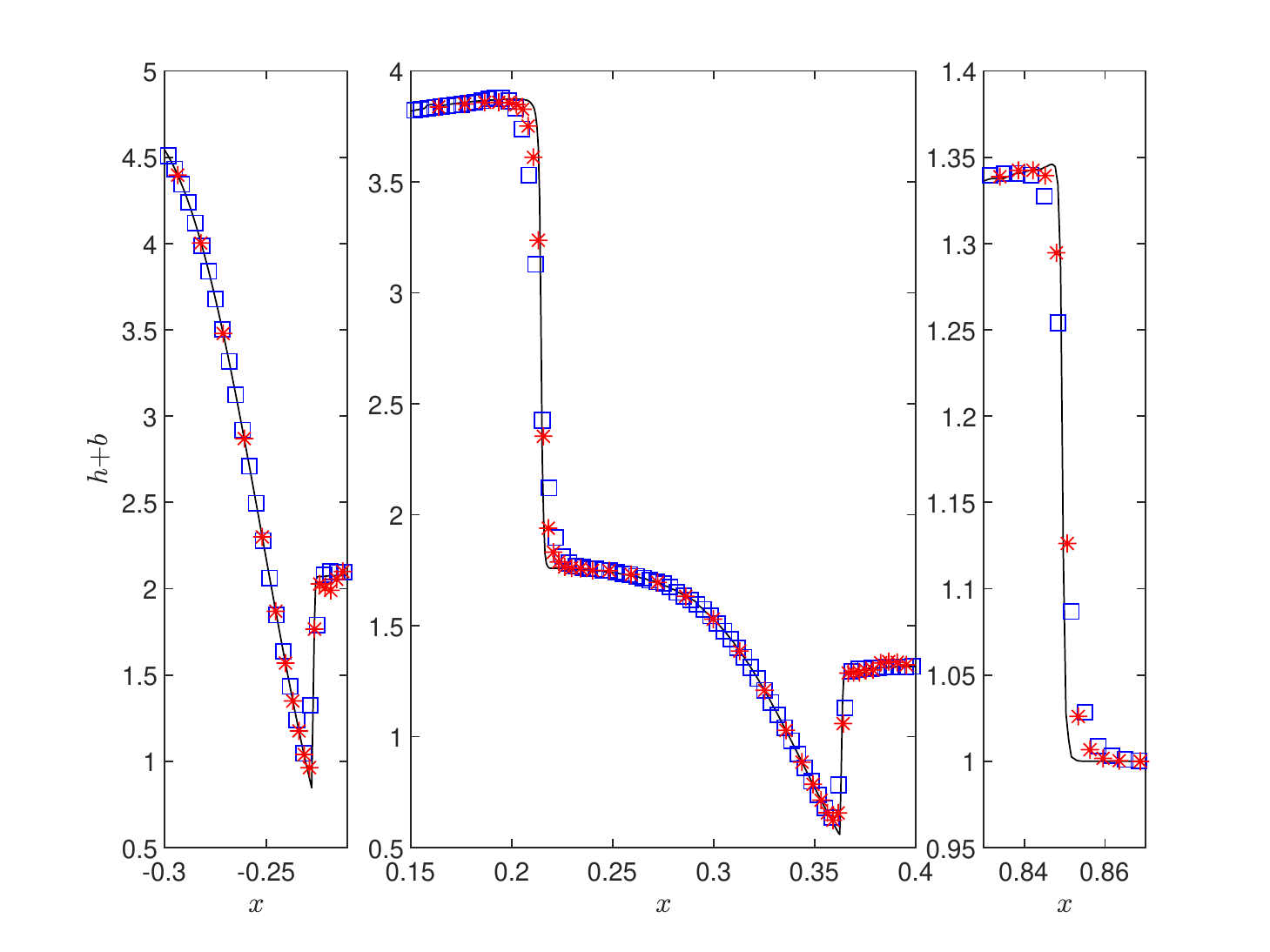}}
\caption{Example \ref{test0-1d}. The free water surface $h+b$ at $t = 0.3$ obtained with $P^2$-DG and a moving mesh of $N=200$ and fixed meshes of $N=200$ and $N=600$.}
\label{Fig:test0-1d-H-pp}
\end{figure}

\begin{figure}[H]
\centering
\subfigure[$\theta$: MM 200 vs FM 200]{
\includegraphics[width=0.4\textwidth,trim=10 0 40 10,clip]
{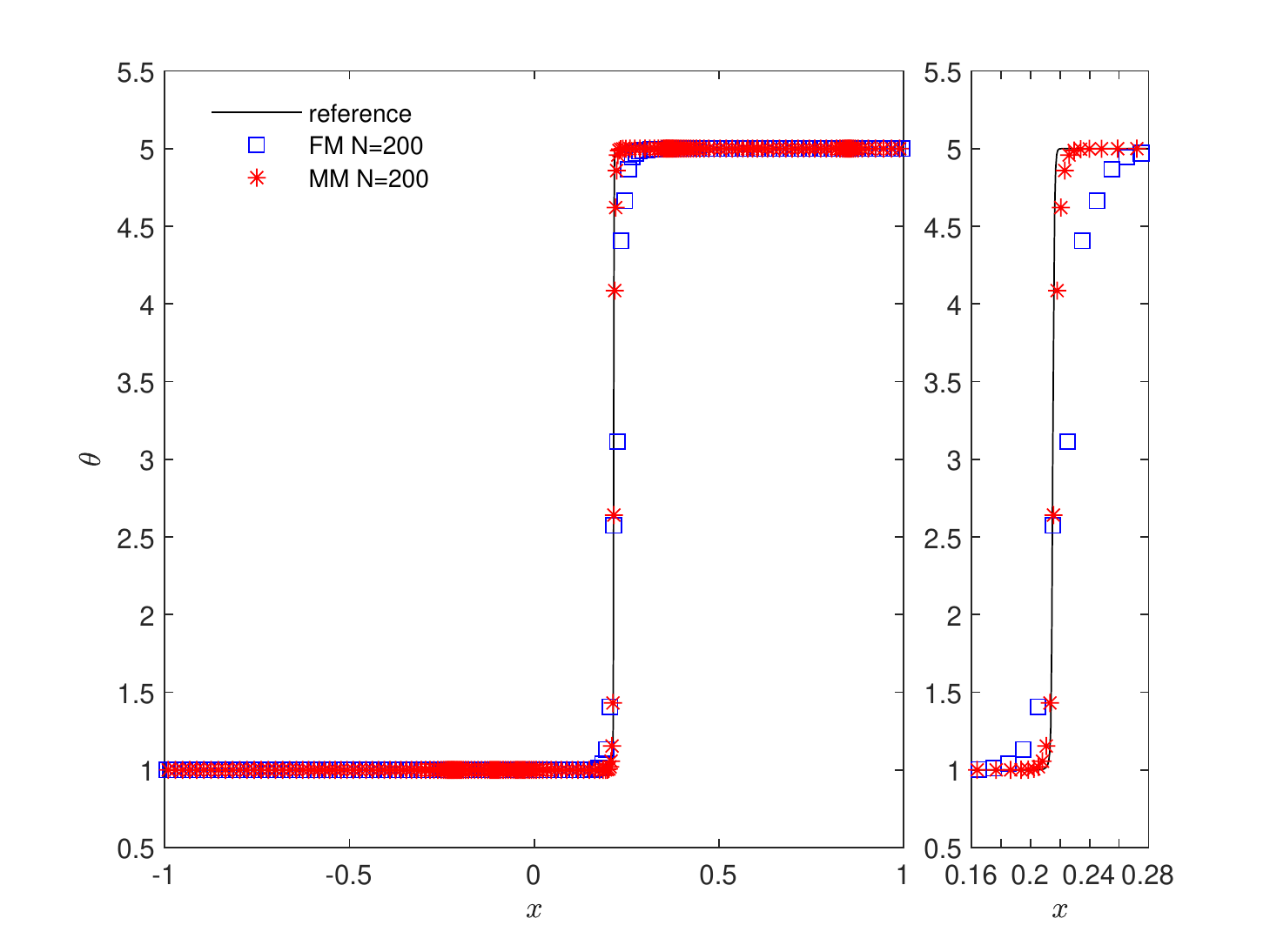}}
\subfigure[$\theta$: MM 200 vs FM 600]{
\includegraphics[width=0.4\textwidth,trim=10 0 40 10,clip]
{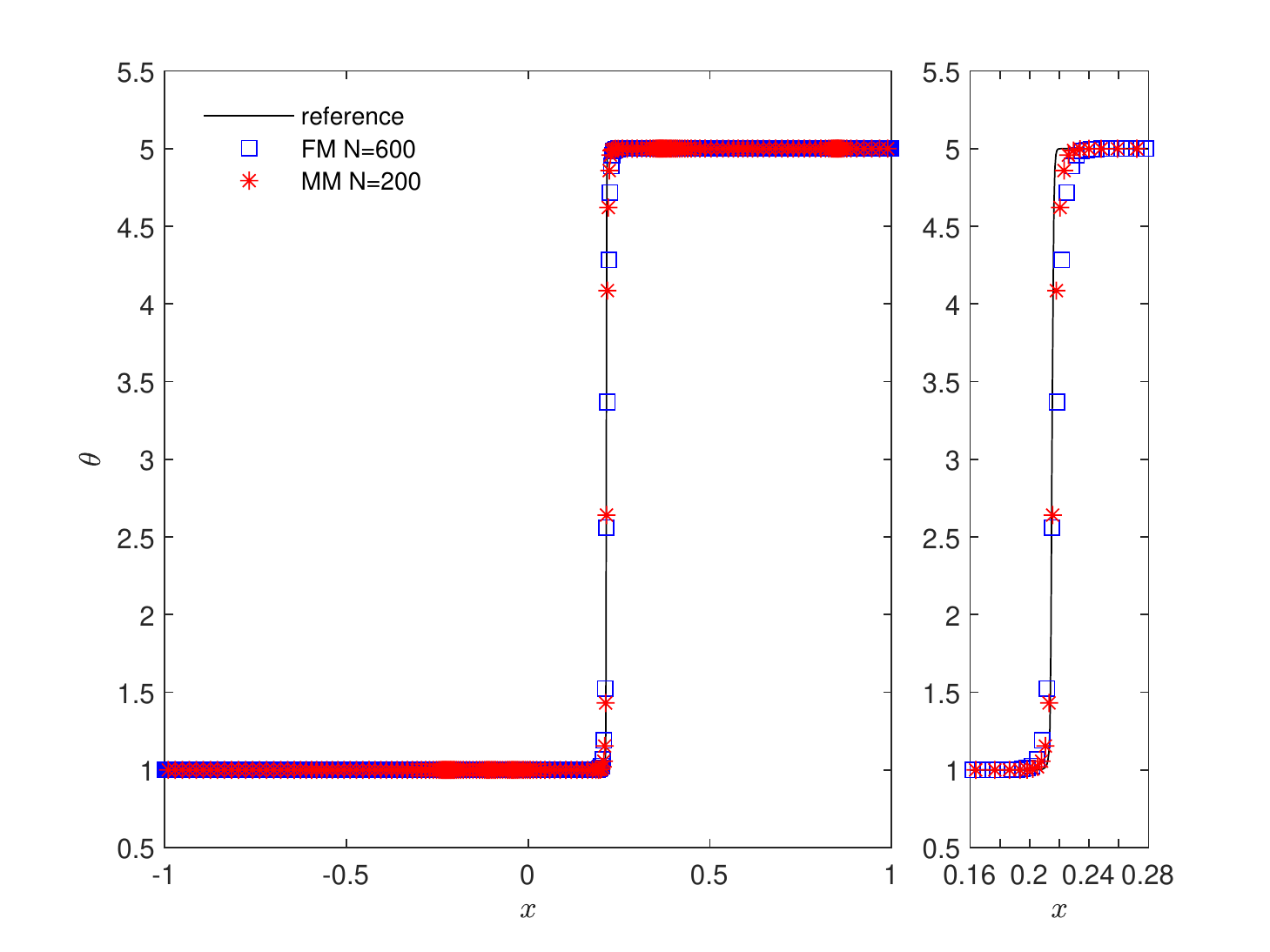}}
\subfigure[$\frac{1}{2}gh^2\theta$: MM 200 vs FM 200]{
\includegraphics[width=0.4\textwidth,trim=10 0 40 10,clip]
{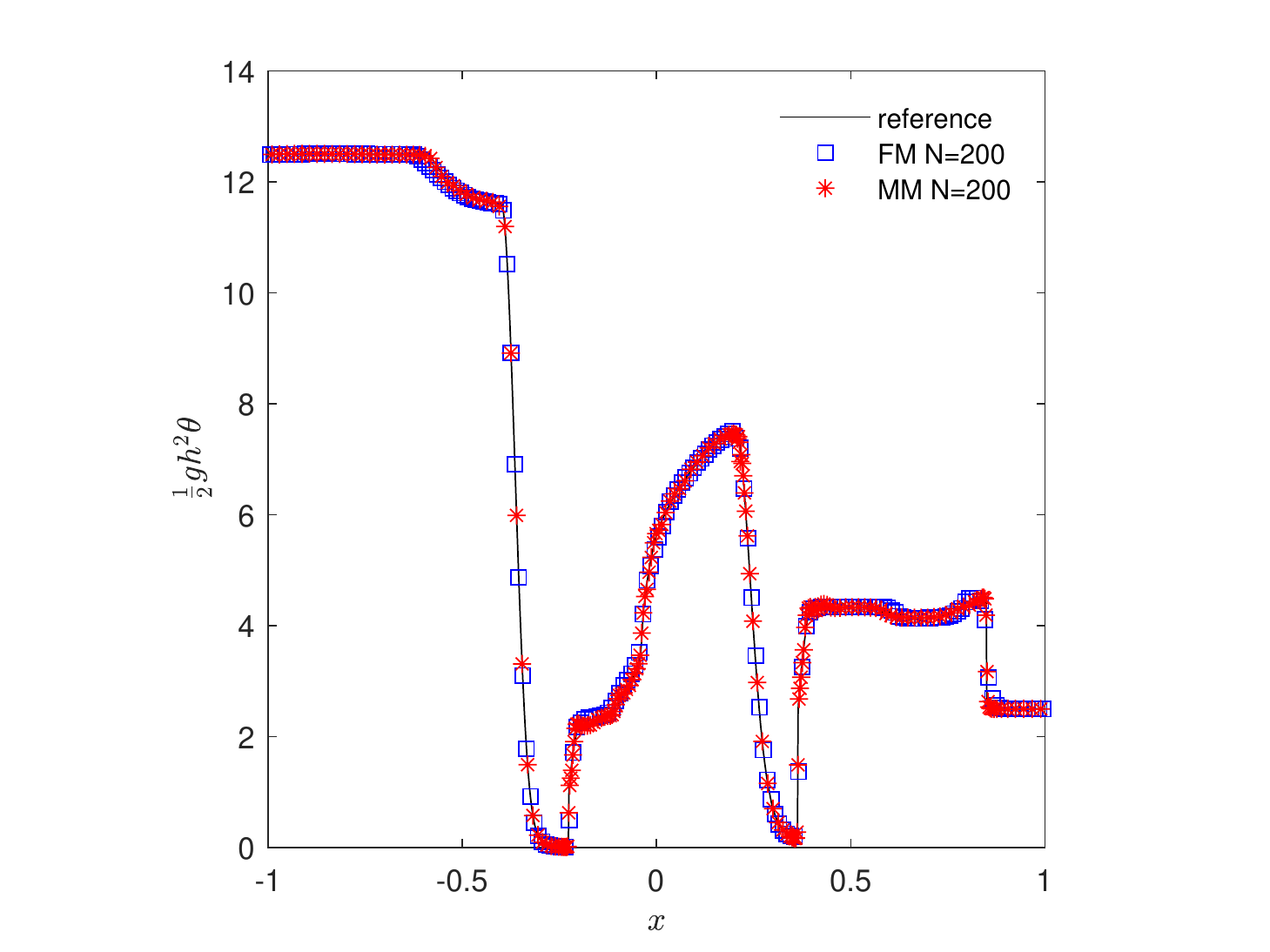}}
\subfigure[$\frac{1}{2}gh^2\theta$: MM 200 vs FM 600]{
\includegraphics[width=0.4\textwidth,trim=10 0 40 10,clip]
{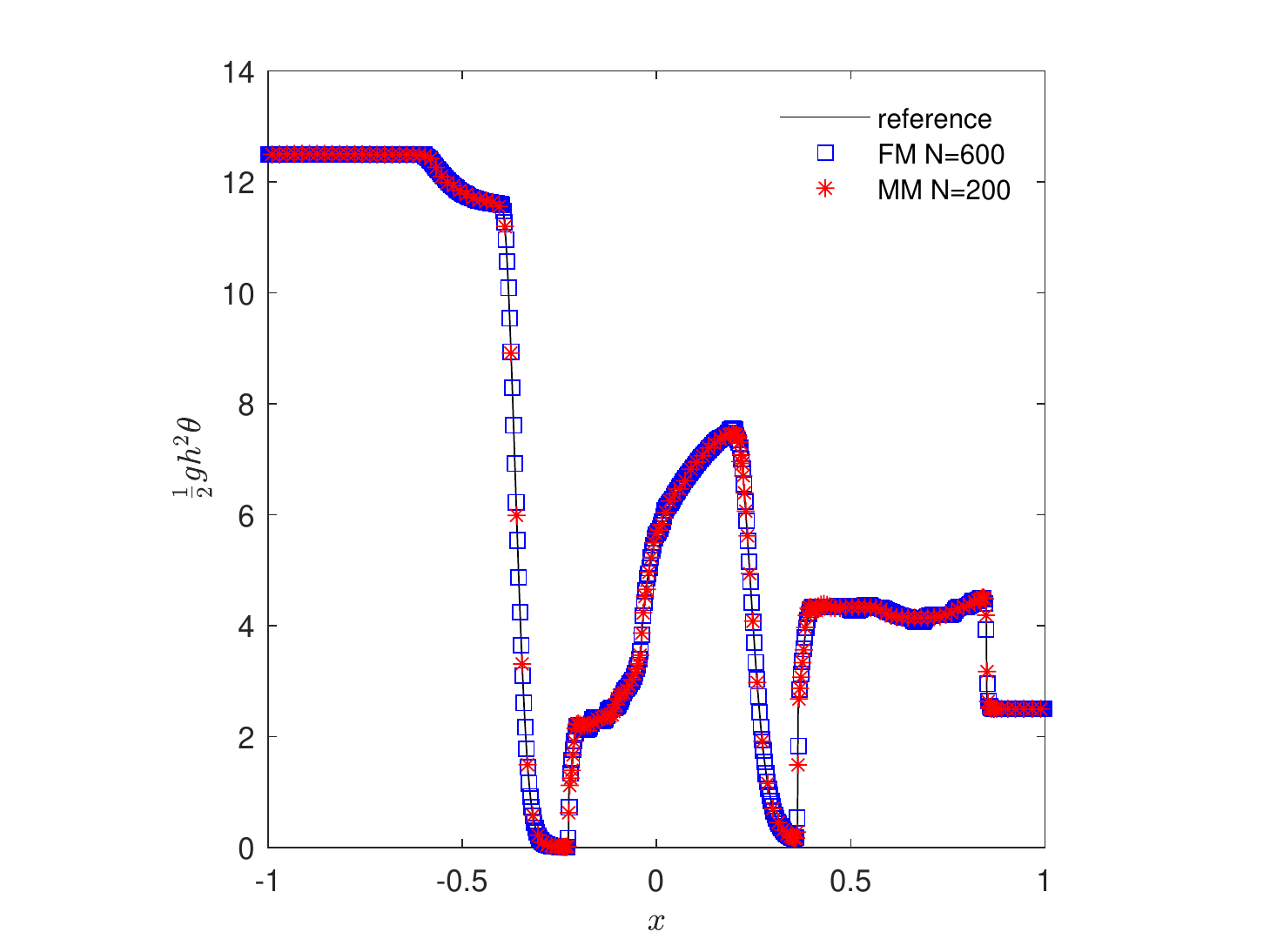}}
\caption{Example \ref{test0-1d}. The temperature $\theta$ and pressure $\frac{1}{2}gh^2\theta$ at $t = 0.3$ obtained with $P^2$-DG and a moving mesh of $N=200$ and fixed meshes of $N=200$ and $N=600$.}
\label{Fig:test0-1d-P-pp}
\end{figure}

\begin{example}\label{test1-2d}
(The lake-at-rest steady-state flow test for the 2D Ripa model.)
\end{example}
We choose this example to verify the well-balance property of the MM-DG scheme in two dimensions.
We solve the system on the domain $\Omega = (-1,1)\times(-1,1)$. The bottom topography reads as
\begin{equation}
b(x,y)=\begin{cases}
0.5e^{-100((x+0.5)^2+(y+0.5)^2)},& \text{for}~x<0\\
0.6e^{-100((x-0.5)^2+(y-0.5)^2)},& \text{otherwise.}
\end{cases}
\end{equation}
The initial water level, velocities and temperature are given by
\begin{equation*}
h(x,y,0)=3-b(x,y),
\quad u(x,y,0)=0,\quad  v(x,y,0)=0,\quad \theta(x,y,0)=\frac{4}{3}.
\end{equation*}
We use periodic boundary conditions for all unknown variables and compute the solution up to $t=0.12$.
Initial meshes used in the computation are shown in Fig.~\ref{Fig:test-2d-tri}.
The $L^1$ and $L^\infty$ error for solutions $h+b$, $hu$, $hv$, and $h\theta$
at $t = 0.12$ is listed in Table~\ref{tab:test1-2d-error}.
They show that our MM-DG method maintains the lake-at-rest steady state
to the level of round-off error in both $L^1$ and $L^\infty$ norm.

\begin{figure}[H]
\centering
\subfigure[$N=400$]{
\includegraphics[width=0.4\textwidth,trim=40 0 40 10,clip]{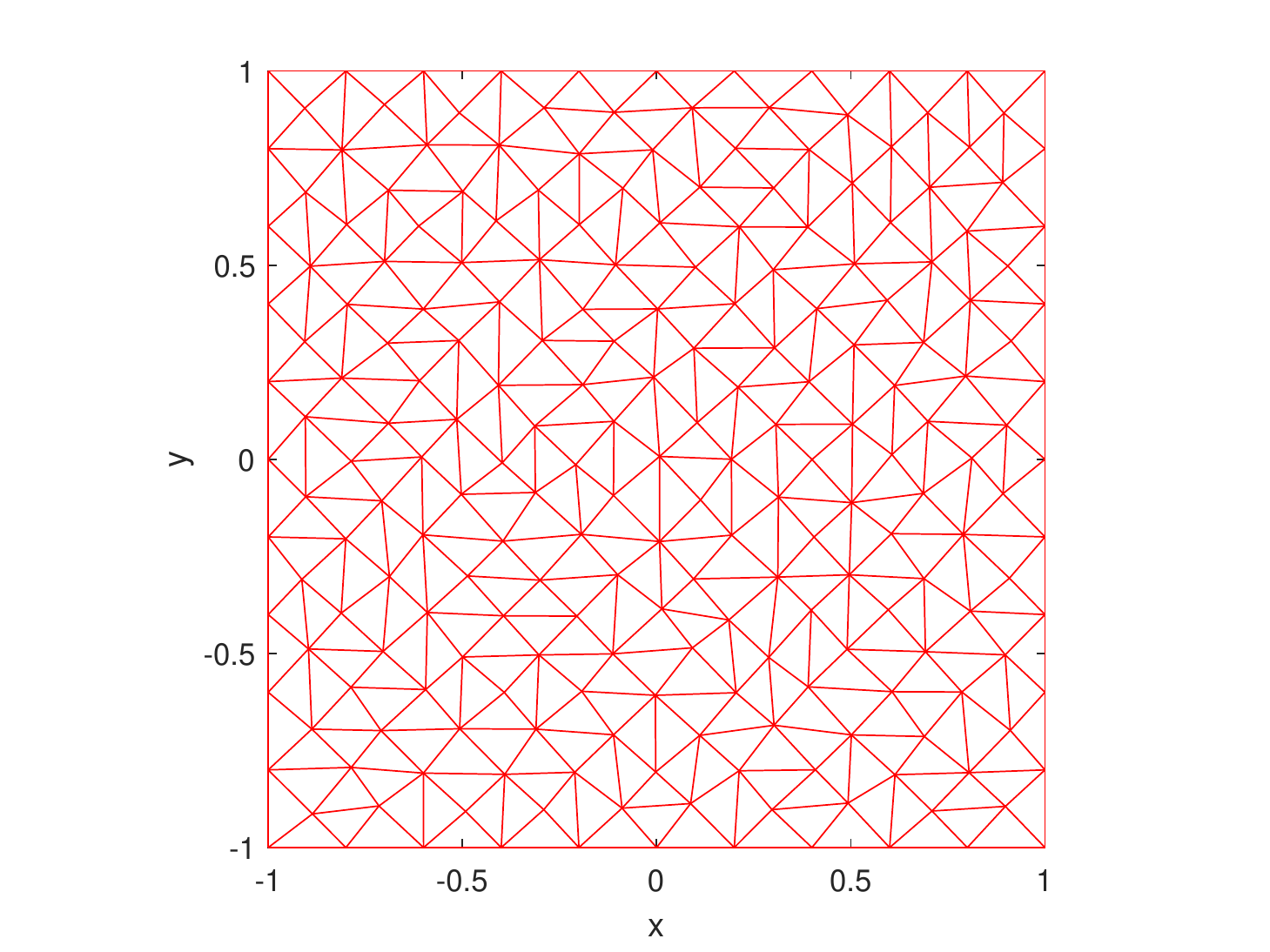}}
\subfigure[$N=1600$]{
\includegraphics[width=0.4\textwidth,trim=40 0 40 10,clip]{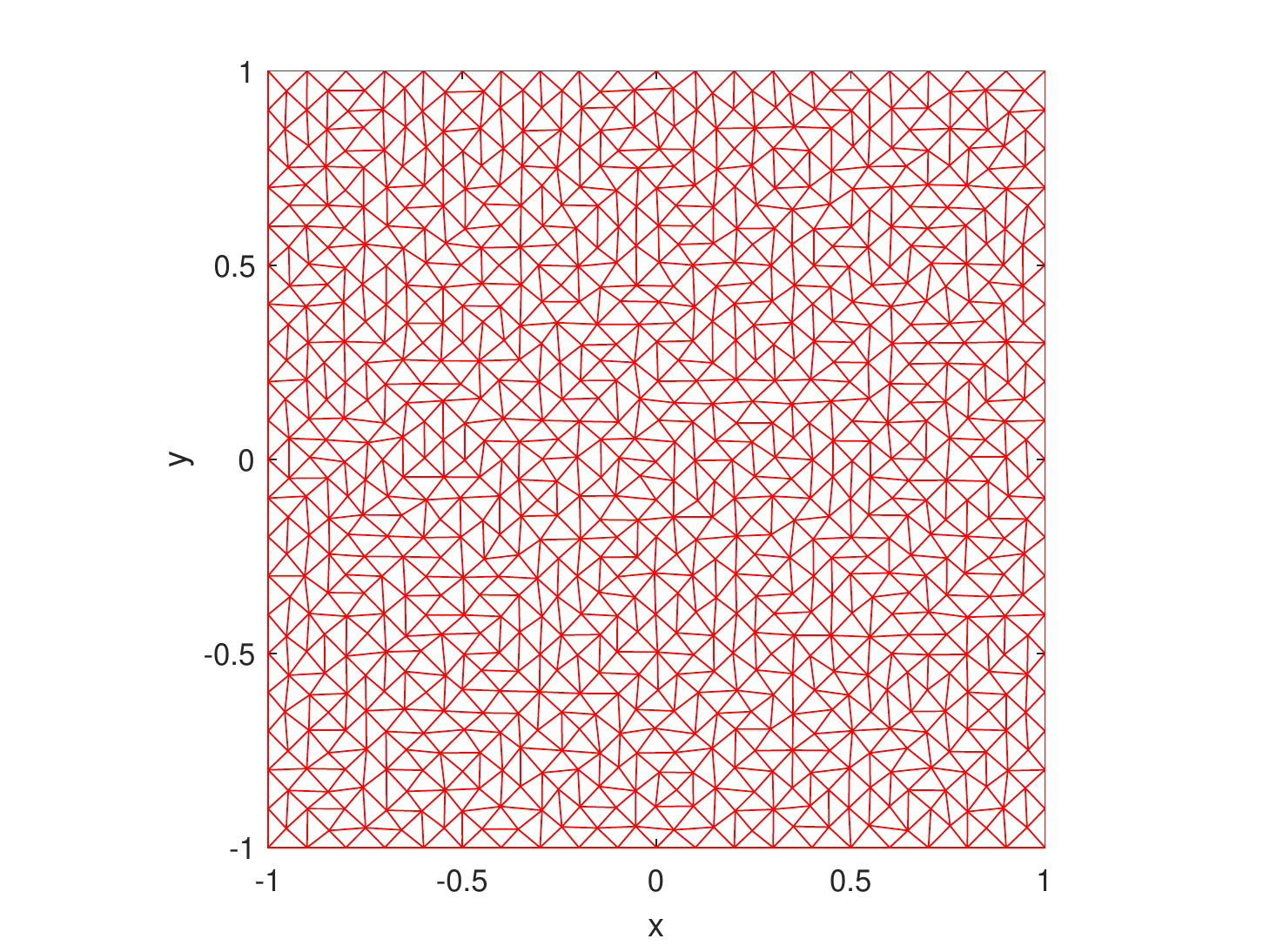}}
\caption{Example \ref{test1-2d}. Initial meshes used in the moving mesh computation.}
\label{Fig:test-2d-tri}
\end{figure}

\begin{table}[h]
\caption{Example \ref{test1-2d}. Well-balance test for the $P^2$ MM-DG method over the two bumps bottom function.}
\vspace{3pt}
\centering
\label{tab:test1-2d-error}
\begin{tabular}{cccccc}
 \toprule
 & $N$  & $h+b$&$hu$&$hv$&$h\theta$\\
\midrule
$L^1$-error
&400 &1.606E-15	&	1.723E-15	&	1.745E-15	&	4.507E-16	\\
&1600&1.501E-15	&	1.822E-15	&	1.839E-15	&	7.841E-16	\\
\midrule
$L^{\infty}$-error
&400  &6.345E-15	&	7.853E-15	&	7.864E-15	&	2.250E-15	\\
&1600 &6.277E-15	&	9.354E-15	&	9.289E-15	&	4.146E-15	\\
 \bottomrule	
\end{tabular}
\end{table}

\begin{example}\label{test6-2d}
(The perturbed lake-at-rest steady-state flow test for the 2D Ripa model.)
\end{example}
We choose this example to verify the ability of our well-balanced MM-DG scheme to capture small perturbations over the lake-at-rest water surface and temperature in two dimensions.
The bottom topography is an isolated elliptical shaped hump,
\begin{equation*}
B(x,y)=3e^{-5(x-0.9)^2-50(y-0.5)^2}, \quad (x,y) \in (-2,2)\times(0,1).
\end{equation*}
The initial depth of water, velocities, and temperature are given by
\begin{equation*}
(h,u,v,\theta)(x,y,0)=\begin{cases}
\big(6-b(x,y)+\varepsilon,~0,~0, ~\frac{24}{6+\varepsilon}\big),& \text{for}~x\in(0.05 , 0.15)\\
\big(6-b(x,y),~~~~~~0,~0, ~~4~~\big),& \text{otherwise}\\
\end{cases}
\end{equation*}
where $\varepsilon=0.1$.
As time being, the initial perturbation splits into three waves, one remaining at the initial position and the others propagating left and right at the characteristic speeds $\pm \sqrt{gh\theta}$.
The reflection boundary conditions are used for all domain boundary.

The mesh at $t= 0.16$ and $0.24$ obtained with the $P^2$ MM-DG method and a moving mesh of $N=14400$ are shown in Fig.~\ref{Fig:test6-s1-mesh}.
One can see that the distribution of the mesh concentration is consistent with the contours of $h+b$ while capturing the complex features in small perturbations.

The contours of $h+b$ at $t=0.16$ and $0.24$ obtained using the $P^2$ MM-DG method with a moving mesh of $N=14400$ and the fixed meshes of $N=14400$ and $N=102400$ are shown in
Fig.~\ref{Fig:test6-s1-H}.
The similar results of $hu$, $hv$, and $h\theta$ are shown in Figs.~\ref{Fig:test6-s1-hu},~\ref{Fig:test6-s1-hv}, and~\ref{Fig:test6-s1-eta}, respectively.
In Figs.~\ref{Fig:test6-s1-H-cuty} and ~\ref{Fig:test6-s1-hu-cuty}, the cut of the corresponding results of $h+b$ and $hu$ along the line $y=0.5$ is compared for the moving and fixed meshes.
We can see that the moving mesh solution with $N=14400$ is more accurate than that with
a fixed mesh of $N=14400$ and $N=102400$.

\begin{figure}[H]
\centering
\subfigure[mesh at $t=0.16$]{
\includegraphics[width=0.8\textwidth,trim=20 90 10 100,clip]
{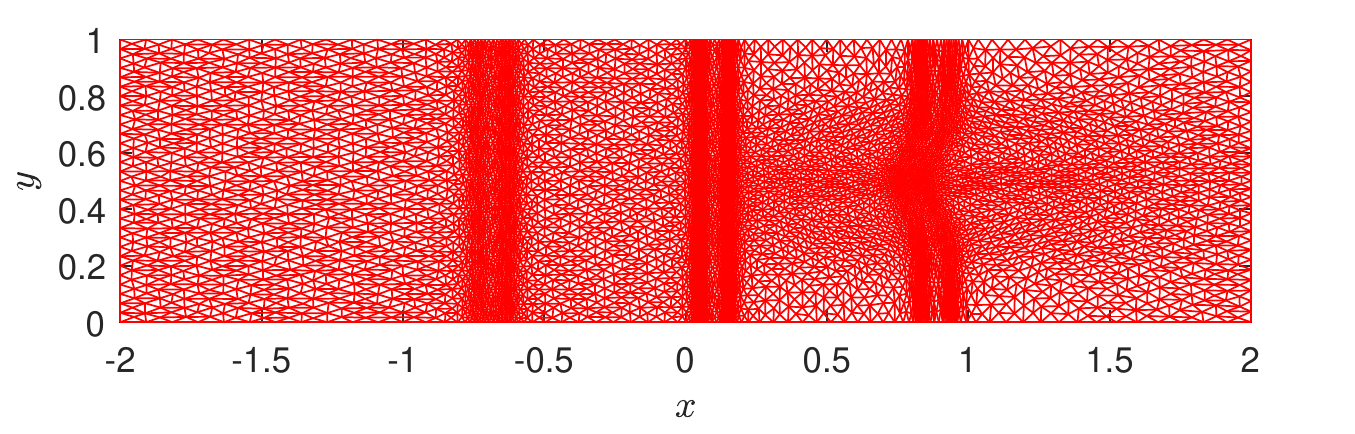}}
\subfigure[mesh at $t=0.24$]{
\includegraphics[width=0.8\textwidth,trim=20 90 10 100,clip]
{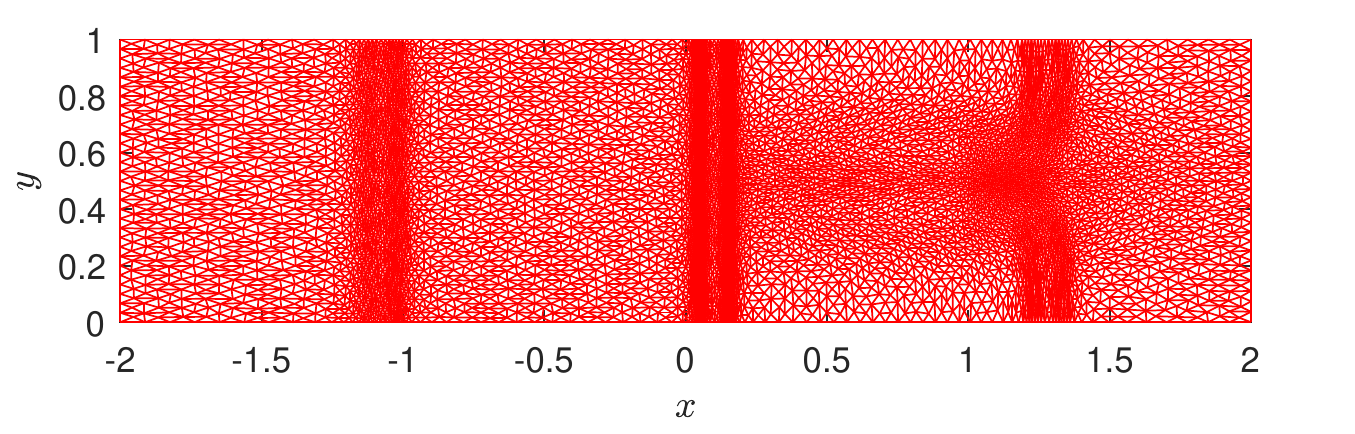}}
\caption{Example \ref{test6-2d}. The moving mesh of $N=14400$ at $t=0.16$ and $0.24$ are obtained with the $P^2$ MM-DG method.}
\label{Fig:test6-s1-mesh}
\end{figure}

\begin{figure}[H]
\centering
\subfigure[$h+b$: MM 14400]{
\includegraphics[width=0.45\textwidth,trim=10 100 10 100,clip]
{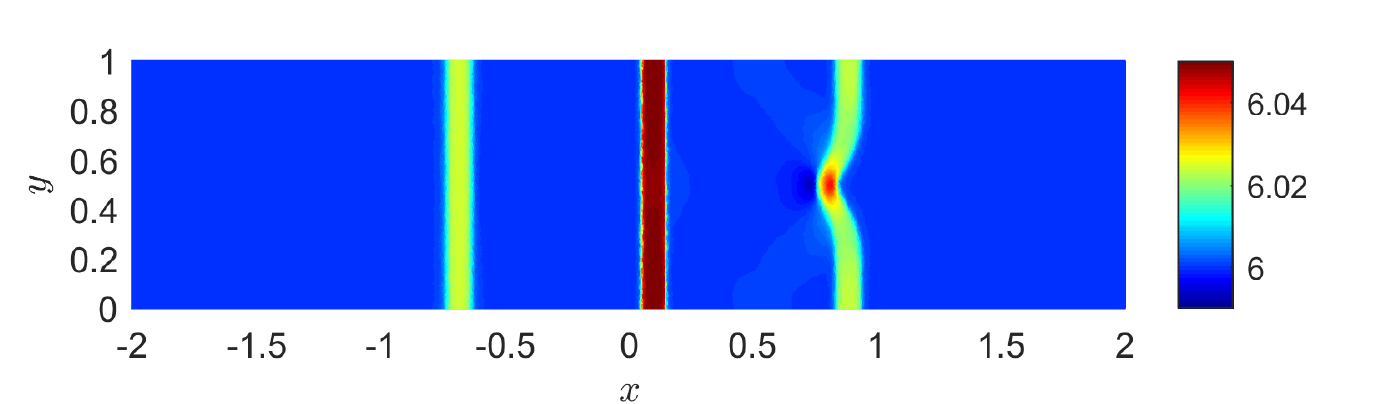}}
\subfigure[$h+b$: MM 14400]{
\includegraphics[width=0.45\textwidth,trim=10 100 10 100,clip]
{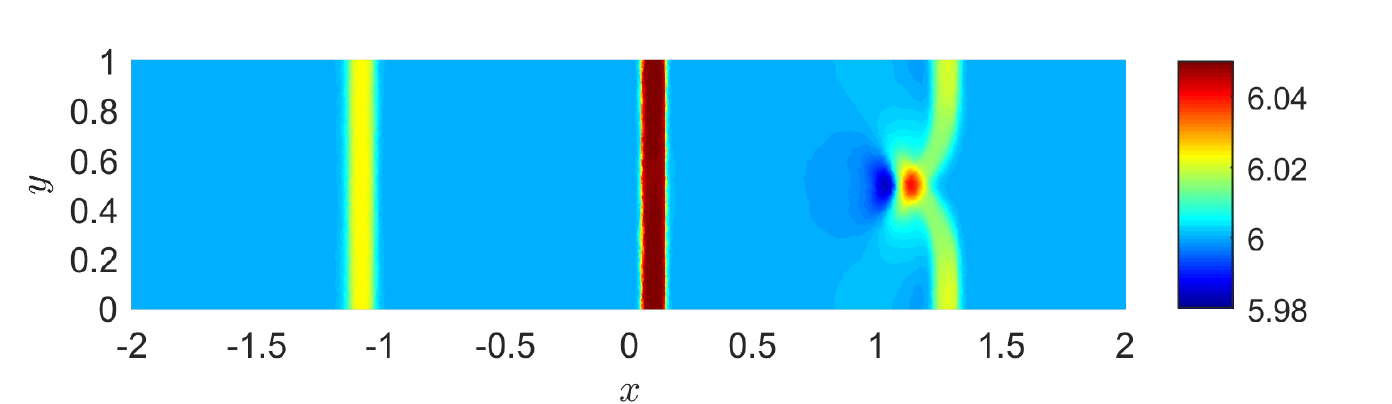}}
\subfigure[$h+b$: FM 14400]{
\includegraphics[width=0.45\textwidth,trim=10 100 10 100,clip]
{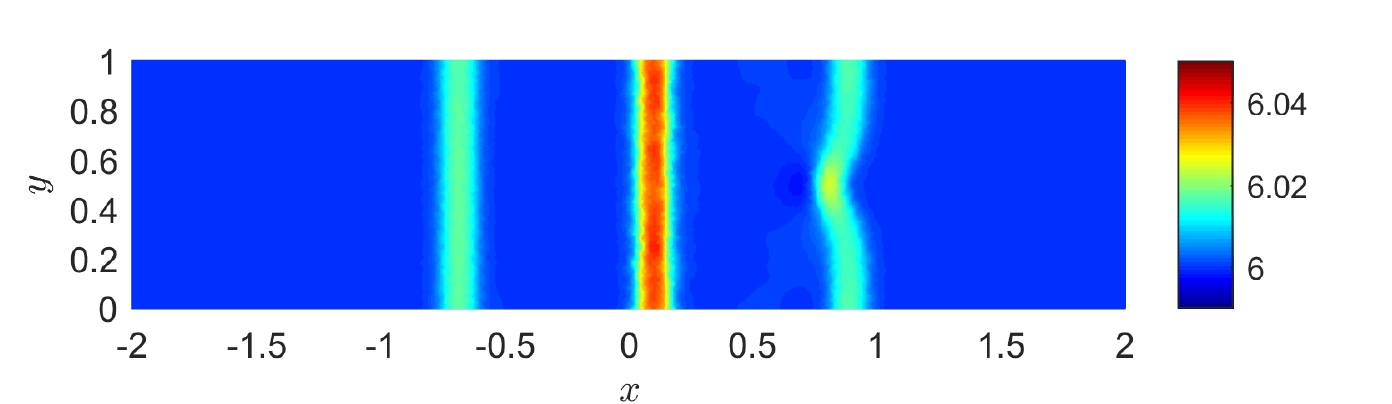}}
\subfigure[$h+b$: FM 14400]{
\includegraphics[width=0.45\textwidth,trim=10 100 10 100,clip]
{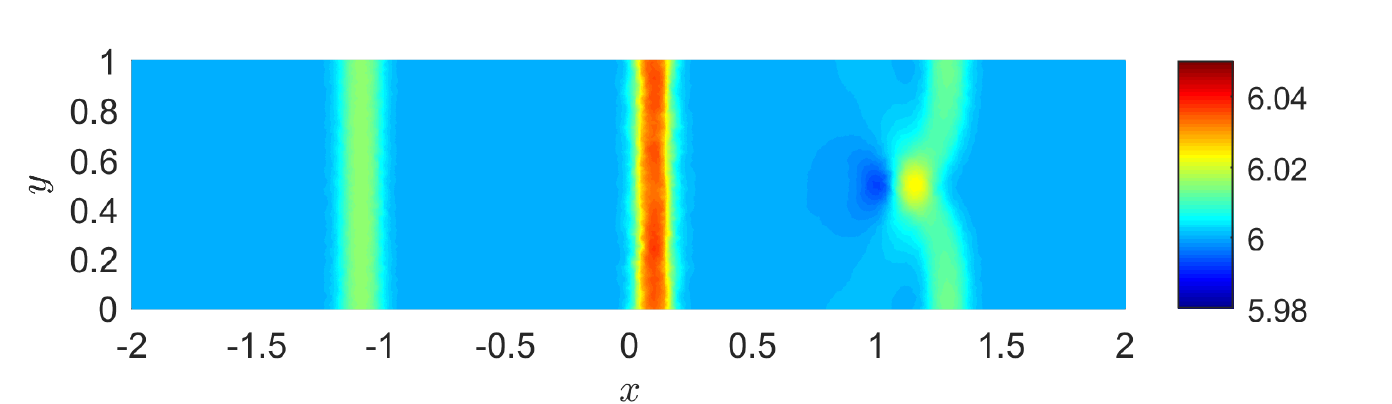}}
\subfigure[$h+b$: FM 102400]{
\includegraphics[width=0.45\textwidth,trim=10 100 10 100,clip]
{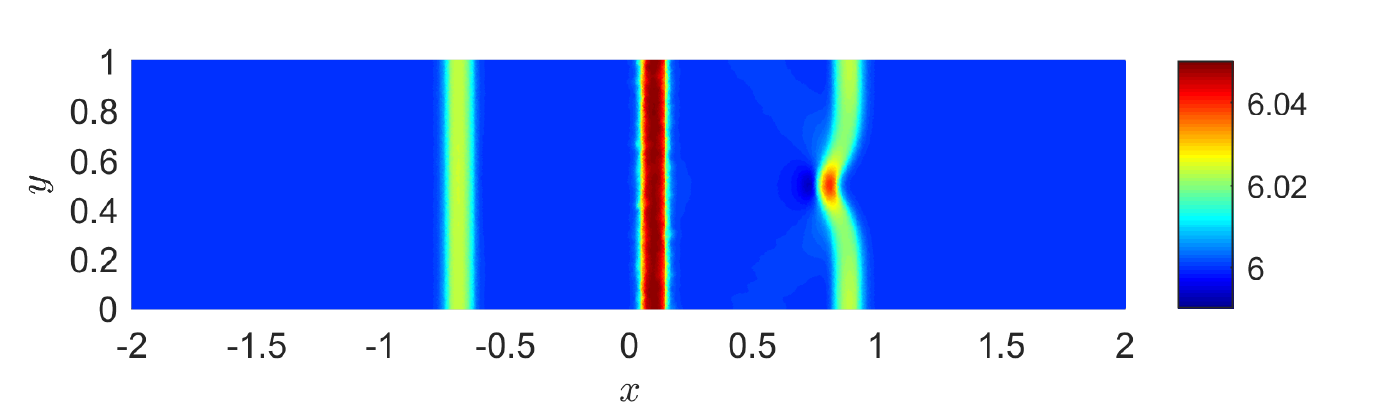}}
\subfigure[$h+b$: FM 102400]{
\includegraphics[width=0.45\textwidth,trim=10 100 10 100,clip]
{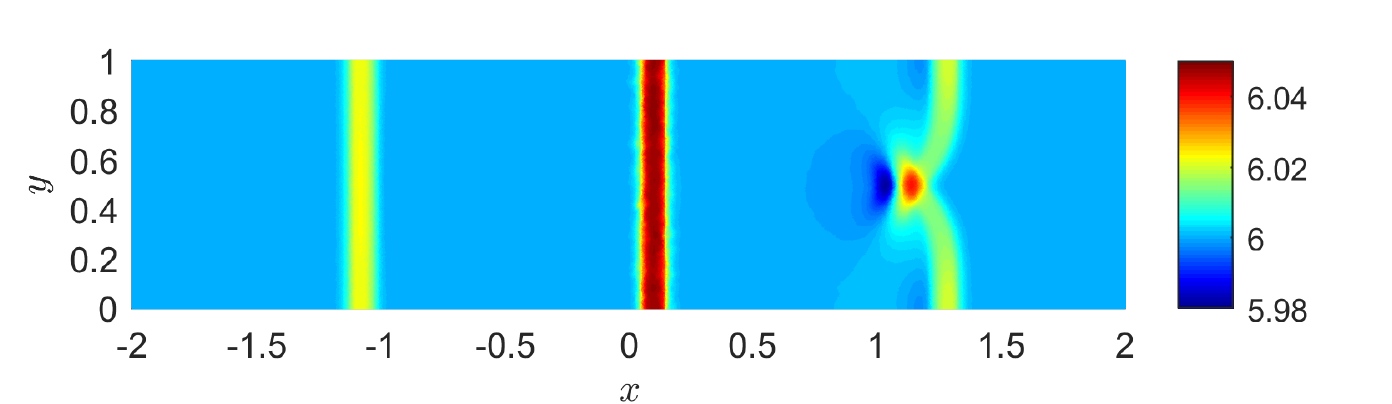}}
\caption{Example \ref{test6-2d}. The contours of $h+b$ at at $t=0.16$ (left column) and $0.24$ (right column) are obtained with the $P^2$ MM-DG method and a moving mesh of $N=14400$ and fixed meshes of $N=14400$ and $N=102400$.}
\label{Fig:test6-s1-H}
\end{figure}

\begin{figure}[H]
\centering
\subfigure[$hu$: MM 14400]{
\includegraphics[width=0.45\textwidth,trim=10 100 10 100,clip]
{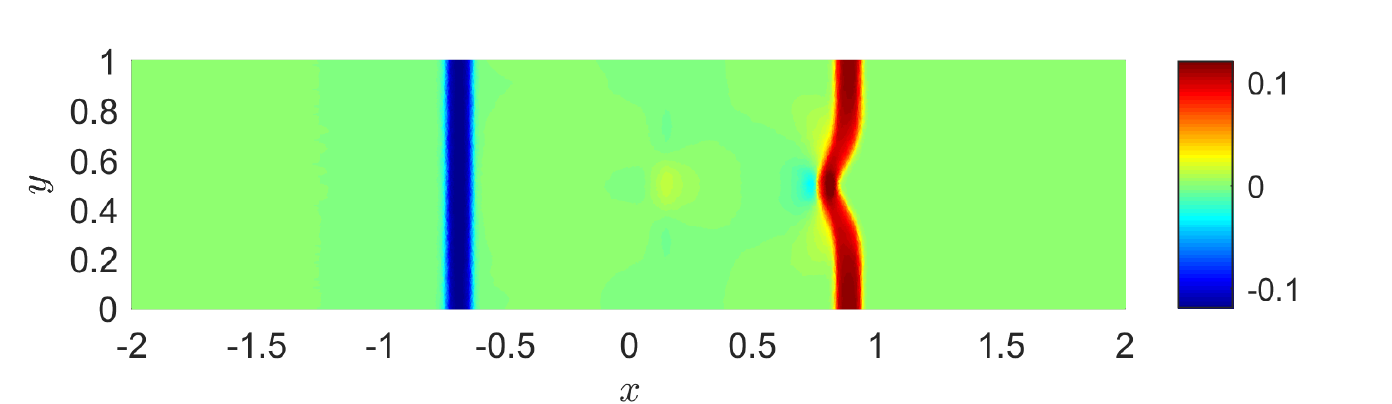}}
\subfigure[$hu$: MM 14400]{
\includegraphics[width=0.45\textwidth,trim=10 100 10 100,clip]
{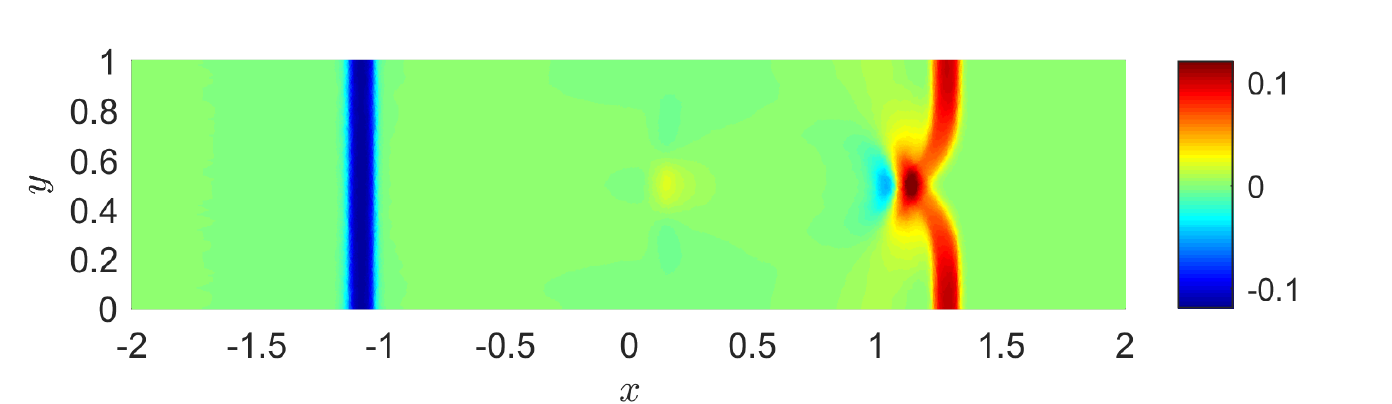}}
\subfigure[$hu$: FM 14400]{
\includegraphics[width=0.45\textwidth,trim=10 100 10 100,clip]
{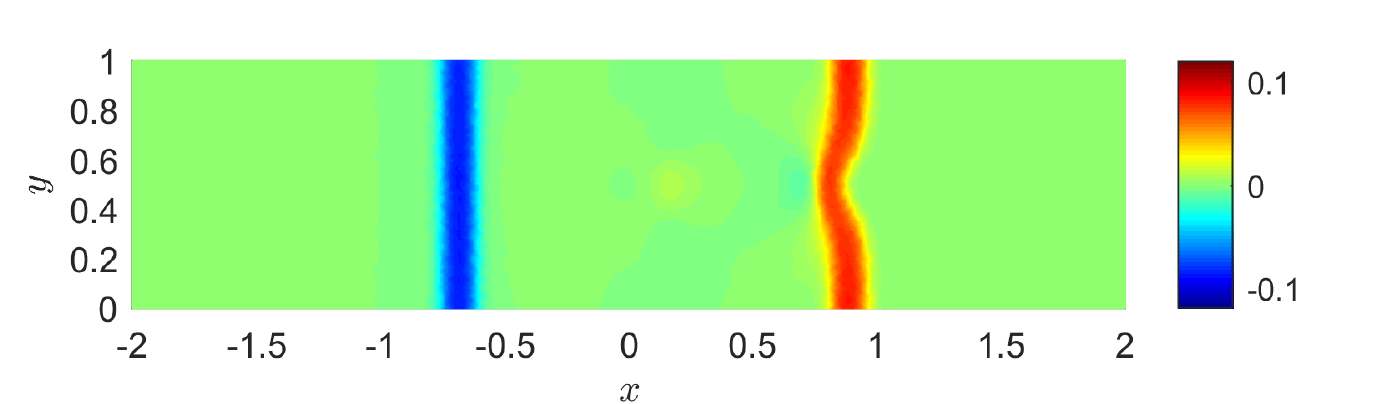}}
\subfigure[$hu$: FM 14400]{
\includegraphics[width=0.45\textwidth,trim=10 100 10 100,clip]
{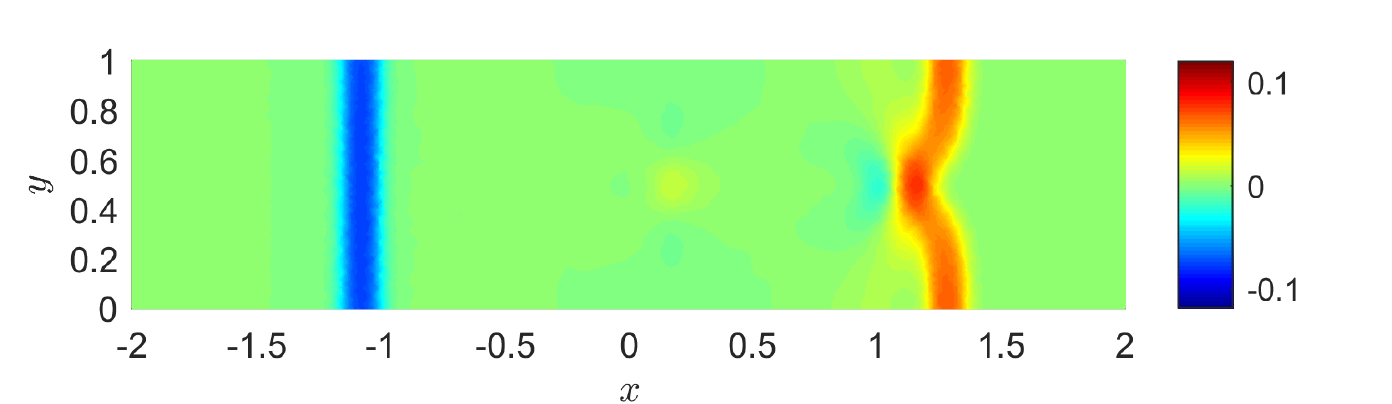}}
\subfigure[$hu$: FM 102400]{
\includegraphics[width=0.45\textwidth,trim=10 100 10 100,clip]
{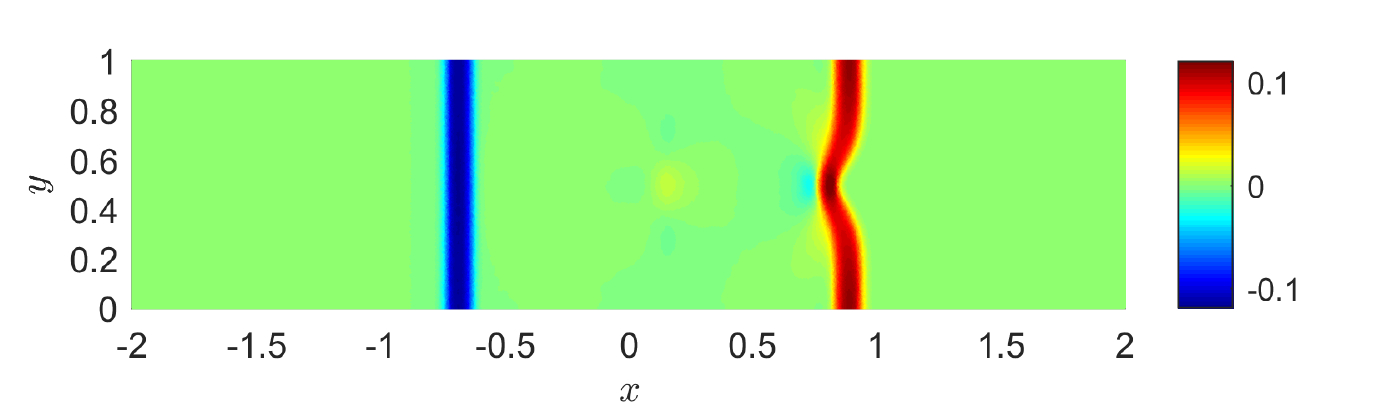}}
\subfigure[$hu$: FM 102400]{
\includegraphics[width=0.45\textwidth,trim=10 100 10 100,clip]
{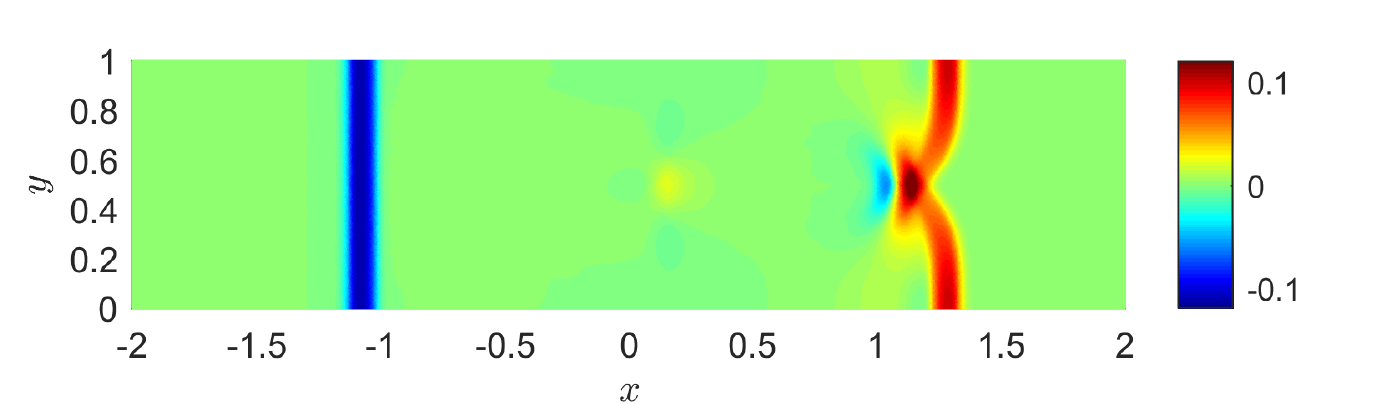}}
\caption{Example \ref{test6-2d}. The contours of $hu$ at at $t=0.16$ (left column) and $0.24$ (right column) are obtained with the $P^2$ MM-DG method and a moving mesh of $N=14400$ and fixed meshes of $N=14400$ and $N=102400$.}
\label{Fig:test6-s1-hu}
\end{figure}

\begin{figure}[H]
\centering
\subfigure[$hv$: MM 14400]{
\includegraphics[width=0.45\textwidth,trim=10 100 10 100,clip]
{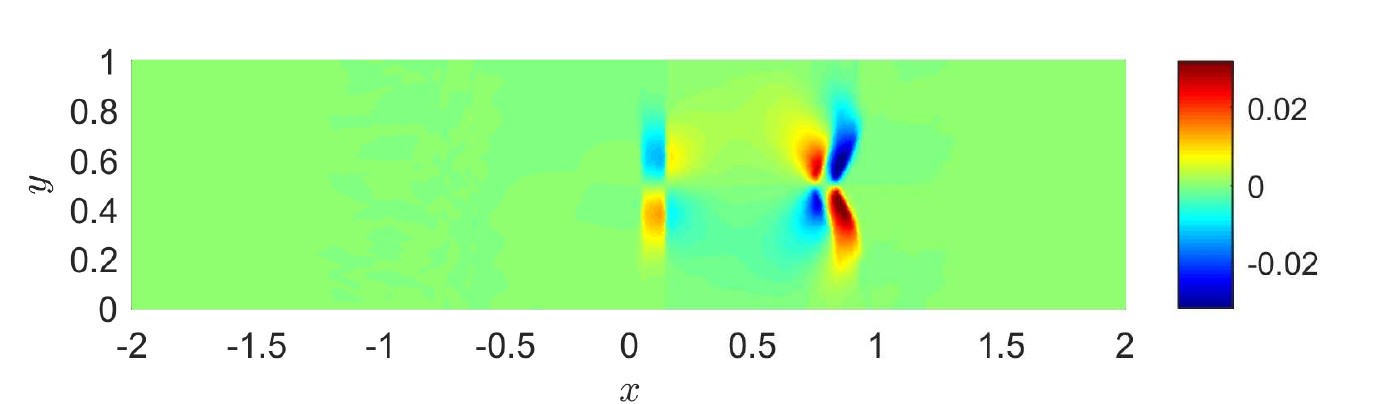}}
\subfigure[$hv$: MM 14400]{
\includegraphics[width=0.45\textwidth,trim=10 100 10 100,clip]
{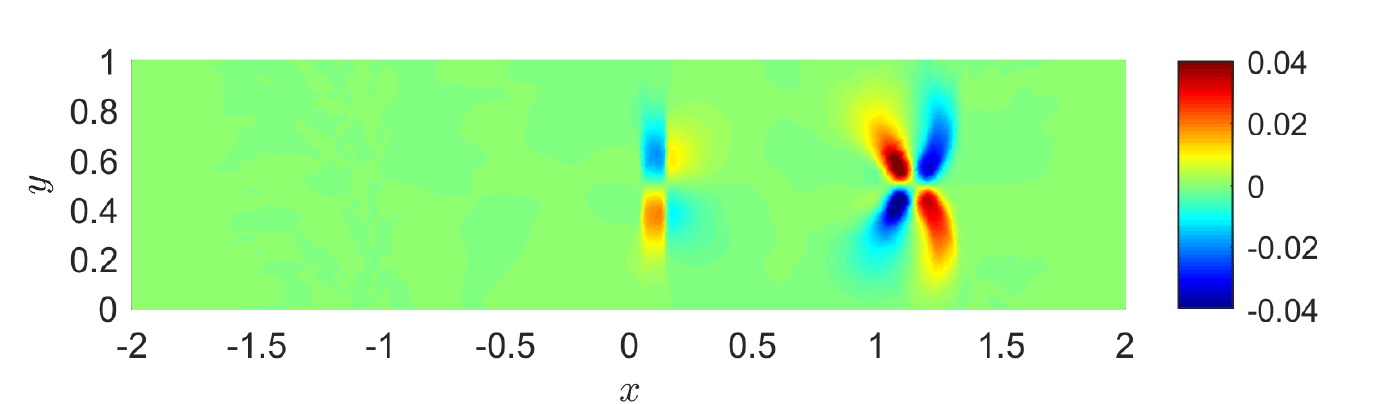}}
\subfigure[$hv$: FM 14400]{
\includegraphics[width=0.45\textwidth,trim=10 100 10 100,clip]
{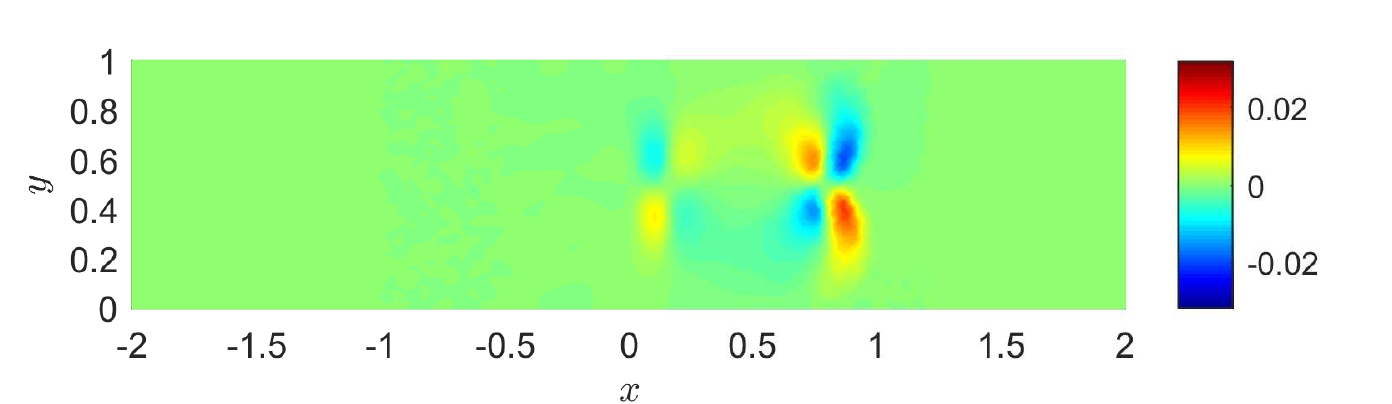}}
\subfigure[$hv$: FM 14400]{
\includegraphics[width=0.45\textwidth,trim=10 100 10 100,clip]
{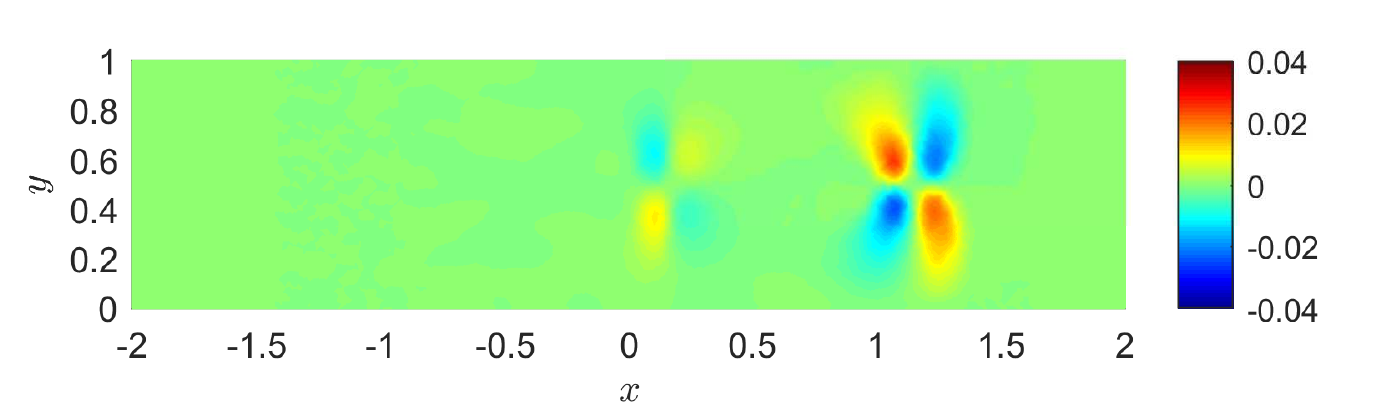}}
\subfigure[$hv$: FM 102400]{
\includegraphics[width=0.45\textwidth,trim=10 100 10 100,clip]
{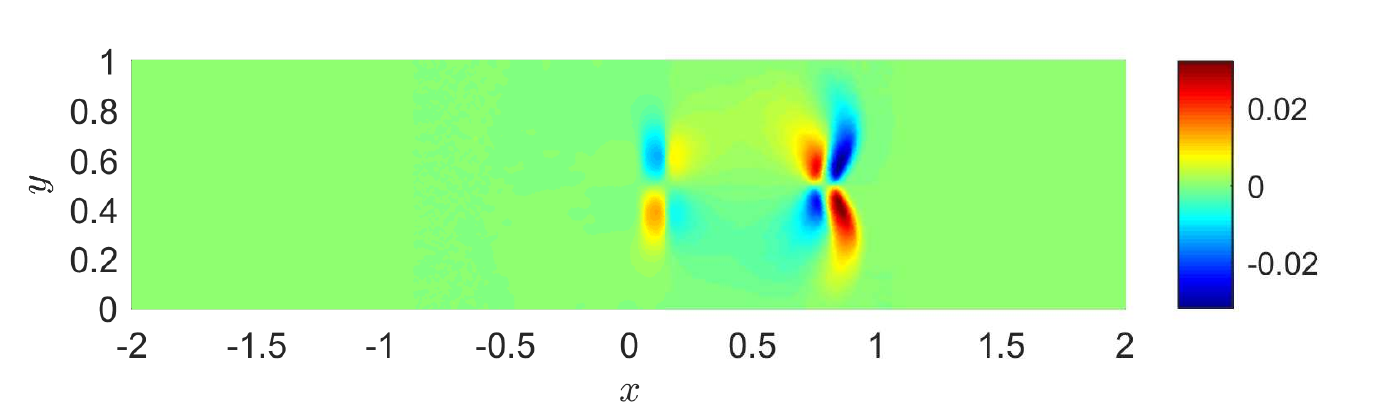}}
\subfigure[$hv$: FM 102400]{
\includegraphics[width=0.45\textwidth,trim=10 100 10 100,clip]
{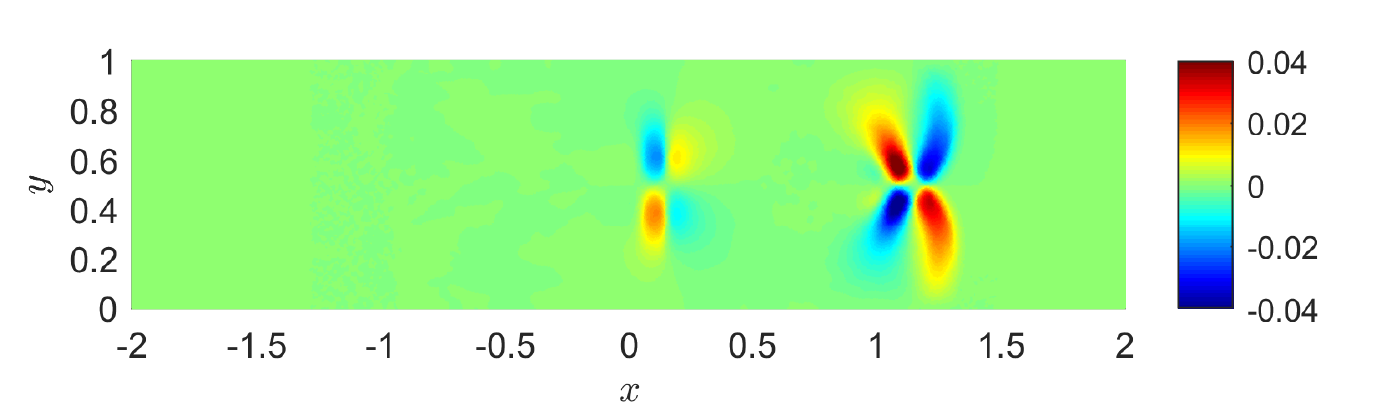}}
\caption{Example \ref{test6-2d}. The contours of $hv$ at at $t=0.16$ (left column) and $0.24$ (right column) are obtained with the $P^2$ MM-DG method and a moving mesh of $N=14400$ and fixed meshes of $N=14400$ and $102400$.}
\label{Fig:test6-s1-hv}
\end{figure}

\begin{figure}[H]
\centering
\subfigure[$h\theta$: MM 14400]{
\includegraphics[width=0.45\textwidth,trim=10 100 10 100,clip]
{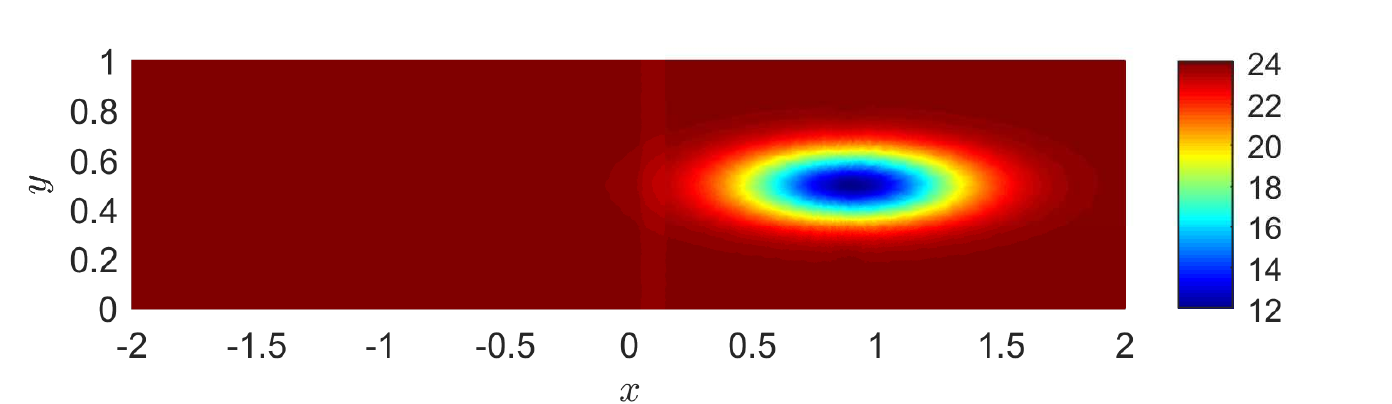}}
\subfigure[$h\theta$: MM 14400]{
\includegraphics[width=0.45\textwidth,trim=10 100 10 100,clip]
{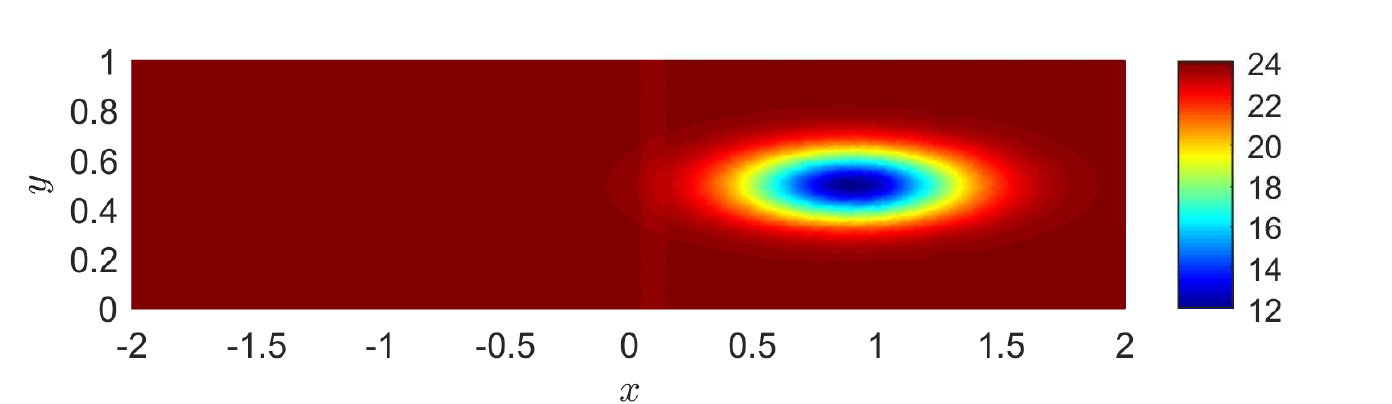}}
\subfigure[$h\theta$: FM 14400]{
\includegraphics[width=0.45\textwidth,trim=10 100 10 100,clip]
{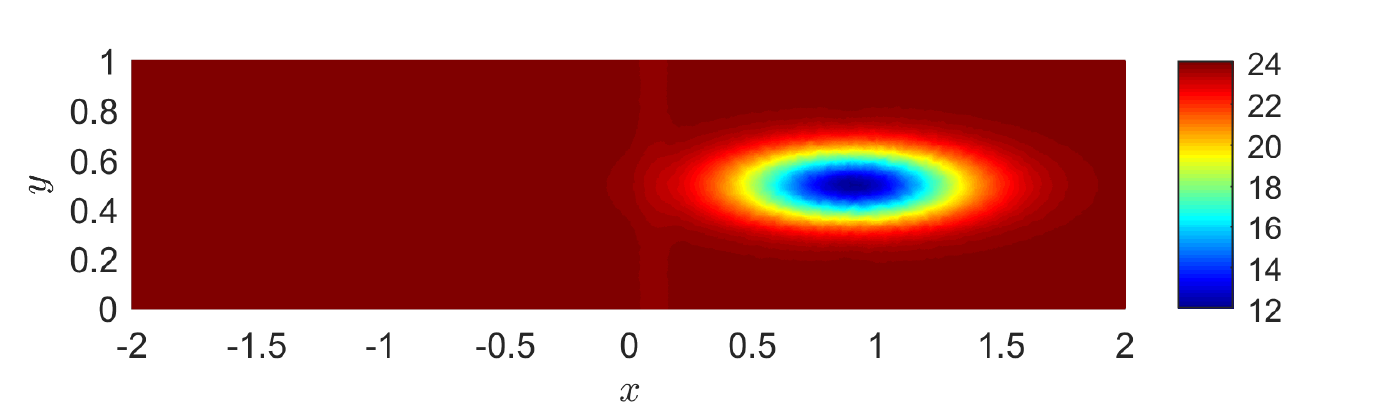}}
\subfigure[$h\theta$: FM 14400]{
\includegraphics[width=0.45\textwidth,trim=10 100 10 100,clip]
{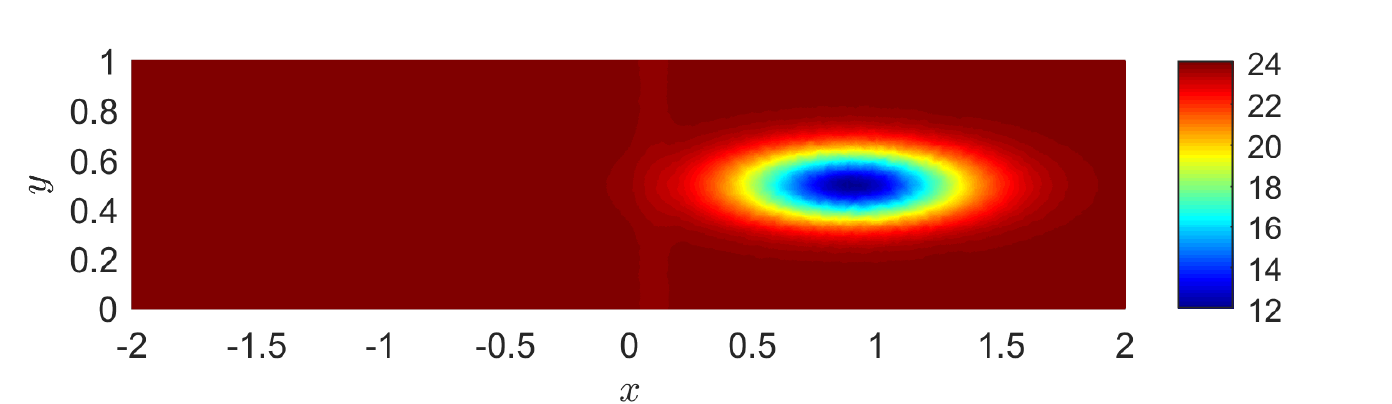}}
\subfigure[$h\theta$: FM 102400]{
\includegraphics[width=0.45\textwidth,trim=10 100 10 100,clip]
{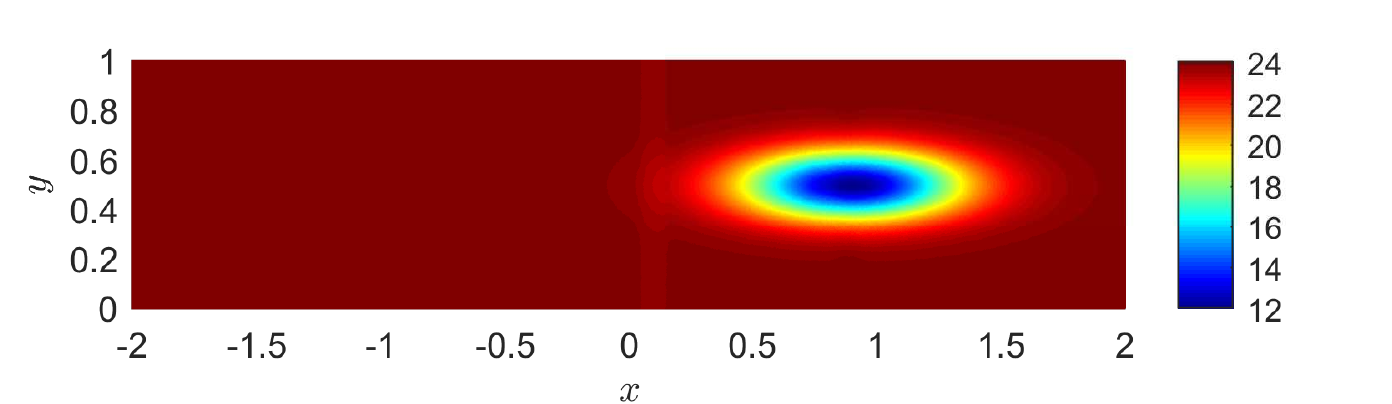}}
\subfigure[$h\theta$: FM 102400]{
\includegraphics[width=0.45\textwidth,trim=10 100 10 100,clip]
{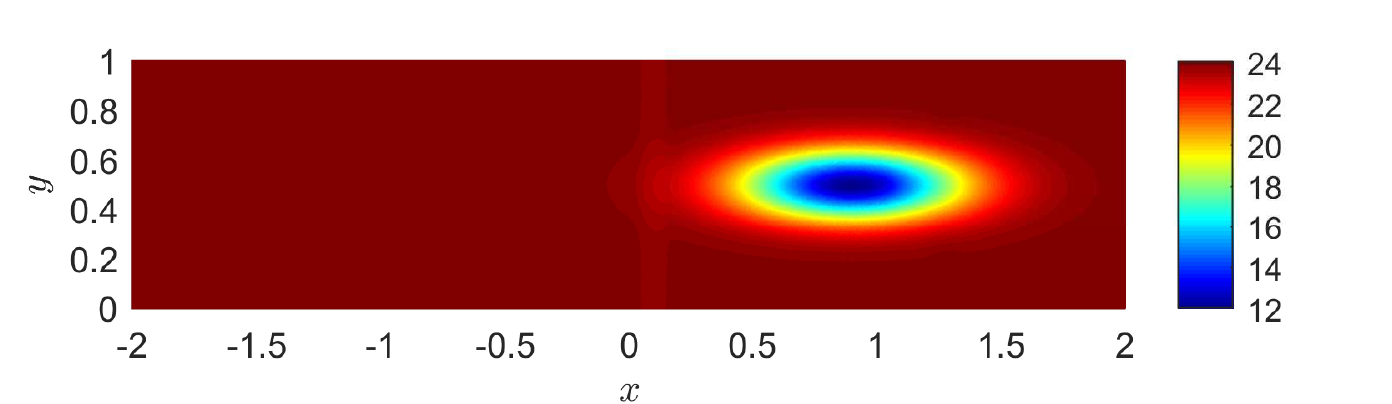}}
\caption{Example \ref{test6-2d}. The contours of $h\theta$ at $t=0.16$ (left column) and $0.24$ (right column) are obtained with the $P^2$ MM-DG method and a moving mesh of $N=14400$ and fixed meshes of $N=14400$ and $N=102400$.}
\label{Fig:test6-s1-eta}
\end{figure}

\begin{figure}[H]
\centering
\subfigure[$h+b$: $t=0.16$]{
\includegraphics[width=0.4\textwidth,trim=10 0 40 10,clip]
{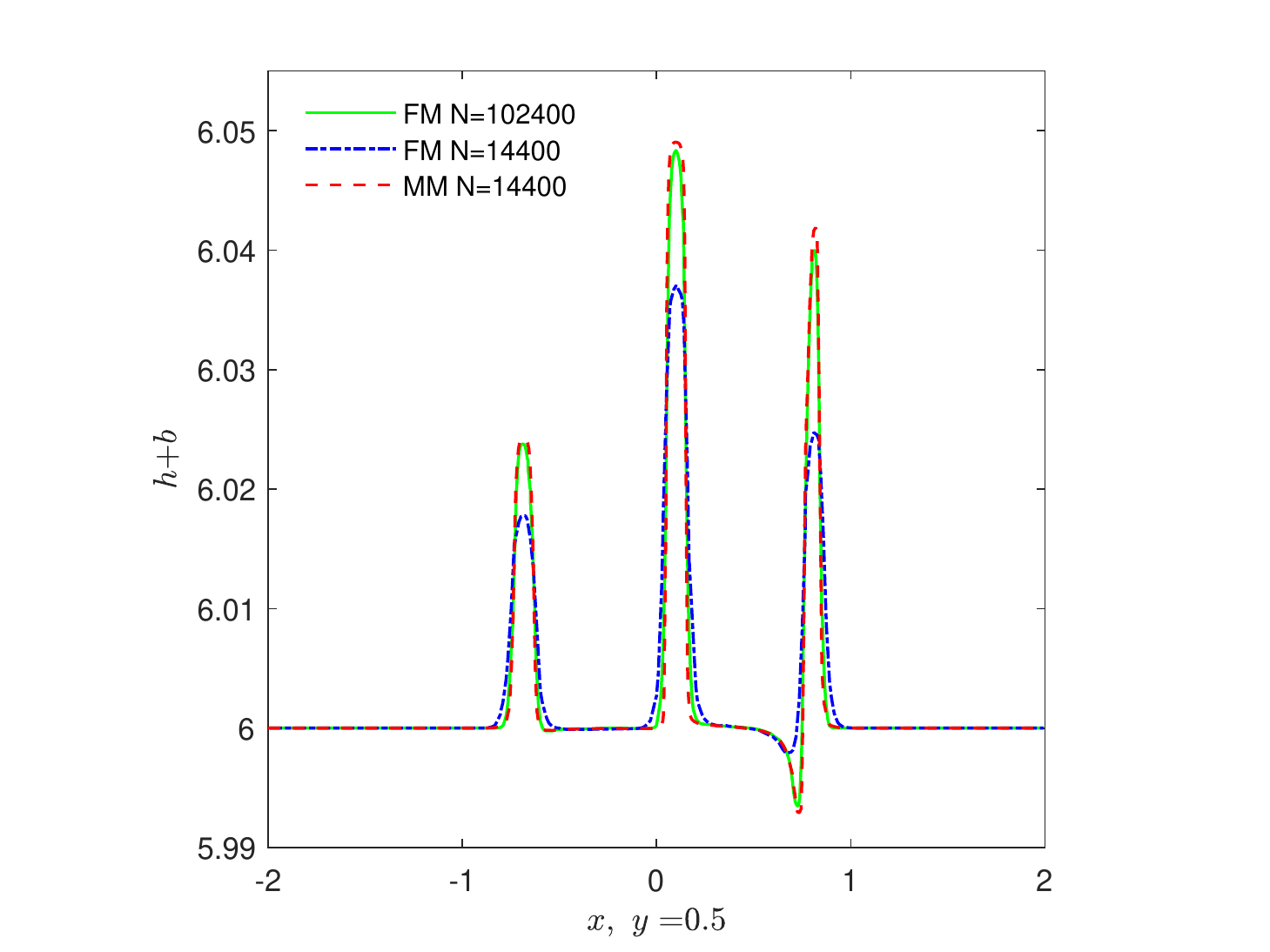}}
\subfigure[close view of (a)]{
\includegraphics[width=0.4\textwidth,trim=10 0 40 10,clip]
{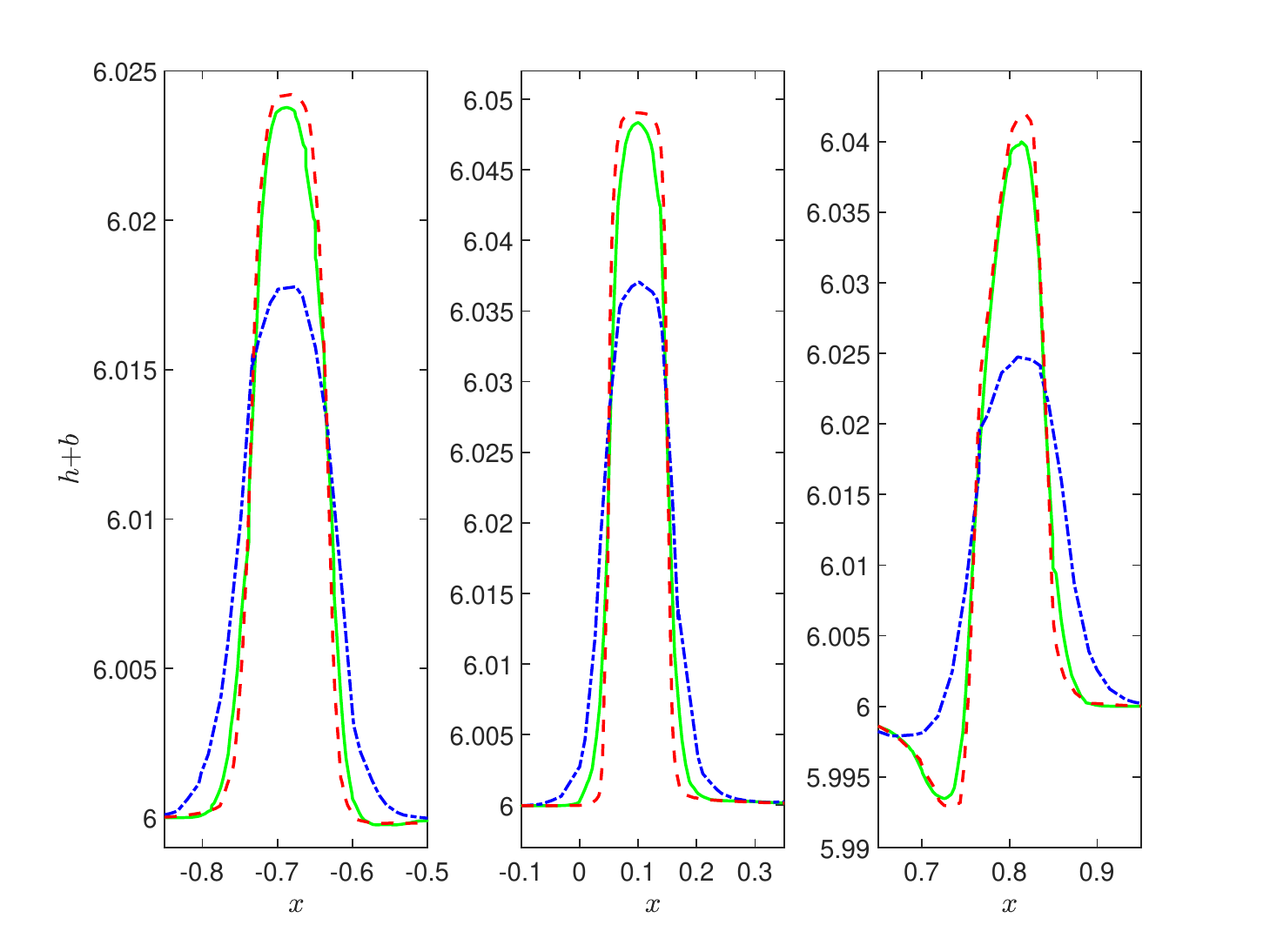}}
\subfigure[$h+b$: $t=0.24$]{
\includegraphics[width=0.4\textwidth,trim=10 0 40 10,clip]
{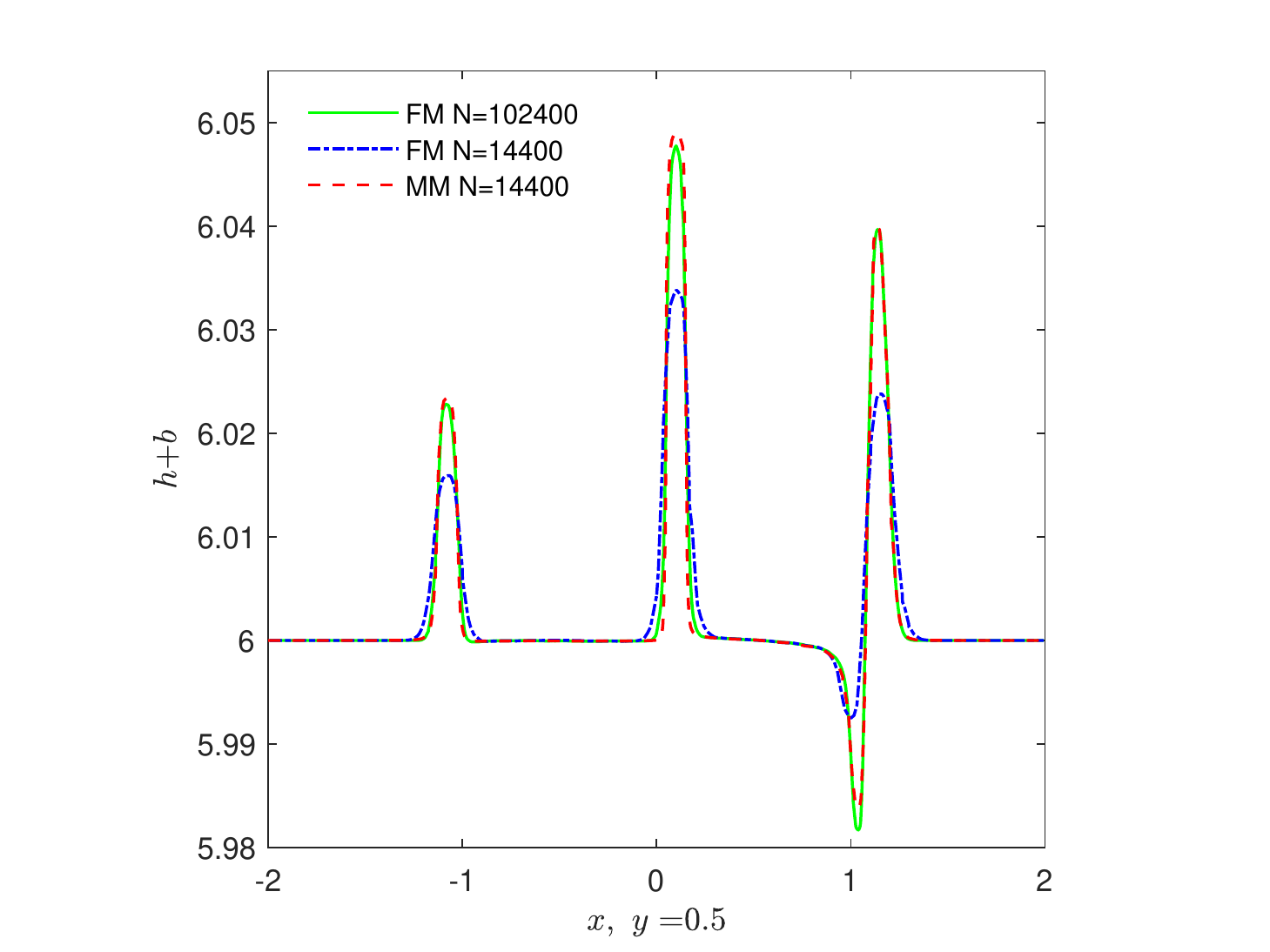}}
\subfigure[close view of (c)]{
\includegraphics[width=0.4\textwidth,trim=10 0 40 10,clip]
{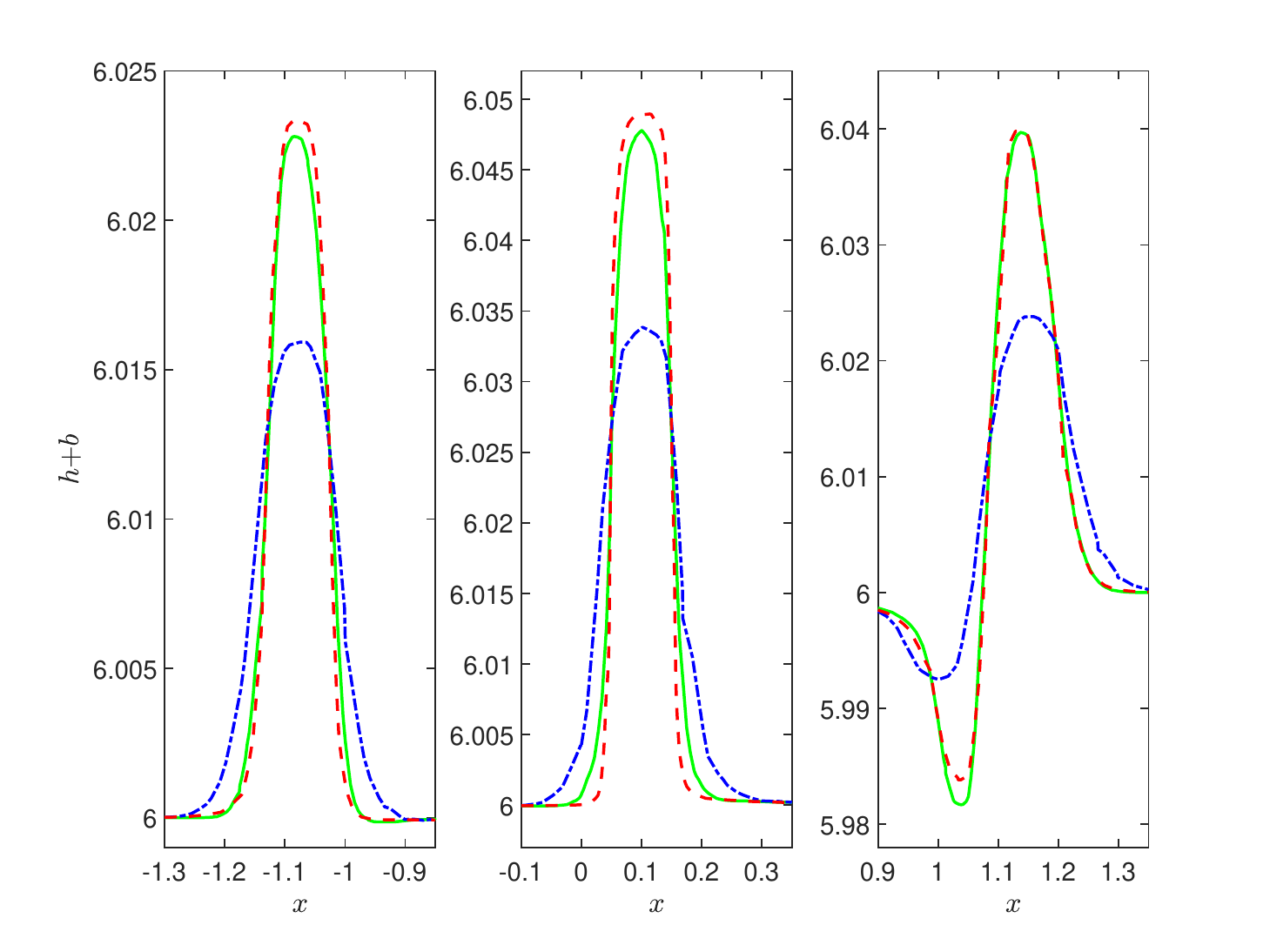}}
\caption{Example \ref{test6-2d}. The cut of $h+b$ along the line $y=0.5$ at $t=0.16$ (first row) and $0.24$ (second row) are obtained with $P^2$-DG and a moving mesh of $N=14400$ and fixed meshes of $N=14400$ and $N=102400$.}
\label{Fig:test6-s1-H-cuty}
\end{figure}

\begin{figure}[H]
\centering
\subfigure[$hu$: $t=0.16$]{
\includegraphics[width=0.4\textwidth,trim=10 0 40 10,clip]
{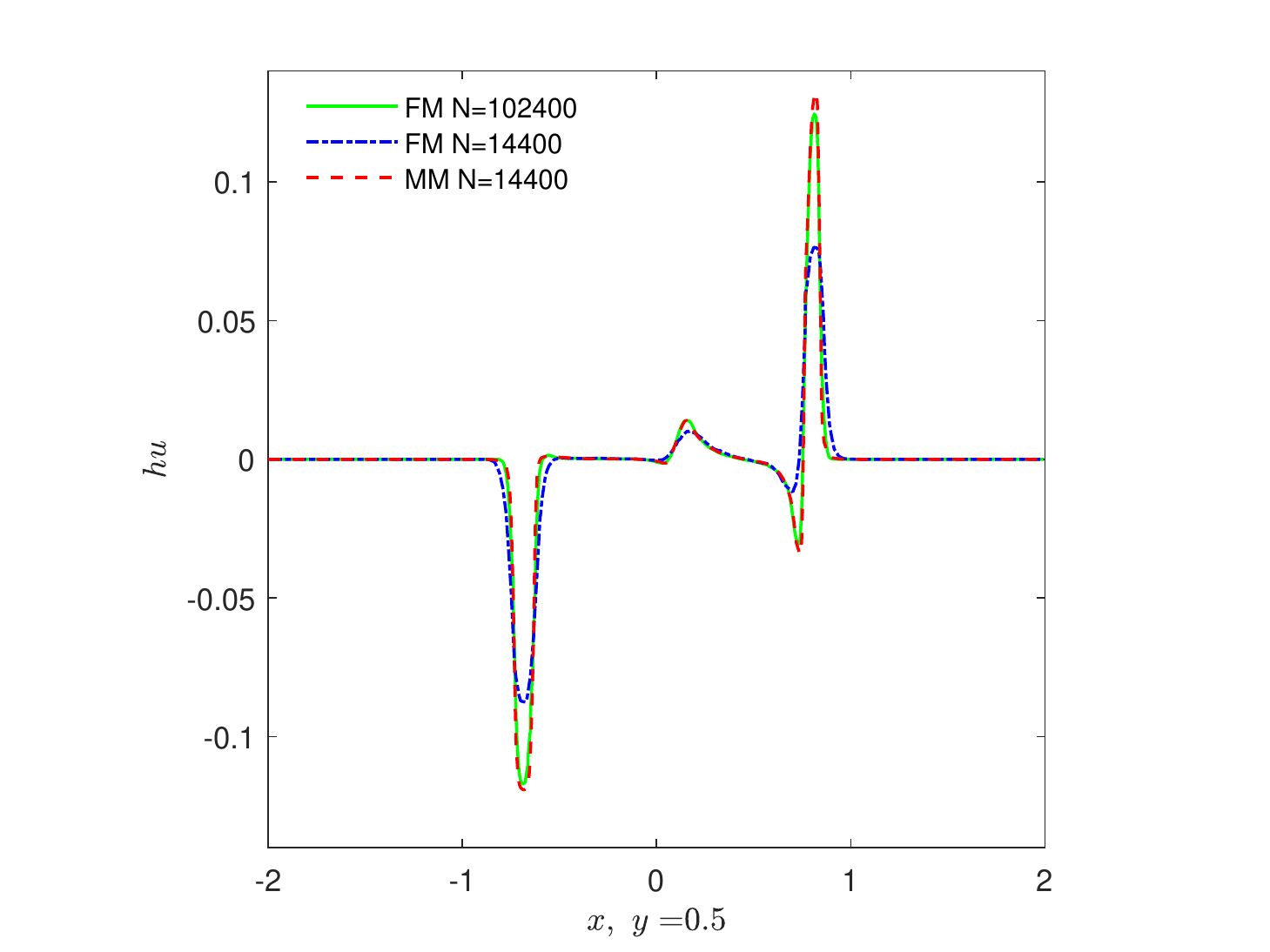}}
\subfigure[close view of (a)]{
\includegraphics[width=0.4\textwidth,trim=10 0 40 10,clip]
{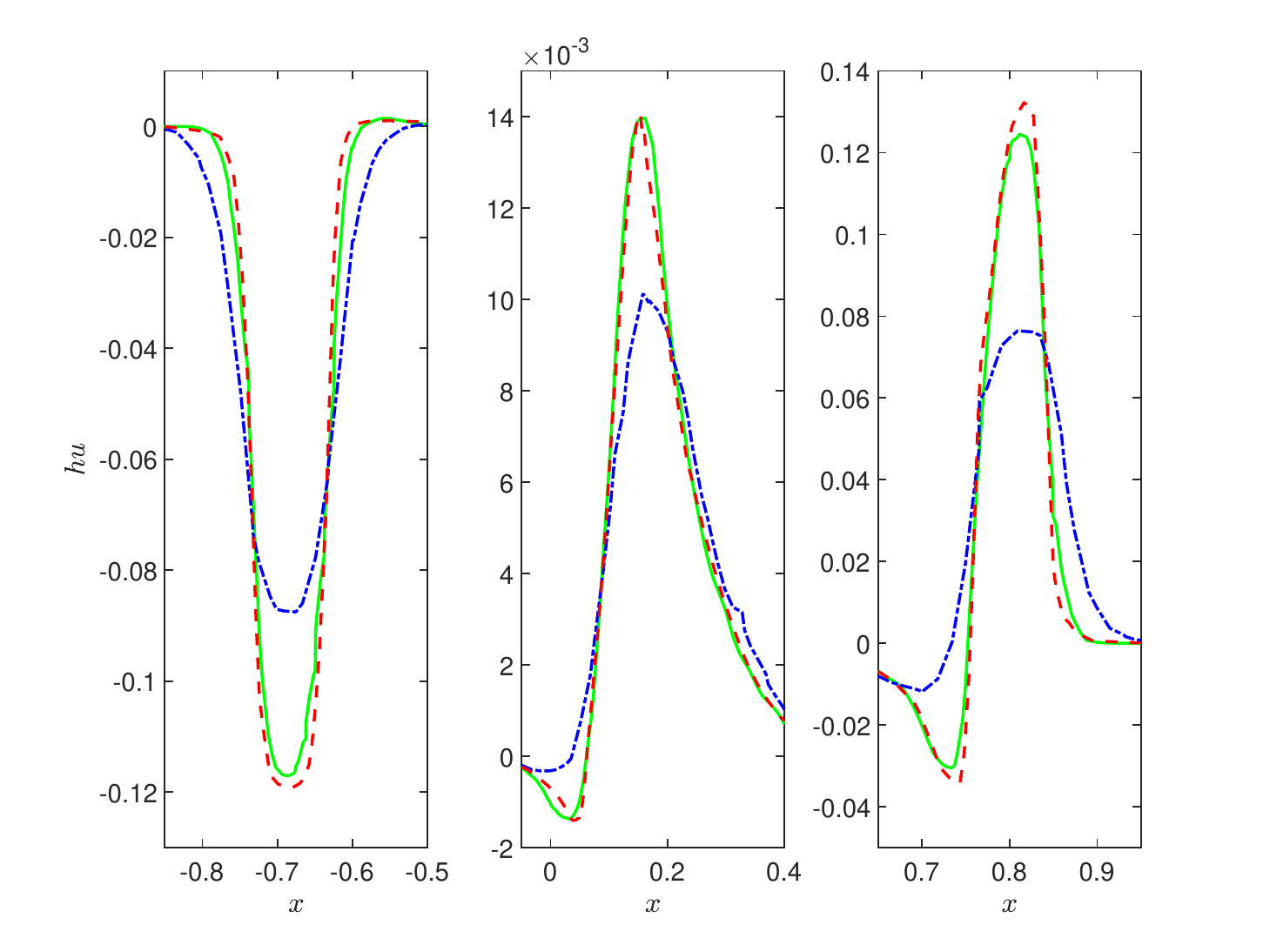}}
\subfigure[$hu$: $t=0.24$]{
\includegraphics[width=0.4\textwidth,trim=10 0 40 10,clip]
{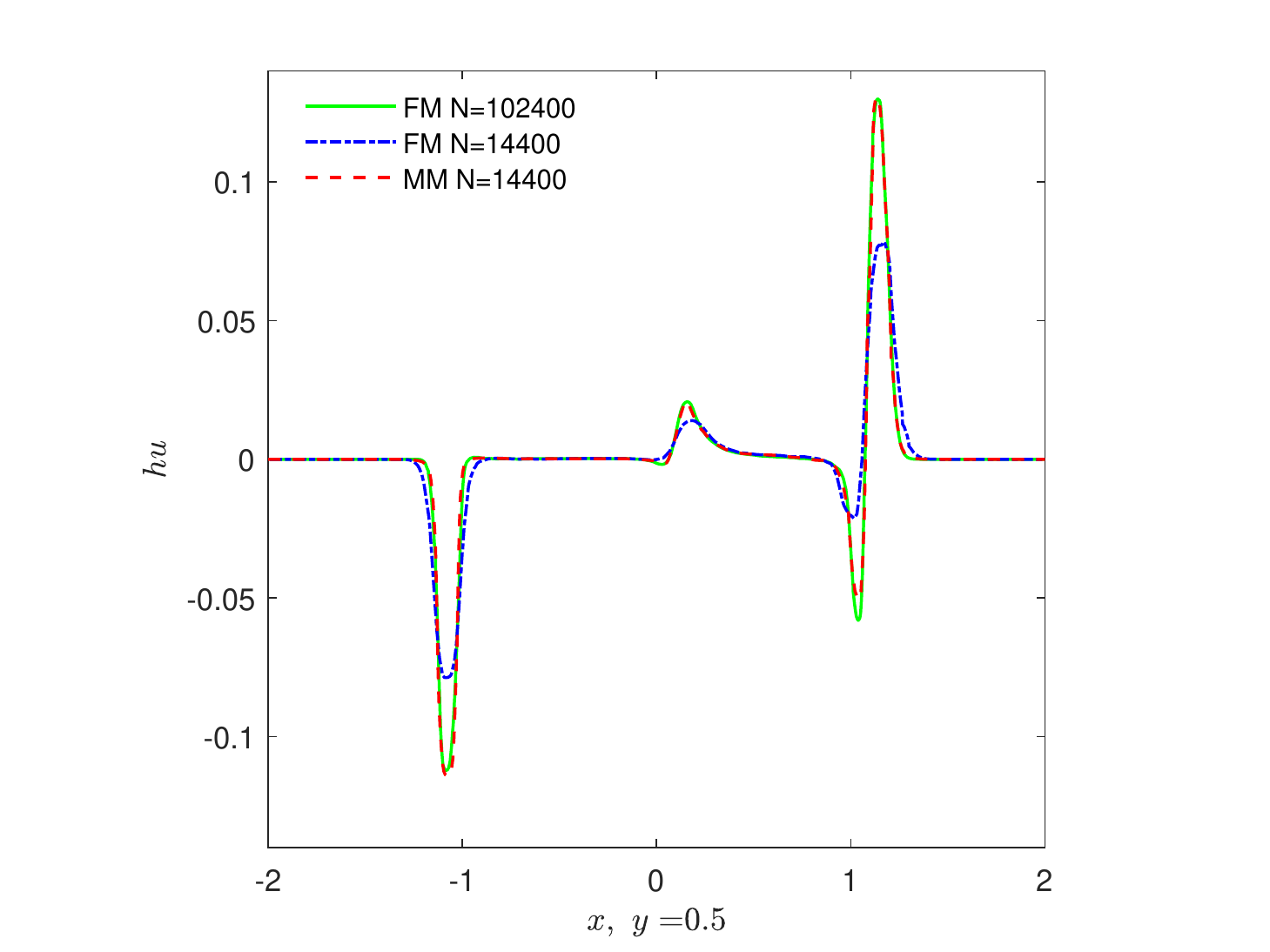}}
\subfigure[close view of (c)]{
\includegraphics[width=0.4\textwidth,trim=10 0 40 10,clip]
{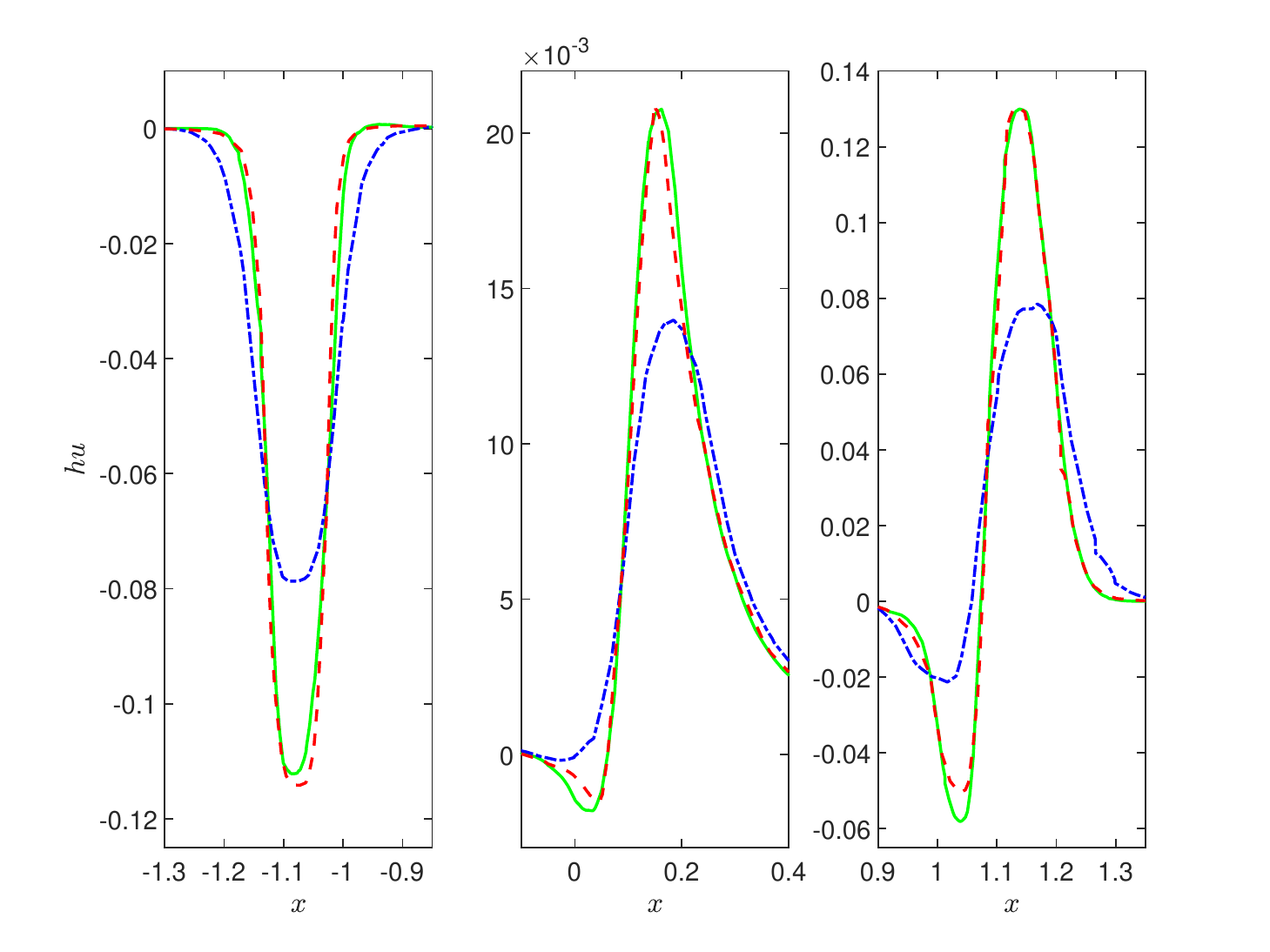}}
\caption{Example \ref{test6-2d}. The cut of $hu$ along the line $y=0.5$ at $t=0.16$ (first row) and $0.24$ (second row) are obtained with $P^2$-DG and a moving mesh of $N=14400$ and fixed meshes of $N=14400$ and $N=102400$.}
\label{Fig:test6-s1-hu-cuty}
\end{figure}

\section{Conclusions}
\label{sec:conclusions}

We have presented a high-order well-balanced and positivity-preserving rezoning-type moving mesh DG method for the Ripa model on triangular meshes. The Ripa model takes into account thermodynamic processes which are important particularly in the upper layers of the ocean where the variations of sea surface temperature  are an important factor in climate change.

The rezoning-type MM-DG method contains three basic components at each time step,
the adaptive mesh movement,
interpolation/remapping of the solution between the old and new meshes,
and numerical solution of the Ripa model on the new mesh.
We have employed the MMPDE scheme in the first component at each
time step. A key of the MMPDE scheme is to define the metric tensor that provides the information needed for controlling the size, shape, and orientation of mesh elements over the whole spatial domain.
We compute the metric tensor based on the equilibrium variable $\mathcal{E}=\frac{1}{2}(u^2+v^2)+g\theta(h+b)$
and the water depth $h$ so that the mesh adapts to the features in the water flow and the bottom topography. To minimize the effect of dimensional difference between $\mathcal{E}$ and $h$, we use  $\ln(\mathcal{E})$ and $\ln(h)$ instead of using
$\mathcal{E}$ and $h$ directly in the computation.

To ensure the well-balance property of the overall MM-DG method for the Ripa model, special attention needs to be paid to the interpolation/remapping of the flow variables and bottom topography, slope limiting, and positivity preserving limiting; cf.~\S\ref{sec:WB-MMDG}.
We have proposed to use a DG-interpolation scheme (cf. \S\ref{sec:DG-interp}) for the purpose.
It has high-order accuracy, conserves the mass, preserves constant solutions, preserves positivity of water depth $h$ and temperature $\eta$ (or $\theta$), and satisfies the linearity. We have also employed a high-order correction
for the approximation of the bottom topography.

The numerical examples in one and two dimensions have been presented to demonstrate the well-balance, positivity preservation, high-order accuracy, and mesh adaptation ability of the MM-DG method.
They have also shown that the method is well suited for the numerical simulation of the lake-at-rest steady-state and its perturbations.
Particularly, the mesh concentration reflects structures in the flow variables and bottom topography and leads to more accurate numerical solutions than a fixed mesh with the same number of elements.

%


\end{document}